\documentclass[12pt]{amsart}
\usepackage
{
	amssymb,
	amsmath,
	caption,
	color,
	comment,
	dashbox,
	enumitem,
	fullpage,
	hyperref,
    longtable,
	mathrsfs,
    multirow,
	stmaryrd,
	tabularx,
	textcomp,
    todonotes,
	verbatim
}
\usepackage[all,pdf]{xy}
\usepackage[noadjust]{cite}
\usepackage[abs]{overpic}
\usepackage[usestackEOL]{stackengine}

\usepackage{tikz}
\usetikzlibrary{calc}
\usetikzlibrary{arrows.meta,calc,decorations.markings,fit,math}

\newcolumntype{M}[1]{>{\centering\arraybackslash}m{#1}}

\newtheorem{thm}{Theorem}[section]
\newtheorem{lem}[thm]{Lemma}
\newtheorem{prop}[thm]{Proposition}
\newtheorem{cor}[thm]{Corollary}
\newtheorem{conj}[thm]{Conjecture}

\theoremstyle{definition}
\newtheorem{defn}[thm]{Definition}
\newtheorem{ex}[thm]{Example}
\newtheorem{ques}[thm]{Question}
\newtheorem{rem}[thm]{Remark}

\newcommand{\bbA}{\mathbb{A}}

\newcommand{\bbF}{\mathbb{F}}

\newcommand{\bbP}{\mathbb{P}}
\newcommand{\bbQ}{\mathbb{Q}}

\newcommand{\bbZ}{\mathbb{Z}}

\newcommand{\calC}{\mathcal{C}}

\newcommand{\calO}{\mathcal{O}}
\newcommand{\calP}{\mathcal{P}}
\newcommand{\calQ}{\mathcal{Q}}

\newcommand{\calT}{\mathcal{T}}

\newcommand{\frakP}{\mathfrak{P}}

\newcommand{\frakp}{\mathfrak{p}}
\newcommand{\frakq}{\mathfrak{q}}

\newcommand{\bse}{{\boldsymbol{e}}}

\newcommand{\Kbar}{\overline{K}}

\renewcommand{\hbar}{\overline{h}}

\newcommand{\QQbar}{\overline{\bbQ}}

\newcommand{\Gammabar}{\overline{\Gamma}}

\newcommand{\Aut}{\operatorname{Aut}}

\newcommand{\Gal}{\operatorname{Gal}}

\newcommand{\ord}{\operatorname{ord}}

\newcommand{\PGL}{\operatorname{PGL}}

\newcommand{\PrePer}{\operatorname{PrePer}}

\renewcommand{\tilde}{\widetilde}
\renewcommand{\phi}{\varphi}

\newcommand{\longto}{\longrightarrow}

\newcommand{\Mod}[1]{\ \left(\operatorname{mod}\ {#1}\right)}

\numberwithin{equation}{section}

\title{Cubic points on dynamical modular curves}

\author{John R. Doyle}
\address{Department of Mathematics, Oklahoma State University,
Stillwater, OK 74078} 
\email{john.r.doyle@okstate.edu}

\author{Alexander Galarraga}
\address{Department of Mathematics, University of Washington, Seattle, WA 98195}
\email{agalar@uw.edu}

\subjclass[2020]{Primary 37P15, 37P35; Secondary 11G30, 14G05}
\keywords{Arithmetic dynamics, preperiodic portrait, dynamical modular curve, cubic points}

\begin{document}
\maketitle

\begin{abstract}
    We consider the family of dynamical modular curves associated to quadratic polynomial maps and determine precisely which of these curves have infinitely many cubic points. We use this to prove a classification statement on preperiodic points for quadratic polynomials over cubic fields, extending previous work of Poonen, Faber, and the first author and Krumm.
\end{abstract}

\section{Introduction}\label{sec:intro}

Let $K$ be a number field, let $\Kbar$ be its algebraic closure, and let $f(x) \in K(x)$ be a rational function. Write $f^0$ for the identity map on $\bbP^1$, and for each $n \ge 1$ set $f^n := f \circ f^{n-1}$. We say that a point $P \in \bbP^1(\Kbar)$ is {\bf periodic} for $f$ if there exists $n \ge 1$ such that $f^n(P) = P$; the minimal such $n$ is the ({\bf exact}) {\bf period} of $P$. More generally, we say that $P$ is {\bf preperiodic} for $f$ if there exist integers $m \ge 0$ and $n \ge 1$ such that $f^{m+n}(P) = f^m(P)$; the minimal such $m$ and $n$ are the {\bf preperiod} and {\bf eventual period} of $P$, respectively. The {\bf preperiodic portrait} for $f$ (over $K$) is the functional graph $G(f,K)$ of the restriction of $f$ to the set of $K$-rational preperiodic points for $f$, that is, the directed graph whose vertices are the $K$-rational preperiodic points for $f$ and whose edges are all ordered pairs of the form $(P, f(P))$. Throughout, we will use the term ``portrait" to refer to a functional graph on a finite set.

Northcott \cite{northcott:1950} showed that if $f$ has degree at least $2$, then $f$ has finitely many $K$-rational preperiodic points; in particular, the preperiodic portrait $G(f,K)$ is a finite graph. Based on an analogy between preperiodic points for rational maps and torsion points on elliptic curves, Morton and Silverman posed the following dynamical analogue of the strong uniform boundedness conjecture for torsion points on elliptic curves (now a theorem of Merel \cite{merel:1996}).

\begin{conj}[{Morton-Silverman \cite[p. 100]{morton/silverman:1994}}]\label{conj:ubc}
Fix integers $d \ge 1$ and $D \ge 2$. There exists a constant $B = B(d,D)$ such that if $K$ is a number field of absolute degree $d$ and $f(x) \in K(x)$ is a rational function of degree $D$, then the number of $K$-rational preperiodic points for $f$ is at most $B$.
\end{conj}

Conjecture~\ref{conj:ubc} is completely open---it is not known for any values of $d$ or $D$---though we mention that Looper \cite{looper:2021, looper} has proven a version of the uniform bounedness conjecture for polynomials assuming a stronger variant of the $abc$ conjecture. To emphasize the difficulty of Conjecture~\ref{conj:ubc}, we note that just the case $D = 4$ would suffice to prove Merel's theorem; see \cite[\textsection 3.3]{silverman:2007}.

\subsection{Preperiodic portraits for quadratic polynomials}

Perhaps the most well-studied family of rational functions in complex and arithmetic dynamics is that of quadratic polynomials. We say that two rational functions $f$ and $g$ are {\bf linearly conjugate} over $K$ if there exists $\gamma \in \PGL_2(K)$ such $g = \gamma^{-1}\circ f\circ\gamma$. This is a natural notion of equivalence in dynamics since conjugation commutes with iteration; in particular, if $f,g \in K(x)$ are linearly conjugate over $K$, then the portraits $G(f,K)$ and $G(g,K)$ are isomorphic. Every quadratic polynomial in $K[x]$ is linearly conjugate over $K$ to a unique polynomial of the form
    \[
    f_c(x) := x^2 + c
    \]
with $c \in K$; see \cite[\textsection 4.2]{silverman:2007}, for example. Thus, the space of quadratic polynomials \emph{up to linear conjugation} is a one-parameter family, which we view as analogous to the fact that elliptic curves up to isomorphism form a one-parameter family (parametrized by the $j$-invariant).

Conjecture~\ref{conj:ubc} asserts that the number of $K$-rational preperiodic points for a quadratic polynomial should be bounded above by a constant depending only on $[K:\bbQ]$. We rephrase this in terms of preperiodic portraits.

\begin{conj}\label{conj:ubc-classification}
Fix an integer $d \ge 1$. There exists a (minimal) finite set $\Gammabar = \Gammabar(d)$ of portraits such that for every number field $K$ of absolute degree $d$ and every quadratic $f(x) \in K[x]$, the preperiodic portrait $G(f,K)$ is isomorphic to a portrait in $\Gammabar$.
\end{conj}

Thus, the problem is to determine for which portraits $\calP$ there exists a degree-$d$ number field $K$ and a quadratic polynomial $f(x) \in K[x]$ for which $G(f,K) \cong \calP$. 
For the $d = 1$ case, Poonen \cite{poonen:1998} has shown that if there are no quadratic polynomials over $\bbQ$ with rational points of period larger than $3$, then $\Gammabar(1)$ consists of exactly the twelve portraits labelled
    \[
    \rm
    \emptyset,\ 2(1),\ 3(1,1),\ 3(2),\ 4(1,1),\ 4(2),\ 5(1,1)a,\ 6(1,1),\ 6(2),\ 6(3),\ 8(2,1,1), \text{ and }  8(3)
    \]
in Appendix~\ref{app:portraits}. It is not known whether there exists a quadratic polynomial over $\bbQ$ with rational points of period larger than $3$; however, it is known that a quadratic polynomial in $\bbQ[x]$ cannot have rational points of period $4$ \cite{morton:1998}, period $5$ \cite{flynn/poonen/schaefer:1997}, or, assuming the Birch and Swinnerton-Dyer conjecture for a certain abelian fourfold, period $6$ \cite{stoll:2008}. Further computational evidence \cite{hutz/ingram:2013} suggests that $3$ the largest possible period arising in this setting.

For the $d = 2$ case, based on a large amount of computational evidence together with techniques similar to those used in \cite{poonen:1998}, it was suggested in \cite[p. 514]{doyle/faber/krumm:2014} and explicitly conjectured in \cite[Conjecture 1.4]{doyle:2018quad} that $\Gammabar(2)$ consists of precisely $46$ portraits, which can be found in \cite[Appendix B]{doyle/faber/krumm:2014}.

\subsection{Portraits realized infinitely often}

Given $d \ge 1$, a complete classification of those preperiodic portraits that can be realized by quadratic polynomials defined over number fields of degree $d$---that is, determining the set $\Gammabar(d)$---seems to be quite difficult. Instead, we consider the following variant: Informally, given $d \ge 1$, which preperiodic portraits can be realized infinitely often by quadratic polynomials over degree-$d$ number fields?

As phrased, this is not quite the right question. Indeed, any preperiodic portrait realized by even one quadratic polynomial over $K$ will be realized infinitely often, since linearly conjugate maps have isomorphic preperiodic portraits. Thus, we wish to know which portraits can be realized by infinitely many \emph{equivalence classes} of quadratic polynomials over degree-$d$ number fields.
For this problem, the analogue of Conjecture~\ref{conj:ubc-classification} is known. The following statement was proven in \cite{doyle/krumm:2024} using the main results of \cite{doyle/poonen:2020}. For $d \ge 1$, we define
    \[
    \bbQ^{(d)} := \{\alpha \in \QQbar : [\bbQ(\alpha) : \bbQ] \le d\}
    \]
to be the set of algebraic numbers of degree at most $d$. 

\begin{thm}[{\cite[Theorem 1.11]{doyle/krumm:2024}}]
Fix an integer $d \ge 1$. There exists a finite set $\Gamma = \Gamma(d)$ of portraits such that, for every portrait $\calP$, the following are equivalent:
    \begin{enumerate}[label=\rm(\alph*)]
        \item There are infinitely many $c \in \bbQ^{(d)}$ such that $G(f_c, K) \cong \calP$ for some degree-$d$ number field $K$ containing $c$.
        \item The portrait $\calP$ is isomorphic to a portrait in $\Gamma$.
    \end{enumerate}
\end{thm}

The sets $\Gamma(1)$ and $\Gamma(2)$ are known. Again referring to the labels of portraits in Appendix~\ref{app:portraits}, Faber showed in \cite{faber:2015} that $\Gamma(1)$ consists of the seven portraits
    \begin{center}
    $\emptyset$,\quad 4(1,1),\quad 4(2),\quad 6(1,1),\quad 6(2),\quad 6(3),\quad and\quad 8(2,1,1),
    \end{center}
and the first author and Krumm showed in \cite{doyle/krumm:2024} that $\Gamma(2)$ consists of the $17$ portraits
    \begin{center}
    $\emptyset$,\quad 4(1,1),\quad 4(2),\quad 6(1,1),\quad 6(2),\quad 6(3),\quad 8(1,1)a/b,\quad 8(2)a/b,\quad 8(2,1,1),\quad 8(3),\quad 8(4),\quad 10(2,1,1)a/b,\quad 10(3,1,1),\quad and\quad 10(3,2).
    \end{center}
In this article, we determine $\Gamma(3)$.

\begin{thm}\label{thm:main_dynamical}
    For an abstract portrait $\calP$, the following are equivalent:
        \begin{enumerate}[label=\rm(\alph*)]
            \item There are infinitely many $c \in \bbQ^{(3)}$ such that $G(f_c, K) \cong \calP$ for some cubic number field $K$ containing $c$.
            \item The portrait $\calP$ is isomorphic to one of the following $19$ portraits from Appendix~\ref{app:portraits}:
            \begin{center}
            \rm
            $\emptyset$,\quad  4(1,1),\quad  4(2),\quad  6(1,1),\quad  6(2),\quad  6(3),\quad  8(1,1)a/b,\quad  8(2)a/b,\quad  8(3),\quad  8(4),\quad  8(2,1,1),\quad  10(1,1)a,\quad  10(2,1,1)a/b,\quad  10(3,1,1),\quad  10(3,2),\quad  12(3,3).
            \end{center}
        \end{enumerate}
    In particular, $\Gamma(3)$ is the set of portraits listed in part \emph{(b)}.
\end{thm}

Each realization of a portrait $\calP$ as the $K$-rational preperiodic portrait of a polynomial $x^2 + c \in K[x]$ yields a $K$-rational point on a certain \emph{dynamical modular curve} $X_1(\calP)$. We defer the definitions to Section~\ref{sec:DMCs}, but we note that these curves are, as the name suggests, dynamical analogues of the classical modular curves which parametrize isomorphism classes of elliptic curves together with marked torsion points of a given order. Much of the work in proving Theorem~\ref{thm:main_dynamical} is in proving the following:

\begin{thm}\label{thm:main_curves}
    For a \emph{generic quadratic portrait} $\calP$ (as defined in Section~\ref{sec:genDMC}), the following are equivalent:
    \begin{enumerate}[label=\rm(\alph*)]
    \item There are infinitely many cubic points on $X_1(\calP)$.
    \item The portrait $\calP$ is isomorphic to one of the following $19$ portraits from Appendix~\ref{app:portraits}:
        \begin{center}
        \rm
        $\emptyset$,\quad  4(1,1),\quad  4(2),\quad  6(1,1),\quad  6(2),\quad  6(3),\quad  8(1,1)a/b,\quad  8(2)a/b,\quad  8(3),\quad  8(4),\quad  8(2,1,1),\quad  10(1,1)a,\quad  10(2,1,1)a/b,\quad  10(3,1,1),\quad  10(3,2),\quad  12(3,3).
        \end{center}
    \end{enumerate}
\end{thm}

Theorem~\ref{thm:main_dynamical} does not follow immediately from Theorem~\ref{thm:main_curves} due to one important subtlety: Roughly speaking, a general $K$-rational point on $X_1(\calP)$ corresponds to a map of the form $f_c(x) = x^2 + c$ for which $G(f_c,K)$ contains a \emph{subgraph} isomorphic to $\calP$. In order to show that, infinitely often, $G(f_c,K)$ is actually \emph{isomorphic} to $\calP$ requires some additional work. This additional work, which was used in \cite{doyle/krumm:2024} but is quite different from the approach from \cite{faber:2015}, involves Hilbert's irreducibility theorem coupled with a local statement about periodic points for maps of good reduction modulo a prime; the relevant arguments are in Section~\ref{sec:realizations}.

\subsection{Rationality questions}
As in \cite{doyle/krumm:2024}, we also offer a slightly more refined version of Theorem~\ref{thm:main_dynamical}, where we consider the degrees of the infinitely many $c \in\bbQ^{(3)}$ appearing in the statement of Theorem~\ref{thm:main_dynamical}.

Consider, for example, the portrait $\calP = \rm 4(1,1)$. Faber showed in \cite{faber:2015} that there are infinitely many $c \in \bbQ$ for which $G(f_c,\bbQ)\cong \calP$. Now fix such a $c \in \bbQ$. Since the set of preperiodic points for $f_c$ has bounded height, for all but finitely many cubic fields $K$ we will also have $G(f_c,K) \cong \calP$. In this case, we are able to conclude that $\rm 4(1,1) \in \Gamma(3)$ simply because $\rm 4(1,1) \in \Gamma(1)$. In an effort to avoid this sort of triviality, we are interested in the following two questions for each $\calP \in \Gamma(3)$:

\begin{itemize}
    \item Are there infinitely many $c \in \bbQ$ for which there is a cubic field $K$ such that $G(f_c,\bbQ) \subsetneq G(f_c,K) \cong \calP$?
    \item Are there infinitely many $c \in \QQbar$ with $[\bbQ(c):\bbQ] = 3$ such that $G(f_c,\bbQ(c)) \cong \calP$?
\end{itemize}

The next two results answer these questions.

\begin{thm}\label{thm:weakly-cubic}
    For an abstract portrait $\calP$, the following are equivalent:
    \begin{enumerate}
    \item There are infinitely many $c \in \bbQ$ such that $G(f_c,\bbQ) \subsetneq G(f_c,K) \cong \calP$ for some cubic number field $K$.
    \item The portrait $\calP$ is isomorphic to one of the following four portraits from Appendix~\ref{app:portraits}:
        \begin{center}
        \rm
        6(3),\quad 10(3,1,1),\quad  10(3,2),\quad  12(3,3).
        \end{center}
    \end{enumerate}
\end{thm}

\begin{thm}\label{thm:strictly-cubic}
    For an abstract portrait $\calP$, the following are equivalent:
    \begin{enumerate}
    \item There are infinitely many $c \in \bbQ^{(3)}\smallsetminus\bbQ$ such that $G(f_c,\bbQ(c)) \cong \calP$.
    \item The portrait $\calP$ is isomorphic to one of the $18$ portraits in $\Gamma(3) \smallsetminus \{{\rm 12(3,3)}\}$.
    \end{enumerate}
\end{thm}
We prove these statements in Section~\ref{sec:FOD} by more carefully analyzing the cubic points on the corresponding dynamical modular curves.

\subsection{Miscellaneous remarks}
Sections~\ref{sec:DMCs-finitely-many} and~\ref{sec:infinitely_many} are largely computational. The many computations required in this paper were carried out in Magma \cite{magma225}; code and output are available as ancillary files with this article's arXiv submission.

We frequently need to refer to specific elliptic curves, especially in the tables that summarize elliptic isogeny factors of certain Jacobian varieties. It will be particularly important to know the ranks (over $\bbQ$) of the Mordell-Weil groups of these curves. When we refer to a specific elliptic curve, we will do so using the labels from Cremona's database \cite{cremona:1997}. For example, the curve we call 32A2 is also labeled 32.a3 in the LMFDB \cite{lmfdb}.

\subsection*{Outline of the paper}

In Section~\ref{sec:DMCs}, we define the dynamical modular curves $X_1(\calP)$ which are the primary focus of this article. We collect in Section~\ref{sec:useful-tools} some general results about algebraic points on algebraic curves, which we then use in Sections~\ref{sec:DMCs-finitely-many} and~\ref{sec:infinitely_many} to determine precisely which dynamical modular curves have infinitely many cubic points, proving Theorem~\ref{thm:main_curves}. Section~\ref{sec:realizations} then contains the proof of Theorem~\ref{thm:main_dynamical}. Finally, we include in Section~\ref{sec:FOD} a discussion of the field of definition of the parameter $c$ corresponding to a point on a curve $X_1(\calP)$ which admits infinitely many cubic points.

\subsection*{Acknowledgments}

The first author's research was partially supported by NSF grant DMS-2302394. The second author's research was partially supported by an NSF Graduate Research Fellowship DGE-2140004.

\section{Dynamical modular curves}\label{sec:DMCs}

In this section, we construct the curves which serve as moduli spaces for quadratic polynomial maps with preperiodic ``level structure."

\subsection{Dynatomic polynomials}\label{sec:dyn-poly}

For a polynomial $f$ defined over a field $K$, the points of period $n$ for $f$ are roots of the polynomial $f^n(x) - x$. However, $f^n(x) - x$ also has among its roots the points of period properly dividing $n$, so there is a factorization
    \[
        f^n(x) - x = \prod_{d\mid n} \Phi_{f,d}(x),
    \]
where the factor $\Phi_{f,d}$ generally has as its roots the points of period $d$ for $f$. By M\"obius inversion, one obtains the formula
    \[
        \Phi_{f,n}(x) = \prod_{d\mid n} (f^d(x) - x)^{\mu(n/d)},
    \]
and we refer to this polynomial as the $n$th {\bf dynatomic polynomial}.

As we are interested in the family $f_c(x) = x^2 + c$, we consider the polynomial in two variables given by
    \[
    \Phi_n(c,x) := \Phi_{f_c,n}(x) \in K[c,x].
    \]
Borrowing notation from the classical modular curves, we let $Y_1(n)$ be the affine plane curve defined by $\Phi_n(c,x) = 0$, and we denote by $X_1(n)$ the normalization of the projective closure of $Y_1(n)$. A general point $(c_0,x_0)$ on $Y_1(n)$ corresponds to a map of the form $f_{c_0}(x) = x^2 + c_0$ together with a marked point $x_0$ of period equal to $n$ for $f_{c_0}$. However, for each $n \ge 2$, there are finitely many exceptional points $(c_0,x_0)$ where the period of $x_0$ for the map $x^2 + c_0$ is a proper divisor of $n$. For example, we have
    \[
    \Phi_2(c,x) = \frac{f_c^2(x) - x}{f_c(x) - x} = x^2 + x + c + 1,
    \]
so one can verify that $(c,x) = (-3/4,-1/2)$ is a point on $X_1(n)$ despite the fact that $-1/2$ is a fixed point for $x^2 - 3/4$.

If $z$ has period $n$ for $f_c$, then $f_c(z)$ also has period $n$ for $f_c$, and from this we obtain an automorphism of $Y_1(n)$ of order $n$ defined by $(c,z) \mapsto (c,f_c(z))$; we write $\sigma_n$ for the extension of this automorphism to $X_1(n)$. We denote the quotients of $X_1(n)$ and $Y_1(n)$ by the automorphism $\sigma_n$ by $X_0(n)$ and $Y_0(n)$, respectively, and we note that---again, with finitely many exceptions for each $n$---points on $X_0(n)$ correspond to maps $f_c$ together with a marked periodic {\it cycle} of length $n$.

For $n \ge 1$, define
    \[
        D_1(n) := \sum_{k\mid n} \mu(n/k) 2^k \quad\text{and}\quad D_0(n) := \frac{D_1(n)}n.
    \]
Then $D_1(n)$ and $D_0(n)$ are the degrees of the maps $c : X_1(n) \to \bbP^1$ and $c : X_0(n) \to \bbP^1$, respectively. These quantities grow exponentially with $n$: it is not difficult to show that
    \begin{equation}\label{eq:D1_bounds}
        2^{n-1} \le D_1(n) \le 2^n, \quad\text{hence}\quad \frac{2^{n-1}}n \le D_0(n) \le \frac{2^n}n,
    \end{equation}
for all $n\ge 1$, with strict inequalities whenever $n \ge 3$.

The points on the curve $X_1(n)$ parametrize maps $f_c$ with a \textit{single} marked periodic point; we now extend this to multiple marked periodic points. Fix a nonincreasing sequence $n_1,n_2,\ldots,n_m$ of positive integers, and assume that each positive integer $n$ appears at most $D_0(n)$ times in the sequence. We denote by $X_1(n_1,n_2,\ldots,n_m)$ the curve whose general points correspond to tuples $(c, \alpha_1,\alpha_2,\ldots,\alpha_m)$ where each $\alpha_i$ has period $n_i$ for the map $f_c$ and no two $\alpha_i$ are in the same orbit under $f_c$.

\begin{ex}\label{ex:(3,3)}
For $n = 3$, we have $D_1(3) = 6$, which implies that a general quadratic polynomial has six points of period $3$ partitioned into two $3$-cycles. A general point on the curve in $\bbA^3$ defined by
    \begin{equation}\label{eq:two3cycles_reducible}
    \Phi_3(c,x) = 0 = \Phi_3(c,y)
    \end{equation}
corresponds to a map $f_c$ together with two points ($x$ and $y$) of period $3$ for $f_c$, but without the additional (Zariski open) restriction that $x$ and $y$ lie in distinct orbits under $f_c$. In fact, the curve defined by \eqref{eq:two3cycles_reducible} has exactly four geometrically irreducible components---a genus-$4$ component birational to $X_1(3,3)$, and three components defined by
    \begin{equation}\label{eq:two3cycles_extra}
        \begin{cases}
            \Phi_3(c,x) &\!\!\!\!= 0\\
            \hfill y &\!\!\!\!= x
        \end{cases}
        ,\quad
        \begin{cases}
            \Phi_3(c,x) &\!\!\!\!= 0\\
            \hfill y &\!\!\!\!= f_c(x)
        \end{cases}
        ,\quad\text{and}\quad
        \begin{cases}
            \Phi_3(c,x) &\!\!\!\!= 0\\
            \hfill y &\!\!\!\!= f_c^2(x)
        \end{cases}
        ,
    \end{equation}
respectively. Note that each curve from \eqref{eq:two3cycles_extra} is isomorphic to $Y_1(3)$.
\end{ex}

\subsection{General dynamical modular curves}\label{sec:genDMC}
We now define our most general notion of dynamical modular curve. We give only an informal description here, and we refer the reader to \cite{doyle:2019} for details.

The structure of the preperiodic points for a map can be illustrated using a functional graph, which we refer to as a ({\it preperiodic}) {\it portrait}. The following definition provides the class of portraits that we will be interested in.

\begin{defn}\label{defn:gen-quad}
    A {\bf generic quadratic portrait} is a directed graph $\calP$ satisfying the following properties:
    \begin{enumerate}
        \item Every vertex of $\calP$ has out-degree $1$ and in-degree $0$ or $2$.
        \item For each $n \ge 1$, the number of $n$-cycles in $\calP$ is at most $D_0(n)$.
        \item If $\calP$ has at least one fixed point, then $\calP$ has exactly two fixed points.
    \end{enumerate}
\end{defn}

\begin{rem}
    The term ``generic quadratic portrait" was adopted in \cite{doyle/krumm:2024}; in the earlier paper \cite{doyle:2019}, the first author had used the term ``strongly admissible".
\end{rem}

Our reason for restricting to such portraits is summarized in the following statement:

\begin{lem}\label{lem:finitely-many-PCF-1/4}
    Fix an integer $d \ge 1$. There are only finitely many $c \in \QQbar$ with $[\bbQ(c) : \bbQ] \le d$ for which there is a number field $K \supseteq \bbQ(c)$ such that $G(f_c,K)$ does not satisfy properties {\rm (a)}, {\rm (b)}, and {\rm (c)} of Definition~\ref{defn:gen-quad}.
\end{lem}

\begin{rem}
    A weaker form of Lemma~\ref{lem:finitely-many-PCF-1/4}, where the number field $K$ is fixed, is stated in \cite[Corollary 2.6]{doyle:2019}. The proof is the same, relying on the fact that the set of exceptional $c \in \QQbar$ is a set of bounded height; see the proof of \cite[Proposition 4.22]{silverman:2007}.
\end{rem}

Given a generic quadratic portrait $\calP$, we denote by $X_1(\calP)$ the smooth projective curve whose general points correspond to tuples $(c, \alpha_1, \ldots, \alpha_k)$, where $\alpha_1, \ldots, \alpha_k$ are preperiodic points for $f_c(x) = x^2 + c$ which, informally, form a portrait isomorphic to $\calP$. We refer to \cite{doyle:2019} for a formal definition of the curve $X_1(\calP)$; here, we give a slightly simpler description that is more in the spirit of \cite{doyle/silverman:2020}.

Let $\calP$ be a generic quadratic portrait, and label its vertices using the integers $1,2,\ldots,N$. For each directed edge $\bse = (i,j)$, define the polynomial
    \[
        F_\bse(c,x_1,\ldots,x_N) := f_c(x_i) - x_j = x_i^2 + c - x_j \in \bbZ[c,x_1,\ldots,x_N].
    \]
Note that, since $\calP$ is a functional graph, there are precisely $N$ edges. Let $U_1(\calP)$ be the subvariety of $\bbA^{N+1}$ cut out by the $N$ equations $F_\bse = 0$ and the $\binom{N}{2}$ Zariski open conditions $x_i \ne x_j$ with $1 \le i < j \le N$. We then take $Y_1(\calP)$ to be the Zariski closure of $U_1(\calP)$ in $\bbA^{N+1}$ and $X_1(\calP)$ to be the smooth projective curve birational to $U_1(\calP)$.

\begin{figure}
    \centering
	\begin{tikzpicture}[scale=1]
	\tikzset{vertex/.style = {}}
	\tikzset{every loop/.style={min distance=10mm,in=45,out=-45,->}}
	\tikzset{edge/.style={decoration={markings,mark=at position 1 with %
	    {\arrow[scale=1.2,>=stealth]{>}}},postaction={decorate}}}
	    
	\node[inner sep=.4mm] (03a) at (0,0) {$1$};
	\node[inner sep=.4mm] (03b) at (-.866,.5) {$2$};
	\node[inner sep=.4mm] (03c) at (-.866,-.5) {$3$};
	\node[inner sep=.4mm] (13a) at (1,0) {$4$};
	\node[inner sep=.4mm] (13b) at (-1.366,1.366) {$5$};
	\node[inner sep=.4mm] (13c) at (-1.366,-1.366) {$6$};

	\node[inner sep=.4mm] (03A) at (3.5,0) {$7$};
	\node[inner sep=.4mm] (03B) at (2.634,.5) {$8$};
	\node[inner sep=.4mm] (03C) at (2.634,-.5) {$9$};
	\node[inner sep=.4mm] (13A) at (4.5,0) {$10$};
	\node[inner sep=.4mm] (13B) at (2.134,1.366) {$11$};
	\node[inner sep=.4mm] (13C) at (2.134,-1.366) {$12$};
	
	\draw[->] (03a) to [bend right=30] (03b);
	\draw[->] (03b) to [bend right=30] (03c);
	\draw[->] (03c) to [bend right=30] (03a);
	\draw[->] (13a) to (03a);
	\draw[->] (13b) to (03b);
	\draw[->] (13c) to (03c);

	\draw[->] (03A) to [bend right=30] (03B);
	\draw[->] (03B) to [bend right=30] (03C);
	\draw[->] (03C) to [bend right=30] (03A);
	\draw[->] (13A) to (03A);
	\draw[->] (13B) to (03B);
	\draw[->] (13C) to (03C);
	\end{tikzpicture}
    \caption{A generic quadratic portrait $\calP$}
    \label{fig:ex-(3,3)}
\end{figure}
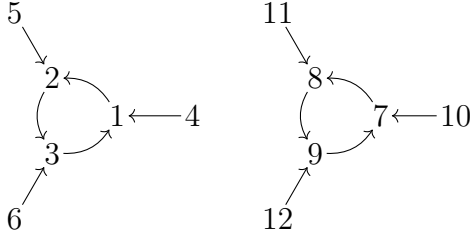

\begin{ex}\label{ex:(3,3)-revisit}
Consider the portrait $\calP$ illustrated in Figure~\ref{fig:ex-(3,3)}. The curve $X_1(\calP)$ is the smooth projective curve birational the subvariety of $\bbA^{13}$ defined by
    \begin{align*}
        f_c(x_1) = f_c(x_5) &= x_2, &
        f_c(x_2) = f_c(x_6) &= x_3,\\
        f_c(x_3) = f_c(x_4) &= x_1, &
        f_c(x_7) = f_c(x_{11}) &= x_8,\\
        f_c(x_8) = f_c(x_{12}) &= x_9, &
        f_c(x_9) = f_c(x_{10}) &= x_7,
    \end{align*}
and $x_i \ne x_j$ for all $1 \le i < j \le 12$.

Alternatively, the approach taken in \cite{doyle:2019} (with the goal of lowering the dimension of the ambient affine space in which the model is defined) would be as follows: First, we note that the portrait $\calP$ in Figure~\ref{fig:ex-(3,3)} is the smallest generic quadratic portrait containing two vertices of period $3$ in different cycles. In the language of \cite{doyle:2019}, the portrait $\calP$ is {\it generated by} any such pair of vertices, say $1$ and $7$. Thus, one obtains a model for $X_1(\calP)$ in $\bbA^3$ by simply taking
\begin{equation}\label{eq:X1(12(3,3))}
    \begin{split}
        \Phi_3(c,x) &= 0\\
        \Phi_3(c,y) &= 0\\
        y &\ne f_c^k(x) \text{ for } k \in \{0,1,2\}.
    \end{split}
\end{equation}
In other words, $X_1(\calP)$ is isomorphic to the curve $X_1(3,3)$; see Example~\ref{ex:(3,3)}.
\end{ex}

\subsection{Morphisms between dynamical modular curves}

The construction $X_1(\cdot)$ is an assignment that takes functional graphs (or, more specifically, generic quadratic portraits) to algebraic curves (over $\bbQ$). The next statement says that this assignment is functorial. For what follows, we use the term {\bf portrait morphism} to mean an {\it injective} graph homomorphism $\calP\hookrightarrow\calP'$ of (generic quadratic) portraits, that is, an injective map from the vertex set of $\calP$ to the vertex set of $\calP'$ that preserves edge relations.

\begin{prop}\label{prop:functoriality}
    Let $\calP$ and $\calP'$ be two generic quadratic portraits, and write $c_\calP$ and $c_{\calP'}$ for the projection-to-$c$ maps from $X_1(\calP)$ and $X_1(\calP')$, respectively, to $\bbP^1$. Suppose there exists a portrait morphism $\psi : \calP \hookrightarrow \calP'$. Then there is a dominant morphism of curves $F_\psi : X_1(\calP') \to X_1(\calP)$, defined over $\bbQ$, satisfying the following properties:
    \begin{enumerate}
        \item The map $c_{\calP'}$ factors through $F_\psi$; more precisely, $c_{\calP'} = c_\calP\circ F_\psi$.
        \item If $\psi$ is not surjective, then $\deg F_\psi \ge 2$.
        \item If $\psi_1$ and $\psi_2$ are distinct portrait morphisms $\calP \hookrightarrow \calP'$, then the corresponding morphisms $F_{\psi_1}$ and $F_{\psi_2}$ are distinct.
    \end{enumerate}
\end{prop}

\begin{proof}
    This essentially follows from \cite[Proposition 3.3]{doyle:2019} and \cite[Theorem 10.1(c)]{doyle/silverman:2020}, but we include the short proof in this case for completeness.

    Label the vertices of $\calP$ and $\calP'$ as $\{1,\ldots,n\}$ and $\{1,\ldots,N\}$, respectively, with $n \le N$. We claim that the projection map
        \begin{equation}\label{eq:proj_map}
            \begin{split}
            \pi : \bbA^{N+1} &\longto \bbA^{n+1}\\
            (c,x_1,\ldots,x_N) &\longmapsto (c,x_{\psi(1)},\ldots,x_{\psi(n)})
            \end{split}
        \end{equation}
    restricts to a map from $U_1(\calP')$ to $U_1(\calP)$ with finite fibers; the desired morphism $F_\psi$ is then the extension of this map to $X_1(\calP')$.

    Let $(c,x_1,\ldots,x_N) \in U_1(\calP')$. Because $\psi$ is a portrait morphism, if $(i,j)$ is a directed edge in $\calP$, then $(\psi(i), \psi(j))$ is a directed edge in $\calP'$. Thus, for every relation $f_c(x_i) = x_j$ defining $U_1(\calP)$, we also have $f_c(x_{\psi(i)}) = x_{\psi(j)}$. Furthermore, the conditions $x_{\psi(i)} \ne x_{\psi(j)}$ are met for all $i\ne j$ by injectivity of $\psi$ and the fact that, on the Zariski open subset $U_1(\calP')$, we have $x_k \ne x_\ell$ for all $k \ne \ell$. Thus, $\pi(c,x_1,\ldots,x_N) \in U_1(\calP)$. Finally, the fibers are finite since, for any given $c$, there are only finitely many preperiodic points for $f_c$ with an orbit of length at most $N$. Thus, our claim is proven.

    Statement (a) is immediate from the construction in \eqref{eq:proj_map}, and (c) is straightforward as well: If $\psi_1 \ne \psi_2$, then there exists $1 \le i \le n$ such that $\psi_1(i) \ne \psi_2(i)$, hence $x_{\psi_1(i)} \ne x_{\psi_2(i)}$. Statement (b) is not difficult, but we refer the reader to the proof of \cite[Proposition 3.3]{doyle:2019}.
\end{proof}

We record two useful consequences of Proposition~\ref{prop:functoriality}.
For what follows, the {\bf cycle structure} of a portrait $\calP$ is the nonincreasing sequence of cycle lengths appearing in $\calP$.

\begin{cor}
If $\calP$ has cycle structure $(n_1,\ldots,n_m)$, then there is a morphism $X_1(\calP) \to X_1(n_1,\ldots,n_m)$ defined over $\bbQ$.
\end{cor}

For a curve $X$ defined over a field $k$, denote by $\Aut_k(X)$ the group of automorphisms of $X$ defined over $k$, and given a dominant morphism $f : X \to Y$ of curves defined over $k$, set
    \[
        \Aut_k(f : X \to Y) := \{\sigma \in \Aut_k(X) : f\circ\sigma = f\}.
    \]
In other words, $\Aut_k(f : X \to Y)$ is isomorphic to the automorphism group of the field extension $k(X)/f^*k(Y)$ induced by $f$.

\begin{prop}\label{prop:auts}
    Let $\calP$ be a generic quadratic portrait. Then
        \[
        \Aut_\bbQ\big(c : X_1(\calP) \to \bbP^1\big) \cong \Aut(\calP).
        \]
\end{prop}

\begin{proof}
    Let $L = \bbQ(X_1(\calP))$ be the function field of $X_1(\calP)$ over $\mathbb{Q}$. The extension $L/\bbQ(c)$ is generated by preperiodic points $x_1,\ldots,x_n$ for $f_c$ that form a portrait isomorphic to $\calP$. Any automorphism of $L/\bbQ(c)$ must commute with the action of $f_c$, hence yields an element of $\Aut(\calP)$. Since the extension is generated by $x_1,\ldots,x_n$, an element of $\Aut(\calP)$ is determined by its action on $x_1,\ldots,x_n$, thus $\Aut_\bbQ(c : X_1(\calP) \to \bbP^1) \cong \Aut(L/\bbQ(c))$ embeds into $\Aut(\calP)$. On the other hand, it follows from Proposition~\ref{prop:functoriality} that every element of $\Aut(\calP)$ yields an automorphism of $c : X_1(\calP) \to \bbP^1$.
\end{proof}

Propositions~\ref{prop:functoriality} and~\ref{prop:auts} give us two (not mutually exclusive) methods for finding morphisms from $X_1(\calP)$ to a curve of lower genus---either let $\calQ$ be a proper subportrait of $\calP$ and consider a map $X_1(\calP) \to X_1(\calQ)$, or take the quotient of $X_1(\calP)$ by a subgroup of $\Aut(\calP)$. These constructions are useful for applying the Castelnuovo-Severi inequality (Theorem~\ref{thm:CS}) as well as for finding elliptic factors of the Jacobian of $X_1(\calP)$. We illustrate with an example.

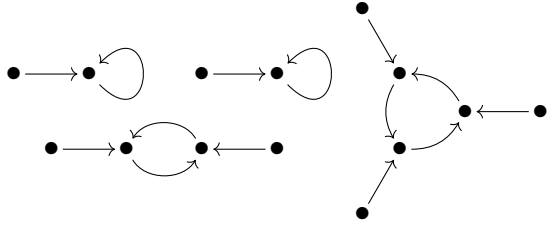
\begin{figure}
    \centering
	\begin{tikzpicture}[scale=1]
	\tikzset{vertex/.style = {}}
	\tikzset{every loop/.style={min distance=10mm,in=45,out=-45,->}}
	\tikzset{edge/.style={decoration={markings,mark=at position 1 with %
	    {\arrow[scale=1.2,>=stealth]{>}}},postaction={decorate}}}
	    
	\node[inner sep=.4mm] (03a) at (0,0) {$\bullet$};
	\node[inner sep=.4mm] (03b) at (-.866,.5) {$\bullet$};
	\node[inner sep=.4mm] (03c) at (-.866,-.5) {$\bullet$};
	\node[inner sep=.4mm] (13a) at (1,0) {$\bullet$};
	\node[inner sep=.4mm] (13b) at (-1.366,1.366) {$\bullet$};
	\node[inner sep=.4mm] (13c) at (-1.366,-1.366) {$\bullet$};
	\node[inner sep=.4mm] (02a) at (-4.5,-.5) {$\bullet$};
	\node[inner sep=.4mm] (02b) at (-3.5,-.5) {$\bullet$};
	\node[inner sep=.4mm] (12a) at (-5.5,-.5) {$\bullet$};
	\node[inner sep=.4mm] (12b) at (-2.5,-.5) {$\bullet$};
	\node[inner sep=.4mm] (01a) at (-5,.5) {$\bullet$};
	\node[inner sep=.4mm] (11a) at (-6,.5) {$\bullet$};
	\node[inner sep=.4mm] (01b) at (-2.5,.5) {$\bullet$};
	\node[inner sep=.4mm] (11b) at (-3.5,.5) {$\bullet$};
	
	\draw[->] (03a) to [bend right=30] (03b);
	\draw[->] (03b) to [bend right=30] (03c);
	\draw[->] (03c) to [bend right=30] (03a);
	\draw[->] (13a) to (03a);
	\draw[->] (13b) to (03b);
	\draw[->] (13c) to (03c);
	\draw[->] (02a) to [bend right=60] (02b);
	\draw[->] (02b) to [bend right=60] (02a);
	\draw[->] (12a) to (02a);
	\draw[->] (12b) to (02b);
	\draw[->] (01a) to [out=-45, in=45, looseness=10] (01a);
	\draw[->] (01b) to [out=-45, in=45, looseness=10] (01b);
	\draw[->] (11a) to (01a);
	\draw[->] (11b) to (01b);
	\end{tikzpicture}
    \caption{A portrait $\calP$}
    \label{fig:(3,2,1)}
\end{figure}

\begin{ex}
Let $\calP$ be the portrait appearing in Figure~\ref{fig:(3,2,1)}. It will be convenient to think of $\calP$ as the smallest generic quadratic portrait containing points of period $1$, $2$, and $3$; in particular, $X_1(\calP)$ is isomorphic to $X_1(3,2,1)$, hence has an affine model (in $\bbA^4$) of the form
    \[
    \Phi_1(c,x) = \Phi_2(c,y) = \Phi_3(c,z) = 0.
    \]
In other words, $X_1(\calP)$ is birational to the fiber product of $X_1(1)$, $X_1(2)$, and $X_1(3)$ relative to the maps $c : X_1(\cdot) \to \bbP^1$.

Now let $\sigma$ be the automorphism that rotates the $3$-cycle component of $\calP$ by a one-third turn. A general point on the quotient curve $X_1(\calP) / \langle \sigma\rangle$ corresponds to a tuple $(c, x, y, \calC)$ consisting of a quadratic polynomial $f_c$, a fixed point $x$ for $f_c$, a point $y$ of period $2$ for $f_c$, and a {\it cycle} $\calC$ of period-$3$ points. In \cite{morton:1992}, Morton shows that the cycle $\calC$ can be parametrized by the sum
    \[
    t = z + f_c(z) + f_c^2(z)
    \]
of the points in the $3$-cycle, and he further shows that the parameter $t$ satisfies
    \[
        t^2 - 2t + 29 + 16c = 0.
    \]
Thus, the curve in $\bbA^4$ defined by
    \[
        \Phi_1(c,x) = \Phi_2(c,y) = t^2 - 2t + 29 + 16c = 0
    \]
is a model for $X_1(\calP) / \langle\sigma\rangle$. One can verify in Magma that this curve is birational over $\bbQ$ to the elliptic curve with Cremona label 32A2, which has Mordell-Weil rank $0$; this is recorded in the ``$3, 2, 1$" row of Table~\ref{table:no repeated elliptic factors}.
\end{ex}

\section{Useful tools for algebraic curves}\label{sec:useful-tools}

We begin by stating some general results towards determining if $X/K$ has infinitely many cubic points.

\subsection{Background}

If a curve $X/K$ has infinitely many points of degree $d \leq 5$, then Kadets and Vogt \cite{kadets/vogt:2025} show that either $X/\overline{K}$ is a degree-$d$ cover of $\mathbb{P}^1$ or an elliptic curve, or $X/K$ is a ``Debarre-Fahlaoui curve." We use their result for $d=3$:

\begin{thm}[{\cite[Theorem 1.2]{kadets/vogt:2025}}]\label{thm:KV}
    Let $X$ be a smooth projective curve of genus greater than $4$ defined over a number field $K$. If $X$ has infinitely many cubic points over $K$, then $X$ admits a $K$-morphism of degree at most $3$ to $\bbP^1$ or to an elliptic curve with positive rank over $K$.
\end{thm}

We can often rule out maps of low degree to curves of low genus via the Castelnuovo-Severi inequality; see \cite[Theorem 3.11.3]{stichtenoth:2009}, for example.

\begin{thm}[Castelnuovo-Severi inequality]\label{thm:CS}
    Let $X$, $Y_1$, and $Y_2$ be curves of genus $g$, $g_1$, and $g_2$, respectively. Suppose that for each $i=1,2$ we have a map $\varphi_i : X \to Y_i$ of degree $d_i$. Then at least one of the following must hold:
        \begin{enumerate}[label=(\alph*)]
            \item the inequality 
                \[
                g \le d_1g_1 + d_2g_2 + (d_1 - 1)(d_2 - 1)
                \]
            is satisfied, or
            \item there is a curve $Y$ and a morphism $\psi : X \to Y$ of degree at least $2$ such that both $\varphi_1$ and $\varphi_2$ factor through $\psi$.
        \end{enumerate}
\end{thm}

Given a curve $X$, one step in determining which elliptic curves are covered by $X$ involves reducing modulo various primes of good reduction. We will rely on the following computationally useful facts about reduction:
\begin{lem}[{\cite[Corollary 10.1.25]{liu:2002}}]\label{lem:liu_reduction}
    Let $X$ be a smooth projective curve defined over a number field $K$, and let $\frakp$ be a maximal ideal in the ring of integers $\calO_K$. Let $X_\frakp$ be the reduction of $X$ modulo $\frakp$. If the normalization of the projective closure of $X_\frakp$ has the same genus as $X$, then $X$ has good reduction at $\frakp$.
\end{lem}

The following is well known; we refer to \cite[Proposition 5]{xarles:2012} for a proof and additional references.
\begin{lem}\label{lem:gonality-point-count}
    Let $X$ be a smooth projective curve defined over a number field $K$, let $\frakp \subset \calO_K$ be a prime of good reduction for $X$, and let $\bbF_q$ be the residue field of $\frakp$. Let $\gamma$ and $\tilde\gamma$ be the gonalities of $X$ and $X_{\bbF_q}$, respectively. Then
        \[
        \gamma \ge \tilde\gamma \ge \frac{|X(\bbF_q)|}{q + 1}.
        \]
\end{lem}

All of our dynamical modular curves are defined over $\bbQ$, so we will be reducing modulo rational primes $p$. Once we have a prime $p$ of good reduction, we can compute the zeta function of the mod-$p$ reduction. Irreducible factors---possibly raised to a nontrivial power---of the numerator of the zeta function correspond to factors (up to isogeny) of the Jacobian \cite[Theorem 1]{tate:1966}. Then, we use the following result of Kani \cite{kani:1994}, but before we state the result, we set some terminology and notation.

Let $X$ be a curve, and let $J(X)$ be its Jacobian. An \textit{elliptic subgroup} of $J(X)$ is an elliptic curve contained in $J(X)$ which is also a subgroup with respect to the addition on $J(X)$. (Note that it is not enough to assume the elliptic curve $E$ is simply a subvariety of $J(X)$, since the group operation on $E$ need not be compatible with the addition law on $J(X)$.) An \emph{elliptic subcover} of $X$ is a morphism from $X$ to an elliptic curve $E$ which does not factor through a nontrivial isogeny of $E$. Note that if $X$ has genus greater than $1$, then any prime-degree morphism $X \to E$ is an elliptic subcover.

\begin{thm}[{\cite[Corollary 4.3]{kani:1994}}]\label{thm: kani elliptic subgroups}
    For a curve $X$ with theta divisor $\Theta_X$ and an integer $d \ge 1$, the map $f \mapsto f^*J(E)$ induces a one-to-one correspondence between isomorphism classes of degree-$d$ elliptic subcovers $f: X \to E$ and elliptic subgroups $E \leq J(X)$ such that $E.\Theta_X = d$. (Here $E.\Theta_X$ denotes the intersection number of $E$ and $\Theta_X$.)
\end{thm}

In the other direction, if $X/K$ has very low genus, then we expect $X/K$ to have infinitely many cubic points. For example, the following is a straightforward consequence of Riemann-Roch:

\begin{lem}[{\cite[Lemma 2.1]{jeon/kim/schweizer:2004}}]\label{lem:JKS_genus2}
    Let $X/K$ be a curve of genus $2$ with at least three $K$-rational points. Then $X$ admits a degree-$3$ map to $\bbP^1$ defined over $K$. In particular, $X$ has infinitely many cubic points over $K$.
\end{lem}

\subsection{General lemmas}\label{sec:general-lemmas}

\subsubsection{Degrees of points with respect to two maps}

We prove a suitable version of \cite[Theorem 4.2]{levin:2016} for later application.

\begin{lem}\label{lem:HIT}
    Let $X/K$ be a curve, let $\varphi$ and $\psi$ be two dominant morphisms to $\bbP^1$, and let $d = \deg\varphi$. Then there is a finite set $S$ of maximal ideals in $\calO_K$ for which the following is true: For all maximal ideals $\frakp \notin S$, and for any finite collection of curves $X_1,\ldots,X_n$ with dominant morphisms $\pi_i : X_i \to X$, there are infinitely many points $P\in X(\Kbar)$ such that
        \begin{enumerate}[label={\textup{(\roman*)}}]
        \item $[K(P) : K] = d$,
        \item the fiber $\pi_i^{-1}(P)$ is irreducible over $K$ for all $i = 1,\ldots,n$, and
        \item for all maximal ideals $\frakq\mid\frakp$ in $\calO_{K(P)}$, we have $\ord_\frakq(\psi(P)) \ge 0$.
        \end{enumerate}
\end{lem}

\begin{proof}
Let $\pi : X \to Y$ be a morphism of maximal degree such that both $\varphi$ and $\psi$ factor through $\pi$, and let $\varphi',\psi' : Y \to \bbP^1$ be such that
    \[
        \varphi = \varphi' \circ \pi \quad\text{and}\quad \psi = \psi' \circ\pi.
    \]
Let $e, d' \ge 1$ be the degrees of $\pi$ and $\varphi'$, respectively; in particular, $d = d'e$.

Let $x$ and $y$ be the elements of the function field of $Y$ corresponding to the maps $\varphi'$ and $\psi'$, respectively. Then, since $\varphi'$ and $\psi'$ do not factor through a common map of degree at least $2$, the curve $Y$ has a plane model defined by the vanishing of a polynomial of the form
    \[
        F(x,y) = a_{d'}(x)y^{d'} + a_{d'-1}(x)y^{d'-1} + \cdots + a_1(x)y + a_0(x) \in K[x][y].
    \]

Choose $x_0 \in \calO_K$ such that $a_{d'}(x_0) \ne 0$, and let $S$ be the set of maximal ideals $\frakP \subset \calO_K$ such that either $\ord_\frakP(a_{d'}(x_0)) \ne 0$ or one of the coefficients of at least one $a_i(x) \in K[x]$ has negative $\frakP$-adic valuation. Note that $S$ is a finite set. Now let $\frakp$ be any maximal ideal not in $S$, and let
    \[
    [x_0]_\frakp := \{t \in K : t\equiv x_0 \Mod \frakp\}
    \]
be the residue class of $x_0$ modulo $\frakp$.

By Hilbert Irreducibility, if $\calT \subset K$ is a thin set (in the sense of Serre), then $[x_0]_\frakp \smallsetminus \calT$ is infinite.
Thus, if we let $\pi_i : X_i \to X$ with $1\le i \le n$ be a collection of maps as in the statement of the lemma, then the set
    \[
    \Sigma := \{t \in [x_0]_\frakp : (\varphi\circ\pi_i)^{-1}(t) \text{ is irreducible for all } i = 1,\ldots,n\}
    \]
is infinite. We claim that every point in the infinite set $\varphi^{-1}(\Sigma) \subset X(\Kbar)$ satisfies the three conditions of the lemma.

Thus, let $P \in \varphi^{-1}(\Sigma)$. For every $i = 1,\ldots,n$ and every preimage $Q_i\in\pi_i^{-1}(P)$, we have
    \[
        [K(P) : K] = \deg \varphi = d \quad\text{and}\quad [K(Q_i) : K(P)] = \deg\pi_i,
    \]
so that conditions (i) and (ii) are satisfied.
Finally, write $t = \varphi(P)$ and $u = \psi(P)$, so that $F(t,u) = 0$. Note that $t \in [x_0]_\frakp$ by assumption, so each $a_i(t)$ is integral at $\frakp$ and $a_{d'}(t)$ is a $\frakp$-adic unit. In particular, this implies that $u$ is integral over each maximal ideal $\frakq \subset \calO_{K(P)} = \calO_{K(u)}$ lying over $\frakp$, so condition (iii) is satisfied as well.
\end{proof}

\subsubsection{Factoring maps between curves}

We would like to thank James Rawson for suggesting the following short lemma, which allows us to study the image of degree-$d$ points arising from a map $X \to Y$ under a second map $X \to Z$. See  \cite{derickx/rawson:2025} for the case $Z = \mathbb{P}^1$.

\begin{lem}\label{lem:map must factor}
    Let $X/K$, $Y/K$, and $Z/K$ be curves with morphisms $\theta: X \to Y$ and $\pi: X \to Z$. Suppose that for infinitely many $y \in Y$, $\pi(\theta^{-1}(y))$ is a single geometric point. Then $\pi$ factors through $\theta$.
\end{lem}
\begin{proof}
    Consider the map $\theta \times \pi: X \to Y \times Z$, let $V$ be the image of $X$ in $Y \times Z$, and let $\gamma: X \to V$ be the map induced by $\theta \times \pi$.
    By assumption, for infinitely many $y \in Y$, there exists $z \in Z$ such that $\pi(\theta^{-1}(y)) = \{z\}$, so that the fiber $\gamma^{-1}((y , z))$ contains $\theta^{-1}(y)$,    
    and thus has degree (as a divisor) at least $\deg \theta$. Thus, there are infinitely many $v \in V$ such that $\gamma^{-1}(v)$ has degree at least $\deg \theta$; in other words, the degree of $\gamma$ is at least the degree of $\theta$. But $\theta$ factors through $\gamma$, so $\theta$ and $\gamma$ have the same degree.
    
    Finally, let $\pi_Y$ and $\pi_Z$ be the projection maps from $Y \times Z$ to $Y$ and $Z$, respectively. Since $\theta = \pi_Y \circ \gamma$, and since $\deg\theta = \deg\gamma$, the restriction $\pi_Y|_V$ is a birational map, hence $\pi = \pi_Z \circ \pi_Y|_V^{-1} \circ\theta$.
\end{proof}

\subsubsection{Curves with finitely many cubic points}

We now provide some sufficient conditions, computable in Magma, for $X/K$ to have finitely many cubic points. Theorem ~\ref{thm:KV} shows that if the genus of $X$ is greater than $4$ and $X$ does not admit a degree-$3$ map to a curve of genus at most $1$, then $X/K$ has finitely many cubic points. Often Theorem~\ref{thm:CS} is enough to rule out degree-$3$ maps to $\mathbb{P}^1$, and thus we are primarily concerned with maps of degree $3$ to an elliptic curve.

If the isogeny decomposition of $J(X)$ does not contain an elliptic curve, then $X$ does not admit a degree-$3$ map to an elliptic curve. (In fact, $X$ does not admit \emph{any} maps to an elliptic curve in this case.) When the isogeny decomposition of $J(X)$ does contain an elliptic curve $E$, Theorem~\ref{thm: kani elliptic subgroups} shows that there will exist a map $X \to E$, and we must determine its degree.

\begin{lem}\label{lem:distinct elliptic factors of jacobian}
    Let $X/K$ be a curve of genus $g > 1$, and fix an integer $d > 1$. Suppose that the isogeny decomposition of $J(X)$ is  $E_1 \times \cdots \times E_n \times A$, where the $E_i$ are pairwise non-isogenous elliptic curves and $A$ does not contain any elliptic curves. Further, suppose that for each $i$ with $1\le i\le n$ there exists an elliptic curve $F_i$ isogenous to $E_i$ and a morphism $\varphi_i: X \to F_i$ of degree coprime to $d$. Then, $X$ does not admit a degree-$d$ morphism to any elliptic curve.
\end{lem}

\begin{proof}
    We first show that the only elliptic subgroups of $J(X)$ are (up to isogeny) the $E_i$. If $E \leq J(X)$ is an elliptic subgroup, then as $A$ does not contain any 1-dimensional subgroups, the image of $E$ under the projection $\pi_A: J(X) \to A$ must be the identity. Thus, $E \leq \pi^{-1}_A(0) = E_1 \times \cdots \times E_n \times 0$. For each $i$ with $1 \le i\le n$, the projections $\pi_i: J(X) \to E_i$ give group homomorphisms $\theta_i: E \to E_i$, which are isogenies if and only if they are non-constant. As the $E_i$ are non-isogenous, we see that at most one---hence \emph{exactly} one---of the maps $\theta_i$ is non-constant. On the other hand, $\theta_i(E)$ is a subgroup of $E_i$, so if $\theta_i$ is constant, then $E$ is contained in the kernel of $\theta_i$. Intersecting all the kernels containing $E$, we see that there must exist a unique $j$ such that $E = E_j$.
    
    By Theorem~\ref{thm: kani elliptic subgroups}, there is a one-to-one correspondence between isomorphism classes of degree-$e$ elliptic subcovers $X \to E$ and elliptic subgroups $E \leq J(X)$ such that $E . \Theta_X = e$. Thus, for each $i = 1,\ldots,n$ there is a unique morphism $\psi_i : X \to E_i$ which does not factor through a nontrivial isogeny of $E_i$; furthermore, any map $X \to E$ for $E$ an elliptic curve must factor through one of the $\psi_i$. This implies that, for each $i = 1,\ldots,n$, the degree of $\psi_i$ is coprime to $d$, as the degree of $\psi_i$ must divide the degree of the map $\phi_i : X \to F_i$. But every map to an elliptic curve would factor through one of the $\psi_i$, so there cannot be a degree-$d$ map to an elliptic curve.
\end{proof}

When the isogeny decomposition of $J(X)$ contains repeated elliptic factors, the situation is more complicated. Indeed, suppose that $J(X)$ contains an abelian subvariety isogenous to $E^2$. By Theorem~\ref{thm: kani elliptic subgroups}, maps to elliptic curves are equivalent to elliptic subgroups of $J(X)$, and $E^2$ contains infinitely many elliptic subgroups: let $m$ and $n$ be integers and consider the map $\theta_{mn}: E \to E \times E$ given by multiplication by $m$ on the first factor and multiplication by $n$ on the second. The image of $\theta_{mn}$ is a subgroup of $E \times E$ isogenous to $E$. Varying $m$ and $n$ gives infinitely many elliptic subgroups of $E^2$.
Instead of attempting to classify the degrees of the infinitely many maps arising from these infinitely many subgroups, we instead note that if $X \to F$ corresponds to a subgroup of $E^2$, then $F$ is isogenous to $E$, and thus if $E$ has rank 0 over $\mathbb{Q}$, so does $F$.

\begin{lem}\label{lem:repeated elliptic factors of jacobian}
    Let $X/K$ be a curve of genus $g > 4$ with finitely many quadratic points. Suppose that $X/K$ does not admit a degree-$3$ morphism to $\mathbb{P}^1$ and that the isogeny decomposition of $J(X)$ is
        \[
        E_1^{q_1} \times \cdots \times E_{m}^{q_m} \times E_{m+1} \times \cdots \times E_{n} \times A,
        \]
    where the $E_i$ are pairwise non-isogenous elliptic curves and $A$ does not contain any elliptic curves. Further, suppose that
    \begin{itemize}
        \item for each $i$ with $1 \leq i \leq m$, $E_i/K$ has rank 0, and
        \item for each $i$ with $m+1 \leq i \leq n$, there exists an elliptic curve $F_i$ isogenous to $E_i$ and a morphism $\phi_i: X \to F_i$ of degree not divisible by 3.
    \end{itemize}
    Then, $X/K$ has finitely many degree-$3$ points.
\end{lem}
\begin{proof}
    By Theorem~\ref{thm:KV}, $X/K$ has infinitely many degree-$3$ points if and only if it is a triple cover of $\mathbb{P}^1$ or an elliptic curve of positive rank over $K$. By hypothesis, $X/K$ does not admit a degree-$3$ map to $\mathbb{P}^1$. Thus, we must consider maps to elliptic curves.
    
    Suppose for the sake of contradiction that there exists a degree-$3$ map $f: X \to E$ for some elliptic curve $E$ of positive rank. As $f$ has degree $3$, we see that $f$ does not factor through an isogeny of degree greater than $1$. Thus, $f$ is an elliptic subcover, and by Theorem~\ref{thm: kani elliptic subgroups}, elliptic subcovers correspond to elliptic subgroups of $J(X)$ via pullback, so that $f^*E \leq J(X)$ is an elliptic subgroup. By the isogeny decomposition of $J(X)$, $f^*E$ must be isogenous to $E_i$ for some $i$. If $1 \leq i \leq m$, then $E_i$ has rank 0, and hence so does $E$, a contradiction. Thus, we can consider only the case where $m+1 \leq i \leq n$.  But the remaining elliptic factors are simple, so an argument as in Lemma~\ref{lem:distinct elliptic factors of jacobian} completes the proof.
\end{proof}

\section{Dynamical modular curves with finitely many cubic points}\label{sec:DMCs-finitely-many}

Our procedure for determining precisely which dynamical modular curves have infinitely many cubic points requires two main steps: First, we show that if $\calP$ has sufficiently many periodic points, then $X_1(\calP)$ has only finitely many cubic points. We do this by considering curves of the form $X_0(n)$, $X_1(n)$, and more generally $X_1(n_1,\ldots,n_m)$, as defined in Section~\ref{sec:dyn-poly}. The result will be that if $X_1(\calP)$ has infinitely many cubic points, then $\calP$ has no cycle of length greater than $4$ and has at most six periodic points.

Once we have restricted our attention to portraits with at most six periodic points, we fix an allowable cycle structure, start with the smallest portrait of that cycle structure, and, gradually increasing the number of vertices, we create larger and larger portraits until all such portraits with a given number of vertices yield curves with finitely many cubic points. Since a portrait morphism $\calP \hookrightarrow \calP'$ yields a morphism $X_1(\calP') \to X_1(\calP)$, once $X_1(\calP)$ has only finitely many cubic points, the same is true for any portrait containing $\calP$. In this way, we are able to determine all curves $X_1(\calP)$ with infinitely many cubic points.

\subsection{Periodic points}

\subsubsection{The quotient curves $X_0(n)$}

The main result of this subsection is the following:

\begin{prop}\label{prop:X0(n)}
    Let $n \ge 1$. Then the curve $X_0(n)$ has infinitely many cubic points if and only if $n \le 6$.
\end{prop}

In order to prove Proposition~\ref{prop:X0(n)}, we require the following lemma.

\begin{lem}\label{lem:CS_X0}
    Let $n \ge 1$, let $g_n$ denote the genus of $X_0(n)$, and assume that $g_n > 2D_0(n) + 1$. Then $X_0(n)$ has finitely many cubic points.
\end{lem}

\begin{proof}
    The inequality $g_n > 2D_0(n) + 1$ fails for $1 \le n \le 11$, so we must have $n \ge 12$. In this case, $D_0(n) \ge 2^{n-1}/n$ by \eqref{eq:D1_bounds}, hence the condition that $g_n > 2D_0(n) + 1$ implies that $g_n > 4$. By Theorem~\ref{thm:KV}, it suffices to show that if $n\ge 12$, then $X_0(n)$ does not admit a map of degree at most $3$ to a curve of genus at most $1$.

    Thus, let $\varphi : X_0(n) \to C$ be a degree-$d$ morphism to a curve of genus $g_C \le 1$.
    Consider the map $c : X_0(n) \to \bbP^1$, which has degree $D_0(n)$. It follows from \cite[\S III, Th\'eor\`eme 3]{bousch:1992} that the Galois group of (the Galois closure of) the morphism $c$ is the full symmetric group $S_{D_0(n)}$, hence the map $c$ does not factor through any intermediate maps. Thus, Castelnuovo-Severi implies that either $d \ge D_0(n)$ (in the case that $\varphi$ factors through $c$) or
        \[
        g_n \le dg_C + (D_0(n) - 1)(d - 1).
        \]
    In the former case, we are done, as $D_0(n) > 3$. In the latter, our assumption that $g_n > 2D_0(n) + 1$ implies that
        \[
            d \ge \frac{g_n + D_0(n) - 1}{g_C + D_0(n) - 1} \ge \frac{g_n - 1}{D_0(n)} + 1 > 3. \qedhere
        \]
\end{proof}

\begin{proof}[Proof of Proposition~\ref{prop:X0(n)}]
    The curves $X_1(n)$ with $n = 1,2,3$ are known to be isomorphic over $\bbQ$ to $\bbP^1$ (see \cite{walde/russo:1994}), so also $X_0(n) \cong \bbP^1$ over $\bbQ$ for those values of $n$. Morton \cite{morton:1998} determined that $X_0(4) \cong_\bbQ \bbP^1$ as well.

    Flynn, Poonen, and Schaefer considered the genus-$2$ curve $X_0(5)$ in \cite{flynn/poonen/schaefer:1997}, where they showed that it has precisely six rational points. Thus, by Lemma~\ref{lem:JKS_genus2}, $X_0(5)$ has infinitely many cubic points.

    In \cite[p. 369]{stoll:2008}, Stoll gives a plane model $\Psi_6(t,c) = 0$ for the genus-$4$ curve, and the polynomial $\Psi_6$ has degree $3$ in $c$. In particular, the map $t : X_0(6)\to\bbP^1$ has degree $3$, so $X_0(6)$ also has infinitely many cubic points.

    Following the strategy from \cite{stoll:2008}, we can explicitly construct models for $X_0(n)$ with $7 \le n \le 11$. We summarize the results as follows; note that the genus of $X_0(n)$ can be computed using an explicit formula provided by \cite[Theorem C]{morton:1996}. We also note that the genus of each of these curves is greater than $4$, and by \cite[Proposition 4.5]{doyle/krumm:2024} each curve has finitely many quadratic points, so it suffices to show that none of these curves admit degree-$3$ maps to a curve of genus at most $1$.
    \begin{itemize}
        \item The curve $X_0(7)$ has genus $16$ and a map of degree $9$ to $\bbP^1$; Castelnuovo-Severi is not enough. Instead, we reduce at $p=3$. Our model is not smooth, but as the geometric genus of the resulting curve is also $16$, Lemma~\ref{lem:liu_reduction} says that $X_0(7)$ has good reduction at $3$. By Lemmas~\ref{lem:gonality-point-count} and~\ref{lem:distinct elliptic factors of jacobian}, it suffices to show that the $\bbF_3$-gonality of $X_0(7)$ is greater than $3$ and that the Jacobian $J(X_0(7))$ has no elliptic isogeny factors.
        
        If $\varphi : X_0(7) \to \bbP^1$ were a map of degree at most $3$ defined over $\bbF_3$, then the pole divisor of $\varphi$ would be defined over $\bbF_3$ and have degree at most $3$. But a Riemann-Roch computation in Magma shows that for any $\bbF_3$-rational degree-$3$ effective divisor $D$ on $X_0(7)$, the Riemann-Roch space of $D$ consists only of constant functions, hence there are no maps of degree at most $3$ to $\bbP^1$ defined over $\bbF_3$. (A slightly longer calculation shows that the $\bbF_3$-gonality of $X_0(7)$ is $6$, so the $\bbQ$-gonality is at least $6$ by Lemma~\ref{lem:gonality-point-count}.)
        
        Finally, the Jacobian over $\mathbb{F}_3$ is simple, as the numerator of the zeta function is irreducible over $\mathbb{Q}$. Therefore, $X_0(7)$ does not admit any maps to elliptic curves.
        \item The curves $X_0(8)$, $X_0(9)$, and $X_0(10)$ have genera $32$, $79$, and $162$, and admit maps of degree $11$, $28$, and $37$, respectively, to $\mathbb{P}^1$. Castelnuovo-Severi (Theorem~\ref{thm:CS}) is therefore sufficient to rule out degree-$3$ maps to curves of genus at most $1$.
        \item By \cite[Theorem D]{doyle/etal:2019} and \cite[Theorem 10.1]{doyle/etal:2019}, both $X_1(11)$ and $X_0(11)$ have good reduction at $3$, and by \cite[Remark 3.12]{doyle/etal:2019}, the curve $X_0(11)_{\bbF_3}$ is the quotient of $X_1(11)_{\bbF_3}$ by the automorphism $(x,c) \mapsto (x^2 + c,c)$; in other words, the quotient of the reduction is the reduction of the quotient.
        By \cite[Proposition 10]{morton:1996}, all points on $X_1(11)_{\bbF_3}$ lying above $c = \infty$ are $\bbF_3$-rational, hence the same is true for $X_0(11)_{\bbF_3}$. The proof of \cite[Theorem 13(a)]{morton:1996} shows that there are precisely $93$ points on $X_0(11)_{\bbF_3}$ lying above $c = \infty$, hence at least $93$ $\bbF_3$-rational points on $X_0(11)$. By Lemma~\ref{lem:gonality-point-count}, the gonality (over $\bbQ$) of $X_0(11)$ is at least $\lceil \frac{93}{4}\rceil = 24$, hence $X_0(11)$ does not admit any maps of degree at most $3$ defined over $\bbF_3$ to either $\bbP^1$ or an elliptic curve.
    \end{itemize}

Now fix $n \ge 12$; we claim that the genus of $X_0(n)$ satisfies $g_n > 2D_0(n) + 1$, so that Lemma~\ref{lem:CS_X0} will yield the desired conclusion. For $12 \le n \le 24$, this can be verified by explicit calculations, so assume $n \ge 25$.

Consider the morphism $c : X_0(n) \to \bbP^1$, which has degree $D_0(n)$. Every affine ramification point of the map $c$ has ramification index $2$, and the number of affine branch points is precisely
    \[
        B(n) := \frac12\cdot \left(D_1(n) - \sum_{\substack{k\mid n\\k < n}} D_1(k)\phi\left(\frac nk\right)\right).
    \]
(See the proof of \cite[Theorem 13]{morton:1996}.) In particular, $B(n)$ is a lower bound for the degree of the ramification divisor of $c : X_0(n) \to \bbP^1$, and this together with the Riemann-Hurwitz formula yields the inequality
    \begin{equation}\label{eq:genus_ineq}
        g_n \ge 1 + \frac{B(n)}2 - D_0(n).
    \end{equation}
Thus, it suffices to show that for $n \ge 25$ we have $B(n) > 6D_0(n)$. By \eqref{eq:D1_bounds}, it suffices to show that
    \begin{equation}\label{eq:B(n)bound}
        B(n) \ge \frac{6 \cdot 2^n}n.
    \end{equation}
For $n \ge 25$, we have
    \begin{align*}
        B(n)
            &= \frac12\cdot \left(D_1(n) - \sum_{\substack{k\mid n\\k < n}} D_1(k)\phi\left(\frac nk\right)\right)\\
            &\ge \frac12\cdot\left(D_1(n) - \sum_{k=1}^{\lfloor n/2\rfloor} D_1(k)\cdot n\right)\\
            &\ge \frac12\cdot\left(2^{n-1} - n\sum_{k=1}^{\lfloor n/2\rfloor} 2^k\right) \qquad\text{(by \eqref{eq:D1_bounds})}\\
            &= 2^{n-2} - n\left(2^{\lfloor n/2\rfloor} - 1\right)\\
            &\ge 2^{n-2} - n\cdot 2^{n/2} + n.
    \end{align*}
Thus, to show that \eqref{eq:B(n)bound} is satisfied, it suffices to show that
    \[ 
        2^{n-2} - n\cdot 2^{n/2} + n \ge \frac{6 \cdot 2^n}n.
    \]
This is equivalent to the statement that
    \[
    2^n \cdot \left(\frac14 - \frac6n - \frac n{2^{n/2}}\right) + n \ge 0.
    \]
As the quantity in parentheses is an increasing function of $n$ for $n \ge 1$ and is positive when $n = 25$, we are done.
\end{proof}

\subsubsection{The curves $X_1(n)$}

\begin{prop}\label{prop:X1(n)}
    Let $n \ge 1$. Then the curve $X_1(n)$ has infinitely many cubic points if and only if $n \le 4$.
\end{prop}

\begin{proof}
As mentioned in the proof of Theorem~\ref{prop:X0(n)}, the curves $X_1(n)$ for $n=1,2,3$ are known to be isomorphic (over $\bbQ$) to $\bbP^1$, so they have infinitely many cubic points.
Morton showed in \cite{morton:1998} that $X_1(4)$ has genus $2$ and six rational points, hence $X_1(4)$ has infinitely many cubic points by Lemma~\ref{lem:JKS_genus2}.

Theorem~\ref{prop:X0(n)} says that $X_0(n)$ has finitely many cubic points for all $n \ge 7$, hence the same is true for $X_1(n)$. It therefore remains to show that $X_1(5)$ and $X_1(6)$ have finitely many cubic points.

The curve $X_1(5)$ has genus $14$ and admits a degree-$7$ map to $\bbP^1$: If we use the affine model for $X_1(5)$ given by the vanishing of the $5$th dynatomic polynomial $\Phi_5(c,x)$, the degree-$7$ map is
    \[
    (c,x) \mapsto \Phi_2(c,x) = x^2 + x + c + 1.
    \]
By Castelnuovo-Severi, $X_1(5)$ has no cubic maps to $\bbP^1$. The curve $X_0(5)$ has good reduction at $p = 3$ by \cite{flynn/poonen/schaefer:1997}, hence the same is true for $X_1(5)$ by \cite[Theorem D]{doyle/etal:2019}. The numerator of the zeta function of $X_1(5)_{\bbF_3}$ has two irreducible factors over $\mathbb{Q}$---one of degree $4$, which corresponds to the $2$-dimensional factor isogenous to $J(X_0(5))$, and one of degree $24$. Thus, $X_1(5)$ does not admit any maps to elliptic curves over $\mathbb{Q}$, as $J(X_1(5))$ does not contain an elliptic curve over $\mathbb{Q}$.

We do something similar for $X_1(6)$: Using the affine model given by $\Phi_6(c,x) = 0$, the map
    \begin{align*}
        X_1(6) &\longto \bbP^1\\
        (c,x) &\longmapsto x^2 - x + c
    \end{align*}
has degree $13$; this can be verified in Magma. Since $X_1(6)$ has genus $34$, Castelnuovo-Severi rules out any maps of degree $3$ to a curve of genus at most $1$. Therefore, $X_1(6)$ has only finitely many cubic points.
\end{proof}

\subsubsection{Curves of the form $X_1(n_1,n_2)$ and $X_1(n_1,n_2,n_3)$} 

From this point forward, there are essentially four different types of arguments we use to show that various dynamical modular curves have only finitely many cubic points, and we list the curves for which each argument is used in Tables~\ref{table:no degree 3 maps to elliptic curves}, \ref{table:no elliptic factors}, \ref{table:no repeated elliptic factors}, and~\ref{table:repeated elliptic factors}. For example, Table~\ref{table:no degree 3 maps to elliptic curves} is a list of dynamical modular curves of large genus which admit low-degree maps to low-genus curves, and hence the Castelnuovo-Severi inequality Theorem~\ref{thm:CS} rules out the existence of degree-$3$ maps to $\bbP^1$ or an elliptic curve. More detailed explanations of Tables~\ref{table:no elliptic factors}, \ref{table:no repeated elliptic factors}, and~\ref{table:repeated elliptic factors} are given below.

\begin{prop}\label{prop:X1(m,4)}
    Let $1 \le n \le 4$. Then $X_1(4,n)$ has finitely many cubic points.
\end{prop}

\begin{proof}
Each of these curves has genus greater than $4$, so Theorem~\ref{thm:KV} applies. The curves $X_1(4,1)$ and $X_1(4,2)$ appear in Table~\ref{table:no repeated elliptic factors}, while $X_1(4,3)$, and $X_1(4,4)$ appear in Table~\ref{table:no degree 3 maps to elliptic curves}.
\end{proof}

\begin{table}
\centering
\begin{tabular}{|l|l|l|l|l|}
 \hline
 Curve $X_1(\cdot)$ & Genus & Degree-$d$ & Genus \\
  & & subcover & of image \\
 \hline \hline
 3, 3, 1 & 16 & $2$: 3, 3 & 4 \\
 \hline
 3, 3, 2 & 16 & $2$: 3, 3 & 4 \\
 \hline
 4, 3 & $49$ & $6$: 8(4) & $2$\\
 \hline
 4, 4 & $\ge 65$ & $8$: 8(4) & $2$\\
 \hline\hline
 14(3,3) & $16$ & $2$: 12(3,3) & $4$\\
 \hline
\end{tabular}
\caption{Curves $X_1(\cdot)$ for which Castelnuovo-Severi directly implies no degree-$3$ maps to curves of genus $g\leq 1$. An entry of the form ``$d: \square$" in the third column indicates a degree-$d$ map to $X_1(\square)$.}
\label{table:no degree 3 maps to elliptic curves}
\end{table}

\begin{rem}
    In Table~\ref{table:no degree 3 maps to elliptic curves}, we record the fact that $X_1(4,4)$ has genus at least $65$. This was verified computationally in Magma by showing that the genus of the base change over $\bbF_7$, where the curve has good reduction, has genus $65$. One can, in fact, show that the genus is equal to $65$ by describing the ramification locus of the map $c : X_1(4,4) \to \bbP^1$ and applying Riemann-Hurwitz, but doing so would take us a bit too far afield.
\end{rem}

\begin{prop}\label{prop:X1(m,n,3)}
    Let $(m,n) \in \{(2,1),\ (3,1),\ (3,2)\}$. Then $X_1(3,m,n)$ has finitely many cubic points.
\end{prop}

\begin{proof}
    The curve $X_1(3,2,1)$ appears in Table~\ref{table:no repeated elliptic factors}, while $X_1(3,3,1)$ and $X_1(3,3,2)$ appear in Table~\ref{table:no degree 3 maps to elliptic curves}.
\end{proof}

\begin{prop}\label{prop:cycle_structures}
    Let $\calP$ be a generic quadratic portrait for which $X_1(\calP)$ has infinitely many cubic points. Then the cycle structure of $\calP$ is one of the following:
        \[
            (1,1),\ (2),\ (3),\ (4),\ (2,1,1),\ (3,1,1),\ (3,2),\ \text{or}\ (3,3).
        \]
\end{prop}

\begin{proof}
Suppose $X_1(\calP)$ has infinitely many cubic points. By Proposition~\ref{prop:X1(n)}, $\calP$ should have no points of period larger than $4$. Note that $\calP$ has at most $D_0(2) = 1$ cycle of length $2$, at most $D_0(3) = 2$ cycles of length $3$, and at most $D_0(4) = 3$ cycles of length $4$; moreover, if $\calP$ has a fixed point, it must have exactly two.

Proposition~\ref{prop:X1(m,4)} says that if $\calP$ has a $4$-cycle, then that is the only cycle appearing in $\calP$, hence the cycle structure must be $(4)$. Moreover, it follows from Proposition~\ref{prop:X1(m,n,3)} that if $\calP$ has a $3$-cycle, then the cycle structure of $\calP$ must be one of $(3,1,1)$, $(3,2)$, or $(3,3)$.
\end{proof}

\subsection{Curves associated to more general portraits}\label{sec:finitely_many}

In this section we consider dynamical modular curves associated to portraits with cycle structure listed in Proposition~\ref{prop:cycle_structures}, and among all such curves, we determine those with infinitely many cubic points. 
As in the previous subsection, our main tool is \cite[Theorem 1.2]{kadets/vogt:2025}, stated above as Theorem~\ref{thm:KV}. We then determine maps to low genus curves using both the Castelnuovo-Severi inequality and computation of the isogeny decomposition of the Jacobian.

\subsubsection{Curves with no maps to elliptic curves}

If Castelnuovo-Severi is insufficient to rule out degree-$3$ maps to genus-$1$ curves, we can turn to the isogeny decomposition of the Jacobian. Given a curve $X$, if the Jacobian of $X$ does not have an elliptic isogeny factor, then $X$ does not admit a map to an elliptic curve.

\begin{lem}\label{lem:12(3,2)a}
    The genus-$9$ curve $X_1(\rm 12(3,2)a)$ has finitely many cubic points.
\end{lem}
\begin{proof}
    Let $X = X_1(\rm 12(3,2)a)$. By the main theorem of \cite{doyle/krumm:2024}, $X$ has finitely many quadratic points. Thus, by Theorem~\ref{thm:KV}, it suffices to show that $X$ is not a triple cover of $\mathbb{P}^1$ or a positive rank elliptic curve.

    We begin by showing that $X$ is not a triple cover of $\mathbb{P}^1$. Observe that $X$ admits a degree-$2$ map $\varphi$ to $Y = X_1({\rm 10(3,2)})$, which has genus $2$. Since $X$ has genus $9$, Theorem~\ref{thm:CS} says that if there exists a map $\theta: X \to \mathbb{P}^1$ of degree $d$, then either $\theta$ factors through $\varphi$ (hence has even degree) or
    \begin{align*}
        9 \leq d \cdot 0 + 2 \cdot 2 + (d-1)(2-1) = d + 3
    \end{align*}
    which shows that $d \geq 6$. Hence, $X$ does not admit a degree-$3$ map to $\mathbb{P}^1$.

    We now show that $X$ does not admit maps to any elliptic curves. Computing in Magma, we find that the reduction $X_{\bbF_3}$ has genus $9$, and thus $X$ has good reduction at $3$ by Lemma~\ref{lem:liu_reduction}. The L-polynomial (i.e., the numerator of the zeta function) of $X_{\bbF_3}$ has three factors over $\mathbb{Q}$: two of degree $4$ and one of degree $10$. By \cite[Theorem 7]{milne/waterhouse:1969} we find that $J(X_{\bbF_3})$ has three isogeny factors---two of dimension $2$ and one of dimension $5$. The isogeny decomposition of $J(X)$ over $\mathbb{Q}$ must be coarser than the isogeny decomposition of $J(X_{\bbF_3})$, and thus cannot contain an elliptic curve. Thus, there are no maps from $X$ to an elliptic curve.
\end{proof}

Table~\ref{table:no elliptic factors} lists several curves $X_1(\calP)$ which, like $X_1({\rm 12(3,2)a})$, satisfy the following conditions:
\begin{itemize}
    \item $X_1(\calP)$ admits a low-degree map to a low-genus curve, hence Castelnuovo-Severi rules out degree-$3$ maps to $\mathbb{P}^1$;
    \item $X_1(\calP)$ has a prime $p$ of good reduction such that the isogeny decomposition of the Jacobian over $\bbF_p$ does not contain an elliptic curve; and
    \item $X_1(\calP)$ has finitely many quadratic points by the main result of \cite{doyle/krumm:2024}.
\end{itemize}
Thus, every curve in Table~\ref{table:no elliptic factors} has finitely many cubic points.

\begin{table}
\centering
\begin{tabular}{|l|l|l|l|}
 \hline
 Curve $X_1(\cdot)$ & Genus & Degree-$d$ &  $p$ \\
  & & subcover & \\
 \hline \hline
 10(3)a & $9$ & $2$: 8(3) & $3$\\
 \hline
 10(3)b & $9$ & $2$: 8(3) & $5$\\
 \hline
 10(4) & $9$ & $2$: 8(4) & $7$\\
 \hline
 12(3,1,1)a & $9$ & $2$: 10(3,1,1) & $5$, $7$\\
 \hline
 12(3,2)a & $9$ & $2$: 10(3,2) & $3$ \\
 \hline
 12(3,2)b & $9$ & $2$: 10(3,2) & $5$ \\
 \hline

\end{tabular}
\caption{Curves $X_1(\mathcal{P})$ for which $J(X_1(\calP))$ has no elliptic isogeny factors, as verified by computing modulo a prime $p$ of good reduction. An entry of the form ``$d: \square$" in the third column indicates a degree-$d$ map to $X_1(\square)$.}
\label{table:no elliptic factors}
\end{table}

For each of the curves in Table~\ref{table:no elliptic factors} except for $X_1(\rm 12(3,1,1)a)$, the argument that the Jacobian has no elliptic isogeny factors is similar to the argument for Lemma~\ref{lem:12(3,2)a}, and we therefore omit the proof. However, the argument for $X_1(\rm 12(3,1,1)a)$ is slightly different, as we require the isogeny decomposition of the Jacobian at two different primes of good reduction. Thus, we prove this case separately.
\begin{lem}
    The Jacobian of the genus-$9$ curve $X_1(\rm 12(3,1,1)a)$ has no elliptic isogeny factors.
\end{lem}

\begin{proof}
    Let $X = X_1(\rm 12(3,1,1)a)$.
    The model
    \[
        \left\{
        \begin{split}
            y^2 &= f(x) := x^6 + 2x^5 + 5x^4 + 10x^3 + 10x^2 + 4x + 1\\
            z^2 &= g(x) := x^6 - 2x^4 + 2x^3 + 5x^2 + 2x + 1
        \end{split}
        \right.
    \]
    for $X$ is given in \cite[Lemma 3.53]{doyle/faber/krumm:2014}.\footnote{While the curve is defined in \cite[Lemma 3.53]{doyle/faber/krumm:2014} to study the portrait 14(3,1,1), the data used to define the curve was only the data of the proper subgraph 12(3,1,1)a}
    The curve $X$ is a double cover of two distinct genus-$2$ curves, namely $y^2 = f(x)$ and $z^2 = g(x)$, as well as the genus-$5$ hyperelliptic curve $w^2 = f(x)g(x)$. (The genus-$2$ curves are birational to $X_1(\rm 10(3,1,1))$ and $X_1(\rm8(3))$, respectively.) Reducing modulo $5$, the genus-$2$ curves have simple Jacobians which are non-isogenous, as the numerators of their zeta functions are not equal. Reducing modulo $7$, the genus-$5$ curve has simple Jacobian. Thus, the factorization of $J(X)$ is completely determined, and $J(X)$ does not contain an elliptic curve.
\end{proof}

\subsubsection{Curves with no repeated elliptic isogeny factors}

We give an example of how Lemma~\ref{lem:distinct elliptic factors of jacobian} can be combined with Castelnuovo-Severi to determine if a curve has finitely many cubic points.

\begin{lem}\label{lem:no repeated factors example}
    The genus-$5$ curve $X_1(\rm 10(1,1)c)$ has finitely many cubic points.
\end{lem}
\begin{proof}
    By the main result of \cite{doyle/krumm:2024}, $X = X_1({\rm 10(1,1)c})$ has only finitely many quadratic points; thus, by Theorem~\ref{thm:KV}, it suffices to show that $X$ does not admit a degree-$3$ map to $\bbP^1$ or to an elliptic curve.
    
    The portrait 10(1,1)c has both 8(1,1)a and 8(1,1)b as subportraits, and thus $X$ has degree-$2$ maps to $X_1(\rm 8(1,1)a)$ and $X_1(\rm 8(1,1)b)$, which are non-isogenous elliptic curves (isomorphic to the curves with Cremona labels 24A4 and 11A3, respectively). Thus, we see that $J(X)$ has at least two distinct elliptic isogeny factors, which we denote $E$ and $F$.

    To apply Lemma~\ref{lem:distinct elliptic factors of jacobian}, we show that $J(X) \sim E \times F \times A$, where $A$ is simple of dimension $3$. Computing in Magma, we find that the reduction $X_{\bbF_{17}}$ has genus $5$, and thus $X$ has good reduction at $17$ by Lemma~\ref{lem:liu_reduction}. The L-polynomial of $X_{\bbF_{17}}$ has three factors over $\mathbb{Q}$: two of degree $2$ and one of degree $6$. By \cite[Theorem 7]{milne/waterhouse:1969} we find that $J(X_{\bbF_{17}})$ has three isogeny factors: two elliptic curves and one abelian variety of dimension $3$. The isogeny decomposition of $J(X)$ over $\mathbb{Q}$ must be coarser than the isogeny decomposition of $J(X_{\bbF_{17}})$ and as we already know that $J(X)$ contains two non-isogenous elliptic curves, we find that $J(X) \sim E \times F \times A$ as desired.
    We can now apply Lemma~\ref{lem:distinct elliptic factors of jacobian} to the two degree-$2$ maps $X \to X_1(8(1,1)\text{a})$ and $X \to X_1(\rm 8(1,1)b)$, showing that $X$ does not admit a degree-$3$ map to any elliptic curve.

    Finally, we show that $X$ does not admit a degree-$3$ map to $\mathbb{P}^1$. The curve $X$ has genus $5$ and a degree-$2$ map $\varphi:X \to E$ to an elliptic curve, so if there exists a map $\theta: X \to \mathbb{P}^1$ of degree $d$, we have that $\theta$ factors through $\varphi$ or that
    \begin{align*}
        5 \leq d \cdot 0 + 2 \cdot 1 + (d-1)(2-1) = d + 1.
    \end{align*}
    In either case, we have $d \geq 4$, so $X$ does not admit a degree-$3$ map to $\mathbb{P}^1$.
\end{proof}

We note that the key ingredients in proving that $X_1(\rm 10(1,1)c)$ has finitely many cubic points were the following:
\begin{itemize}
    \item a map to a lower genus curve to use with Castelnuovo-Severi to rule out degree-$3$ maps to $\mathbb{P}^1$;
    \item a collection of pairwise non-isogenous elliptic curves $E_1,\ldots,E_n$ and maps $\varphi_i : X \to E_i$ of degree not divisible by $3$;
    \item a prime $p$ of good reduction for $X$ such that the isogeny decomposition of $J(X_{\bbF_p})$ is $E_1 \times \ldots \times E_n \times A$, where $A$ does not contain any elliptic curves; and
    \item knowledge that $X$ has finitely many quadratic points (from \cite{doyle/krumm:2024}).
\end{itemize}

By the main theorem of \cite{doyle/krumm:2024}, all the curves $X_1(\mathcal{P})$ for $\mathcal{P}$ in Table~\ref{table:no repeated elliptic factors} have finitely many quadratic points. We list the necessary information in Table~\ref{table:no repeated elliptic factors} to conclude that those curves also have only finitely many cubic points. Note that the elliptic curves listed in the fourth column of Table~\ref{table:no repeated elliptic factors} are described in one of two ways: Let $\calP$ be one of the portraits in the first column. If the fourth column of that row includes ``$d$: $\calP'$" for a portrait $\calP'$, then $\calP'$ is a subgraph of $\calP$, and the elliptic quotient is one of the maps $X_1(\calP)\to X_1(\calP')$ guaranteed by Proposition~\ref{prop:functoriality}. If not, then the target elliptic curve $E$ is listed by its Cremona label, and the map $X_1(\calP)\to E$ is a quotient obtained by an automorphism of $\calP$; see Proposition~\ref{prop:auts} and the discussion that follows the proof.

\begin{table}
\centering
\begin{tabular}{|l|l|l|l|l|l|}
 \hline
 \Shortunderstack{Curve \\ $X_1(\cdot)$} & Genus & \Shortunderstack{Degree-$d$\\subcover} & \Shortunderstack{Degree-$d$\\elliptic\\quotients} & $p$ & \Shortunderstack{Elliptic\\ isogeny\\factors} \\
 \hline \hline
 4, 1 & $9$ & $2$: 8(4) & 4: 43A1 & $7$ & 43A1 ($r=1$)\\
 \hline
 4, 2 & $9$ & $2$: 8(4) & $4$: 27A3 & $7$ & 27A3\\
 \hline
 3, 2, 1 & $9$ & $2$: 10(3,2) & $3$: 32A2 & 7 &32A2\\
 \hline\hline
 10(1,1)b & $5$ & $2$: 8(1,1)a & $2$: 8(1,1)a & $5$ & 24A4\\
 \hline
 10(1,1)c & $5$ & $2$: 8(1,1)a & $2$: 8(1,1)a & $17$ & 24A4\\
 & & & $2$: 8(1,1)b & & 11A3 \\
 \hline
 10(2)a & $5$ & $2$: 8(2)b & $2$: 8(2)b & $3$ & 11A3\\
 \hline
 10(2)b & $5$ & $2$: 8(2)a & $2$: 8(2)a & $3$ & 40A3\\
  & & $2$: 8(2)b & $2$: 8(2)b & & 11A3\\
 \hline
 12(2,1,1)b & $5$ & $2$: 10(2,1,1)b & $2$: 8(1,1)b & $19$ & 11A3\\
  & & & $2$: 10(2,1,1)b & & 15A8\\
 \hline
 12(2,1,1)c & $5$ & $2$: 10(2,1,1)a & $2$: 8(2)b& $13$ & 11A3\\
  & & & $2$: 10(2,1,1)a & & 17A4\\
 \hline
 12(2,1,1)e & $5$ & $2$: 10(2,1,1)b & $2$: 10(2,1,1)a & $19$ & 17A4\\
  & & & $2$: 10(2,1,1)b & & 15A8\\
 \hline
 12(3,1,1)b & $9$ & $2$: 10(3,1,1) & $3$: 14A4 & $11$ & 14A4\\
 \hline

\end{tabular}
\caption{Curves $X_1(\mathcal{P})$ without repeated elliptic isogeny factors of $J(X_1(\mathcal{P}))$, verified by computing modulo a prime $p$ of good reduction. An entry of the form ``$d: \square$" in the third or fourth column indicates a degree-$d$ map to $\square$, if $\square$ is a Cremona label, or $X_1(\square)$, if $\square$ is a portrait.}
\label{table:no repeated elliptic factors}
\end{table}

Table ~\ref{table:no repeated elliptic factors} shows that the curves $X_1(3,2,1)$ and $X_1(12(3,1,1)\text{b})$ admit degree-$3$ maps to elliptic curves---namely those with Cremona labels 32A2 and 14A4, respectively---so the methods of Lemma ~\ref{lem:no repeated factors example} do not apply. Instead, we use the fact that, in each case, the single elliptic isogeny factor has rank $0$. We provide the short proof for $X_1(12(3,1,1)\text{b})$ here; the proof for $X_1(3,2,1)$ is similar.

\begin{lem}
    The genus-$9$ curve $X_1(\rm 12(3,1,1)b)$ has finitely many cubic points.
\end{lem}
\begin{proof}
    Let $X = X_1(\rm 12(3,1,1)b)$. As $X$ has genus $9$, it suffices (by Theorem~\ref{thm:KV}) to show that $X$ is not a triple cover of $\mathbb{P}^1$ or a positive rank elliptic curve. Castelnuovo-Severi rules out triple covers of $\mathbb{P}^1$, and the isogeny decomposition of $J(X_{\bbF_{11}})$ shows $X$ covers a unique elliptic curve; more specifically, $X$ is a triple cover of 14A4. Over $\mathbb{Q}$, however, 14A4 has rank 0, and hence $X/\mathbb{Q}$ has finitely many cubic points.
\end{proof}

\subsubsection{Curves with repeated elliptic isogeny factors}

We now show how Lemma~\ref{lem:repeated elliptic factors of jacobian} can be used to prove that a curve whose Jacobian has repeated elliptic isogeny factors has finitely many cubic points.

\begin{lem}
    The genus-$5$ curve $X_1(\rm 12(2,1,1)a)$ has finitely many cubic points.
\end{lem}
\begin{proof}
    Let $X = X_1(\rm 12(2,1,1)a)$. Note first that Castelnuovo-Severi applied to the degree-$2$ map from $X$ to the genus-$1$ curve $X_1(\rm 10(2,1,1)a)$ rules out degree-$3$ maps to $\mathbb{P}^1$.
    
    By Proposition~\ref{prop:auts}, the automorphism group\footnote{In fact, by \cite[Remark 3.10]{doyle:2019}, the cover $c: X \to \bbP^1$ is Galois and $\Aut({\rm 12(2,1,1)a})$ is its Galois group.} of the morphism $c: X \to \mathbb{P}^1$ is isomorphic to the automorphism group of the portrait $\rm 12(2,1,1)a$. That group has order $16$ and is generated by three elements of order $2$: interchanging the fixed point components, reflecting the period-$2$ component across the vertical axis, and interchanging the leftmost vertices of the period-$2$ component. Denote these automorphisms by $\sigma_1$, $\sigma_2$, and $\sigma_3$, respectively.
    
    The automorphism $\sigma_3$ swaps the leftmost vertices of the period-$2$ component, and $\sigma_2\sigma_3\sigma_2$ swaps the rightmost vertices. The quotient of $X$ by each of these involutions is isomorphic to $X_1({\rm 10(2,1,1)a})$, which is an elliptic curve with Cremona label 17A4. Since these two maps cannot factor through a common map, the curve 17A4 appears (at least) twice in the isogeny factorization of $J(X)$.
    Similarly, the quotient of $X$ by the involution $\sigma_1$ is $X_1({\rm 8(2)a})$, which has Cremona label 40A3.

    Furthermore, a computation in Magma shows that taking the quotient of $X$ by each of the order-$8$ subgroups $\langle \sigma_1\sigma_3,\ \sigma_2\rangle$ and $\langle \sigma_1\sigma_3,\ \sigma_2\sigma_3\rangle$ yields an elliptic curve, namely 136A2 and 170A2, respectively.

    Given all of the quotient maps above, we conclude that $J(X)$ can, up to isogeny, be decomposed as
    \[
    J(X) \sim \rm 17A4^2 \times 40A3 \times 136A2 \times 170A2.
    \]
    Since the curve 17A4 has rank $0$, and since all of the quotient maps described above have degree a power of $2$, we can apply Lemma~\ref{lem:repeated elliptic factors of jacobian} to conclude that $X$ has finitely many cubic points.
\end{proof}

We note that similar arguments apply to both $\rm 12(2,1,1)d$ and $\rm 10(2)c$, so we omit the details in these cases. A summary of the relevant details is given in Table~\ref{table:repeated elliptic factors}.

\begin{table}
\centering
\begin{tabular}{|l|l|l|l|l|l|}
 \hline
 \Shortunderstack{Curve \\ $X_1(\cdot)$} & 
 Genus &
 \Shortunderstack{Degree-$d$\\subcover} & \Shortunderstack{Degree-$d$\\elliptic\\quotients} & $p$ & \Shortunderstack{Elliptic\\ isogeny\\factors} \\
 \hline \hline
  10(2)c & $5$ & $2$: 8(2)b & $2$: 8(2)b \ ($\times 2$)& $5$ & 11A3\ ($\times 2$)\\
  & & & $2$: 176C1 & & 176C1 ($r=1$)\\
 \hline
 12(2,1,1)a & $5$ & $2$: 10(2,1,1)a & $2$: 10(2,1,1)a ($\times 2$) & N/A & 17A4 ($\times 2$)\\
  & & & $2$: 40A3 & & 40A3\\
  & & & $4$: 136A2 & & 136A2 ($r=1$)\\
  & & & $4$: 170A2 & & 170A2 ($r=1$)\\
 \hline
 12(2,1,1)d & $5$ & $2$: 10(2,1,1)b & $2$: 10(2,1,1)b ($\times 2$) & N/A & 15A8 ($\times 2$)\\
  & & & $4$: 24A1 & & 24A1\\
  & & & $4$: 30A1 & & 30A1\\
  & & & $4$: 120B1 & & 120B1\\
 \hline

\end{tabular}
\caption{Curves $X_1(\mathcal{P})$ with repeated elliptic isogeny factors of $J(X_1(\mathcal{P}))$. An entry of the form ``$d: \square$" in the third or fourth column indicates a degree-$d$ map to $\square$, if $\square$ is a Cremona label, or $X_1(\square)$, if $\square$ is a portrait.}
\label{table:repeated elliptic factors}
\end{table}

\section{Dynamical modular curves with infinitely many cubic points}\label{sec:infinitely_many}

In this section, we prove Theorem~\ref{thm:main_curves}. We begin with a brief outline of the proof.

Note that if $\calP \subseteq \calP'$ are two generic quadratic portraits, then there is a morphism $X_1(\calP') \to X_1(\calP)$ defined over $\bbQ$, so if $X_1(\calP)$ has only finitely many cubic points, the same is true of $X_1(\calP')$. Thus, we proceed as follows:

\begin{enumerate}[label=(\arabic*)]
    \item If $\calP$ has a cycle structure not appearing in the statement of Proposition~\ref{prop:cycle_structures}, then $X_1(\calP)$ has finitely many cubic points; thus, we restrict ourselves to the cycle structures in Proposition~\ref{prop:cycle_structures}.
    \item For each allowable cycle structure, we provide (in Figures~\ref{fig:(1,1)}--\ref{fig:(3,3)}) the directed system of small portraits with that cycle structure. An arrow $\calP' \to \calP$ indicates that $\calP \subseteq \calP'$, hence there is a morphism $X_1(\calP') \to X_1(\calP)$. The box surrounding a portrait label is dashed (resp., solid) if the corresponding dynamical modular curve has infinitely many (resp., finitely many) cubic points.
    \item Finally, we use the results of Section~\ref{sec:finitely_many} to determine, for each cycle structure, which portraits have infinitely many cubic points and which have finitely many.
\end{enumerate}

\subsection{Cycle structure $(1,1)$}

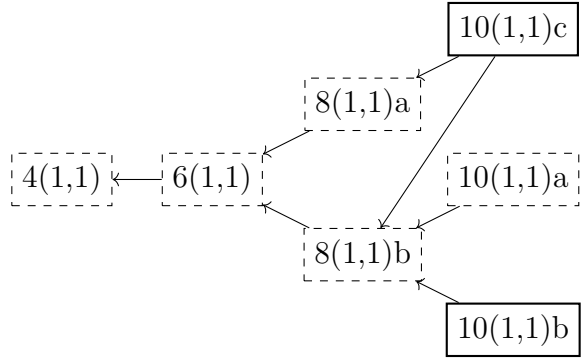
\begin{figure}
    \begin{tikzpicture}[scale=1,
    dashednode/.style={rectangle, draw=black, dashed, minimum size=5mm},
    solidnode/.style={rectangle, draw=black, thick, minimum size=5mm}
    ]
    \tikzset{vertex/.style = {}}
    \tikzset{every loop/.style={min distance=10mm,in=45,out=-45,->}}
    \tikzset{edge/.style={decoration={markings,mark=at position 1 with %
        {\arrow[scale=1.2,>=stealth]{>}}},postaction={decorate}}}
    %
    %
    \node[dashednode] (411) at (0,0) {4(1,1)};
    \node[dashednode] (611) at (2,0) {6(1,1)};
    \node[dashednode] (811a) at (4,1) {8(1,1)a};
    \node[dashednode] (811b) at (4,-1) {8(1,1)b};
    \node[dashednode] (1011a) at (6,0) {10(1,1)a};
    \node[solidnode] (1011c) at (6,2) {10(1,1)c};
    \node[solidnode] (1011b) at (6,-2) {10(1,1)b};
    %
    %
    \draw[->] (611) to (411);
    \draw[->] (811a) to (611);
    \draw[->] (811b) to (611);
    \draw[->] (1011a) to (811b);
    \draw[->] (1011b) to (811b);
    \draw[->] (1011c) to (811a);
    \draw[->] (1011c) to (811b);
    \end{tikzpicture}
    \caption{Portraits with cycle structure $(1,1)$ and at most ten vertices}
    \label{fig:(1,1)}
\end{figure}

The dynamical modular curves associated to 4(1,1) and 6(1,1) are isomorphic to $\bbP^1$, and the curves associated to 8(1,1)a and 8(1,1)b are elliptic curves; see \cite{poonen:1998}. In particular, each of these four curves has infinitely many cubic points. Furthermore, the curve associated to 10(1,1)a has genus $4$ but admits a degree-$3$ map to $\bbP^1$ defined over $\bbQ$, as verified by the function {\tt Genus4GonalMap} in Magma. (Note that the map $c : X_1(\rm10(1,1)a)\to \bbP^1$ has degree $16$, so the degree-$3$ map does not have a natural dynamical interpretation.)

The curves associated to 10(1,1)b/c were shown to have finitely many cubic points in Table~\ref{table:no repeated elliptic factors}. Finally, any portrait $\calP$ with cycle structure $(1,1)$ and at least twelve vertices must contain one of 10(1,1)b and 10(1,1)c, hence $X_1(\calP)$ has only finitely many cubic points.

\subsection{Cycle structure $(2)$}

\begin{figure}
    \begin{tikzpicture}[scale=1,
    dashednode/.style={rectangle, draw=black, dashed, minimum size=5mm},
    solidnode/.style={rectangle, draw=black, thick, minimum size=5mm}
    ]
    \tikzset{vertex/.style = {}}
    \tikzset{every loop/.style={min distance=10mm,in=45,out=-45,->}}
    \tikzset{edge/.style={decoration={markings,mark=at position 1 with %
        {\arrow[scale=1.2,>=stealth]{>}}},postaction={decorate}}}
    %
    %
    \node[dashednode] (42) at (0,0) {4(2)};
    \node[dashednode] (62) at (2,0) {6(2)};
    \node[dashednode] (82a) at (4,1) {8(2)a};
    \node[dashednode] (82b) at (4,-1) {8(2)b};
    \node[solidnode] (102c) at (6,0) {10(2)c};
    \node[solidnode] (102b) at (6,2) {10(2)b};
    \node[solidnode] (102a) at (6,-2) {10(2)a};
    %
    %
    \draw[->] (62) to (42);
    \draw[->] (82a) to (62);
    \draw[->] (82b) to (62);
    \draw[->] (102a) to (82b);
    \draw[->] (102b) to (82a);
    \draw[->] (102b) to (82b);
    \draw[->] (102c) to (82b);
    \end{tikzpicture}
    \caption{Portraits with cycle structure $(2)$ and at most ten vertices}
    \label{fig:(2)}
\end{figure}
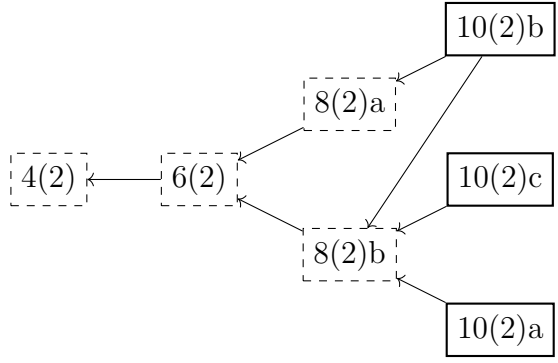

The dynamical modular curves associated to 4(2) and 6(2) are isomorphic to $\bbP^1$, and the curves associated to 8(2)a and 8(2)b are elliptic curves; see \cite{poonen:1998}. In particular, each of these four curves has infinitely many cubic points.

The curves associated to 10(2)a/b were shown to have finitely many cubic points in Table~\ref{table:no repeated elliptic factors}, while the curve associated to 10(2)c appears in Table~\ref{table:repeated elliptic factors}. Finally, any portrait $\calP$ with cycle structure $(2)$ and at least twelve vertices must contain one of 10(2)a/b/c, hence $X_1(\calP)$ has only finitely many cubic points.

\subsection{Cycle structure $(3)$}

\begin{figure}
    \begin{tikzpicture}[scale=1,
    dashednode/.style={rectangle, draw=black, dashed, minimum size=5mm},
    solidnode/.style={rectangle, draw=black, thick, minimum size=5mm}
    ]
    \tikzset{vertex/.style = {}}
    \tikzset{every loop/.style={min distance=10mm,in=45,out=-45,->}}
    \tikzset{edge/.style={decoration={markings,mark=at position 1 with %
        {\arrow[scale=1.2,>=stealth]{>}}},postaction={decorate}}}
    %
    %
    \node[dashednode] (63) at (0,0) {6(3)};
    \node[dashednode] (83) at (2,0) {8(3)};
    \node[solidnode] (103a) at (4,1) {10(3)a};
    \node[solidnode] (103b) at (4,-1) {10(3)b};
    %
    %
    \draw[->] (83) to (63);
    \draw[->] (103a) to (83);
    \draw[->] (103b) to (83);
    \end{tikzpicture}
    \caption{Portraits with cycle structure $(3)$ and at most ten vertices}
    \label{fig:(3)}
\end{figure}
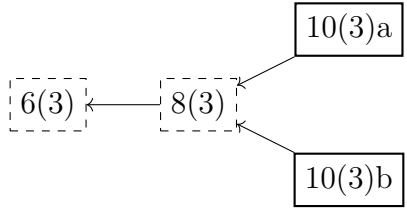

The dynamical modular curve associated to 6(3) is isomorphic to $\bbP^1$, and the curve associated to 8(3) is a curve of genus $2$ with eight rational points; see \cite{poonen:1998}. In particular, these two curves have infinitely many cubic points. (For the curve 8(3), this follows from Lemma~\ref{lem:JKS_genus2}.)

The curves associated to 10(3)a/b were shown to have finitely many cubic points in Table~\ref{table:no elliptic factors}. Finally, any portrait $\calP$ with cycle structure $(3)$ and at least twelve vertices must contain one of 10(3)a/b, hence $X_1(\calP)$ has only finitely many cubic points.

\subsection{Cycle structure $(4)$}

\begin{figure}
    \begin{tikzpicture}[scale=1,
    dashednode/.style={rectangle, draw=black, dashed, minimum size=5mm},
    solidnode/.style={rectangle, draw=black, thick, minimum size=5mm}
    ]
    \tikzset{vertex/.style = {}}
    \tikzset{every loop/.style={min distance=10mm,in=45,out=-45,->}}
    \tikzset{edge/.style={decoration={markings,mark=at position 1 with %
        {\arrow[scale=1.2,>=stealth]{>}}},postaction={decorate}}}
    %
    %
    \node[dashednode] (84) at (0,0) {8(4)};
    \node[solidnode] (104) at (2,0) {10(4)};
    %
    %
    \draw[->] (104) to (84);
    \end{tikzpicture}
    \caption{Portraits with cycle structure $(4)$ and at most ten vertices}
    \label{fig:(4)}
\end{figure}
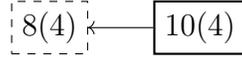

The dynamical modular curve associated to 8(4) is a genus-$2$ curve with six rational points (see \cite{morton:1998}), hence has infinitely many cubic points by Lemma~\ref{lem:JKS_genus2}.

The curve associated to 10(4) was shown to have finitely many cubic points in Table~\ref{table:no elliptic factors}. Finally, any portrait $\calP$ with cycle structure $(4)$ and at least twelve vertices must contain 10(4), hence $X_1(\calP)$ has only finitely many cubic points.

\subsection{Cycle structure $(2,1,1)$}

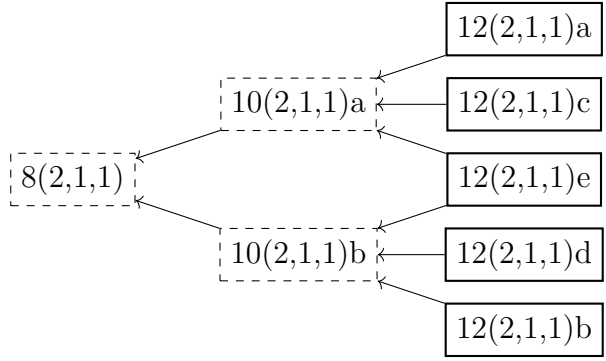
\begin{figure}
    \begin{tikzpicture}[scale=1,
    dashednode/.style={rectangle, draw=black, dashed, minimum size=5mm},
    solidnode/.style={rectangle, draw=black, thick, minimum size=5mm}
    ]
    \tikzset{vertex/.style = {}}
    \tikzset{every loop/.style={min distance=10mm,in=45,out=-45,->}}
    \tikzset{edge/.style={decoration={markings,mark=at position 1 with %
        {\arrow[scale=1.2,>=stealth]{>}}},postaction={decorate}}}
    %
    %
    \node[dashednode] (8211) at (0,0) {8(2,1,1)};
    \node[dashednode] (10211a) at (3,1) {10(2,1,1)a};
    \node[dashednode] (10211b) at (3,-1) {10(2,1,1)b};
    \node[solidnode] (12211a) at (6,2) {12(2,1,1)a};
    \node[solidnode] (12211c) at (6,1) {12(2,1,1)c};
    \node[solidnode] (12211e) at (6,0) {12(2,1,1)e};
    \node[solidnode] (12211d) at (6,-1) {12(2,1,1)d};
    \node[solidnode] (12211b) at (6,-2) {12(2,1,1)b};
    %
    %
    \draw[->] (10211a) to (8211);
    \draw[->] (10211b) to (8211);
    \draw[->] (12211a) to (10211a);
    \draw[->] (12211c) to (10211a);
    \draw[->] (12211e) to (10211a);
    \draw[->] (12211e) to (10211b);
    \draw[->] (12211d) to (10211b);
    \draw[->] (12211b) to (10211b);
    \end{tikzpicture}
    \caption{Portraits with cycle structure $(2,1,1)$ and at most twelve vertices}
    \label{fig:(2,1,1)}
\end{figure}

The dynamical modular curve associated to 8(2,1,1) is isomorphic to $\bbP^1$, and the curves associated to 10(2,1,1)a/b are elliptic curves; see \cite{poonen:1998}. In particular, these three curves have infinitely many cubic points.

The curves associated to 12(2,1,1)b/c/e were shown to have finitely many cubic points in Table~\ref{table:no repeated elliptic factors}, while the curves associated to 12(2,1,1)a/d appear in Table~\ref{table:repeated elliptic factors}. Finally, any portrait $\calP$ with cycle structure $(2,1,1)$ and at least fourteen vertices must contain one of 12(2,1,1)a/b/c/d/e, hence $X_1(\calP)$ has only finitely many cubic points.

\subsection{Cycle structure $(3,1,1)$}

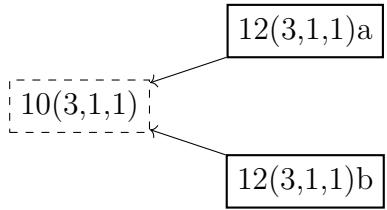
\begin{figure}
    \begin{tikzpicture}[scale=1,
    dashednode/.style={rectangle, draw=black, dashed, minimum size=5mm},
    solidnode/.style={rectangle, draw=black, thick, minimum size=5mm}
    ]
    \tikzset{vertex/.style = {}}
    \tikzset{every loop/.style={min distance=10mm,in=45,out=-45,->}}
    \tikzset{edge/.style={decoration={markings,mark=at position 1 with %
        {\arrow[scale=1.2,>=stealth]{>}}},postaction={decorate}}}
    %
    %
    \node[dashednode] (10311) at (0,0) {10(3,1,1)};
    \node[solidnode] (12311a) at (3,1) {12(3,1,1)a};
    \node[solidnode] (12311b) at (3,-1) {12(3,1,1)b};
    %
    %
    \draw[->] (12311a) to (10311);
    \draw[->] (12311b) to (10311);
    \end{tikzpicture}
    \caption{Portraits with cycle structure $(3,1,1)$ and at most twelve vertices}
    \label{fig:(3,1,1)}
\end{figure}

The dynamical modular curve associated to 10(3,1,1) is a genus-$2$ curve with six rational points (see \cite{poonen:1998}),  and thus has infinitely many cubic points by Lemma~\ref{lem:JKS_genus2}.

The curve associated to 12(3,1,1)a was shown to have finitely many cubic points in Table~\ref{table:no elliptic factors}, while the curve associated to 12(3,1,1)b appears in Table~\ref{table:no repeated elliptic factors}. Finally, any portrait $\calP$ with cycle structure $(3,1,1)$ and at least fourteen vertices must contain one of 12(3,1,1)a/b, hence $X_1(\calP)$ has only finitely many cubic points.

\subsection{Cycle structure $(3,2)$}

\begin{figure}
    \begin{tikzpicture}[scale=1,
    dashednode/.style={rectangle, draw=black, dashed, minimum size=5mm},
    solidnode/.style={rectangle, draw=black, thick, minimum size=5mm}
    ]
    \tikzset{vertex/.style = {}}
    \tikzset{every loop/.style={min distance=10mm,in=45,out=-45,->}}
    \tikzset{edge/.style={decoration={markings,mark=at position 1 with %
        {\arrow[scale=1.2,>=stealth]{>}}},postaction={decorate}}}
    %
    %
    \node[dashednode] (1032) at (0,0) {10(3,2)};
    \node[solidnode] (1232a) at (3,1) {12(3,2)a};
    \node[solidnode] (1232b) at (3,-1) {12(3,2)b};
    %
    %
    \draw[->] (1232a) to (1032);
    \draw[->] (1232b) to (1032);
    \end{tikzpicture}
    \caption{Portraits with cycle structure $(3,2)$ and at most twelve vertices}
    \label{fig:(3,2)}
\end{figure}
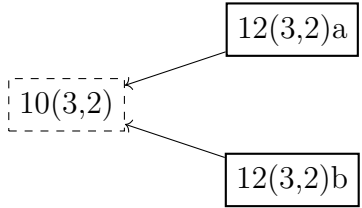

The dynamical modular curve associated to 10(3,2) is a genus-$2$ curve with six rational points, and thus has infinitely many cubic points by Lemma~\ref{lem:JKS_genus2}.

The curves associated to 12(3,2)a/b were shown to have finitely many cubic points in Table~\ref{table:no elliptic factors}. Finally, any portrait $\calP$ with cycle structure $(3,2)$ and at least fourteen vertices must contain one of 12(3,2)a/b, hence $X_1(\calP)$ has only finitely many cubic points.

\subsection{Cycle structure $(3,3)$}\label{sec:(3,3)}

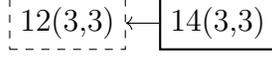
\begin{figure}
    \begin{tikzpicture}[scale=1,
    dashednode/.style={rectangle, draw=black, dashed, minimum size=5mm},
    solidnode/.style={rectangle, draw=black, thick, minimum size=5mm}
    ]
    \tikzset{vertex/.style = {}}
    \tikzset{every loop/.style={min distance=10mm,in=45,out=-45,->}}
    \tikzset{edge/.style={decoration={markings,mark=at position 1 with %
        {\arrow[scale=1.2,>=stealth]{>}}},postaction={decorate}}}
    %
    %
    \node[dashednode] (1233) at (0,0) {12(3,3)};
    \node[solidnode] (1433) at (2,0) {14(3,3)};
    %
    %
    \draw[->] (1433) to (1233);
    \end{tikzpicture}
    \caption{Portraits with cycle structure $(3,3)$ and at most fourteen vertices}
    \label{fig:(3,3)}
\end{figure}

The dynamical modular curve associated to 12(3,3) is a triple cover of $X_1(3) \cong \bbP^1$, and thus has infinitely many cubic points.

The curve associated to 14(3,3) was shown to have finitely many cubic points in Table~\ref{table:no degree 3 maps to elliptic curves}. Finally, any portrait $\calP$ with cycle structure $(3,3)$ and at least sixteen vertices must contain 14(3,3), hence $X_1(\calP)$ has only finitely many cubic points.

\begin{proof}[Proof of Theorem~\ref{thm:main_curves}]
    The theorem follows by combining the results of Sections~\ref{sec:DMCs-finitely-many} and~\ref{sec:infinitely_many}.
\end{proof}

\section{Portraits realized infinitely often over cubic fields}\label{sec:realizations}

In this section, we prove Theorem~\ref{thm:main_dynamical}. By Lemma~\ref{lem:finitely-many-PCF-1/4} and Theorem~\ref{thm:main_curves}, it suffices to prove the following:

\begin{prop}\label{prop:realizing-portraits}
Let $\calP$ be a generic quadratic portrait for which $X_1(\calP)$ has infinitely many cubic points. Then there are infinitely many $c \in \bbQ^{(3)}$ such that $G(f_c, K) \cong \calP$ for some cubic number field $K$ containing $c$.
\end{prop}

The proof of the proposition will rely on the Hilbert irreducibility theorem (via Lemma~\ref{lem:HIT}) as well as the following statement about periodic points for maps with good reduction, for which we refer to \cite[Theorem 2.21]{silverman:2007}.

\begin{prop}\label{prop:good-red}
Let $K$ be a number field, and let $\frakp$ be a maximal ideal in $\calO_K$ with residue field of cardinality $q$. Then there exists a bound $B = B(q)$ such that for all $c \in K$ with $\ord_\frakp(c) \ge 0$, every $K$-rational periodic point for $x^2 + c$ has period less than $B$.
\end{prop}

\begin{proof}[Proof of Proposition~\ref{prop:realizing-portraits}]
Let $\calP$ be a portrait for which $X_1(\calP)$ has infinitely many cubic points. For all such $\calP$, there is a degree-$3$ map $\varphi : X_1(\calP) \to \bbP^1$ defined over $\bbQ$, so we may apply Lemma~\ref{lem:HIT} with $X = X_1(\calP)$, $d = 3$, and the maps $\varphi$ and $\psi = c$. Let $S$ be the finite set of (rational) primes guaranteed by Lemma~\ref{lem:HIT}, and let $p$ be any rational prime not in $S$.

Set $B = \max\{B(p), B(p^2), B(p^3)\}$, with $B(\cdot)$ as in Proposition~\ref{prop:good-red}.
Define $\Sigma(\calP, B)$ to be the set of generic quadratic portraits $\calP'$ such that
    \begin{itemize}
        \item $\calP'$ properly contains $\calP$,
        \item there are no intermediate subportraits $\calP \subsetneq \calP'' \subsetneq \calP'$, and
        \item $\calP'$ has no cycles of length greater than $B$.
    \end{itemize}
Note that $\Sigma(\calP, B)$ is finite, as each $\calP' \in \Sigma(\calP, B)$ can be obtained from $\calP$ by either adding a new cycle of length at most $B$ or by adding two preimages to a vertex in $\calP$. Let $\calP_1,\ldots,\calP_n$ be an enumeration of the set $\Sigma(\calP, B)$.

By Proposition~\ref{prop:functoriality}, for each $i = 1,\ldots,n$, there is a morphism $\pi_i : X_1(\calP_i) \to X_1(\calP)$ of degree at least $2$. Thus, by Lemmas~\ref{lem:finitely-many-PCF-1/4} and~\ref{lem:HIT}, there are infinitely many cubic points $P$ on $X_1(\calP)$ such that
    \begin{itemize}
        \item the preperiodic portrait $G(f_{c(P)}, \bbQ(P))$ is generic quadratic,
        \item for all $i = 1,\ldots,n$, the points in the fiber $\pi_i^{-1}(P)$ are defined over proper extensions of $\bbQ(P)$, and
        \item $\ord_\frakq(c(P)) \ge 0$ for all primes $\frakq$ in $\calO_{\bbQ(P)}$ lying over $p$.
    \end{itemize}
    
To complete the proof, it now suffices to show that if $P$ is one of these infinitely many cubic points, then $G(f_{c(P)}, \bbQ(P))$ is isomorphic to $\calP$. Thus, let $P$ be such a point, and let $K = \bbQ(P)$. Since $P$ is a $K$-rational point on $X_1(\calP)$, and since $G(f_{c(P)},K)$ is generic quadratic, the portrait $G(f_{c(P)}, K)$ has a subgraph isomorphic to $\calP$; we claim that $G(f_{c(P)}, K)$ can be no larger than $\calP$.

Let $\frakq$ be any prime ideal of $\calO_K$ lying over $p$. Since the residue field of $\frakq$ has cardinality $p^k$ for some $1 \le k \le 3$, and since $\ord_\frakq(c(P)) \ge 0$, it follows from Proposition~\ref{prop:good-red} that $f_{c(P)}$ has no $K$-rational periodic points of period greater than $B$. Thus, since we have assumed that $G(f_{c(P)}, K)$ is a generic quadratic portrait, it must be that either $G(f_{c(P)}, K) \cong \calP$ or $G(f_{c(P)}, K)$ contains a portrait isomorphic to $\calP_i$ for some $1 \le i \le n$. But the latter is impossible, as this would imply that $\pi_i^{-1}(P)$ contained a $K$-rational point, which contradicts our assumption that $\pi_i^{-1}(P)$ is irreducible over $K$. Therefore, $G(f_{c(P)}, K) \cong \calP$.
\end{proof}

We pose the following more general question:

\begin{ques}\label{ques:degree-d-points}
Are the following statements equivalent for all integers $d \ge 1$ and all generic quadratic portraits $\calP$?
    \begin{enumerate}
        \item There are infinitely many degree-$d$ points on $X_1(\calP)$.
        \item There are infinitely many $c \in \bbQ^{(d)}$ for which there is a degree-$d$ number field $K$ with $G(f_c,K) \cong \calP$.
    \end{enumerate}
\end{ques}

The answer to Question~\ref{ques:degree-d-points} is ``yes" for $d$ equal to $1$, $2$, and $3$ by \cite[Theorem 4.1]{faber:2015}, \cite[Theorem 1.5]{doyle/krumm:2024}, and Proposition~\ref{prop:realizing-portraits}, respectively. Note that the implication (b) $\implies$ (a) is straightforward, so it remains to show that (a) $\implies$ (b) (as in Proposition~\ref{prop:realizing-portraits}).
We also note that the proof of Proposition~\ref{prop:realizing-portraits} can be used to show that the answer to Question~\ref{ques:degree-d-points} is ``yes" when the curve $X_1(\calP)$ admits a degree-$d$ morphism to $\bbP^1$. 

\section{Rationality of the $c$-coordinate for cubic points}\label{sec:FOD}
Let us say that a cubic point $P \in X_1(\calP)$ is {\bf strictly cubic} if $c(P) \notin \bbQ$ and {\bf weakly cubic} otherwise. We aim to refine Theorems~\ref{thm:main_dynamical} and~\ref{thm:main_curves} by considering whether a curve $X_1(\calP)$ can have infinitely many strictly (resp., weakly) cubic points.

For the portraits that have no cycles of length larger than $2$, the map $c : X_1(\calP) \to \bbP^1$ is a composition of degree-$2$ maps. Thus, if $c \in \bbQ$ and $K$ is a number field for which $\PrePer(f_c,K)$ contains $\calP$, then the field of definition of the $K$-rational preperiodic points has degree a power of $2$. In particular, if $K$ is a cubic field and $c \in K$ is such that $G(f_c,K)$ contains $\calP$, then either $c \notin \bbQ$ or $c \in \bbQ$ and $G(f_c,\bbQ)$ also contains $\calP$, hence $c$ corresponds to a \textit{rational} point on $X_1(\calP)$. In other words, $X_1(\calP)$ has no weakly cubic points and therefore infinitely many strictly cubic points. The portraits in this category are
\[\rm
    \emptyset,\  4(1,1),\  4(2),\  6(1,1),\  6(2),\  8(1,1)a/b,\  8(2)a/b,\  8(2,1,1),\  10(1,1)a,\  10(2,1,1)a/b.
\]
We handle the remaining portraits from Theorem~\ref{thm:main_curves} individually.\\

\noindent {\bf 6(3).} The curve $X_1({\rm 6(3)})$ is isomorphic to $\bbP^1$, hence has infinitely many strictly cubic points by Hilbert irreducibility. Moreover, Morton \cite[Theorem 2(b)]{morton:1992} showed that there are infinitely many $c \in \bbQ$ for which there is a cubic field $K$ containing period-$3$ points for $f_c$ \emph{not} contained in $\bbQ$, so $X_1({\rm 6(3)})$ also has infinitely many weakly cubic points.\\

\noindent {\bf 8(3).} 
We claim that there are only finitely many weakly cubic points on $X_1({\rm 8(3)})$.
Indeed, suppose $c_0 \in \bbQ$ is the image of a cubic point under the map $c : X_1({\rm 8(3)}) \to \bbP^1$.
Let $K$ be the field of definition of any cubic preimage of $c_0$ under $c$. Ignoring at most finitely many such choices of $c_0$, we may make the following assumptions about the map $f_{c_0}(x) = x^2 + c_0$:
    \begin{itemize}
    \item The preperiodic portrait $G(f_{c_0}, K)$ contains a subgraph isomorphic to 8(3). This is equivalent to saying that the $K$-rational preimage of $c_0$ under $c$ lies in the dense open subset $U_1({\rm 8(3)}) \subset X_1({\rm 8(3)})$. Let $\pm \alpha_1,\pm\alpha_2,\pm\alpha_3,\pm\beta$ be $K$-rational preperiodic points arranged as in Figure~\ref{fig:8(3)argument}.
    \item The points $\pm \beta$ are in $K \smallsetminus\bbQ$. Indeed, if $\pm \beta \in \bbQ$, then each of the $\alpha_i$ would be in $\bbQ$, since for each $i$ we can write $\alpha_i = f_{c_0}^k(\beta)$ for some $k\ge1$. But this would yield a rational point on $X_1({\rm 8(3)})$, of which there are finitely many. 
    \item The map $f_{c_0}$ has only three $K$-rational points of period $3$, namely $\alpha_1$, $\alpha_2$, and $\alpha_3$. Indeed, if there were additional points of period $3$, then $c_0$ would actually correspond to a $K$-rational point on $X_1({\rm 14(3,3)})$, which has only finitely many cubic points; see Section~\ref{sec:(3,3)}.
    \end{itemize}

\begin{figure}
    \centering
	\begin{tikzpicture}[scale=1.3]
	\tikzset{vertex/.style = {}}
	\tikzset{every loop/.style={min distance=10mm,in=45,out=-45,->}}
	\tikzset{edge/.style={decoration={markings,mark=at position 1 with %
	    {\arrow[scale=1.2,>=stealth]{>}}},postaction={decorate}}}
	    
	\node[inner sep=.4mm] (03a) at (0,0) {$\alpha_1$};
	\node[inner sep=.4mm] (03b) at (-.866,.5) {$\alpha_2$};
	\node[inner sep=.4mm] (03c) at (-.866,-.5) {$\alpha_3$};
	\node[inner sep=.4mm] (13a) at (1,0) {$-\alpha_3$};
	\node[inner sep=.4mm] (13b) at (-1.366,1.366) {$-\alpha_1$};
	\node[inner sep=.4mm] (13c) at (-1.366,-1.366) {$-\alpha_2$};
	\node[inner sep=.4mm] (23a) at (1.866,.5) {$\beta$};
	\node[inner sep=.4mm] (23b) at (1.866,-.5) {$-\beta$};
	
	\draw[->] (03a) to [bend right=30] (03b);
	\draw[->] (03b) to [bend right=30] (03c);
	\draw[->] (03c) to [bend right=30] (03a);
	\draw[->] (13a) to (03a);
	\draw[->] (13b) to (03b);
	\draw[->] (13c) to (03c);
	\draw[->] (23a) to (13a);
	\draw[->] (23b) to (13a);
	\end{tikzpicture}
    \caption{Preperiodic points arranged according to portrait 8(3)}
    \label{fig:8(3)argument}
\end{figure}
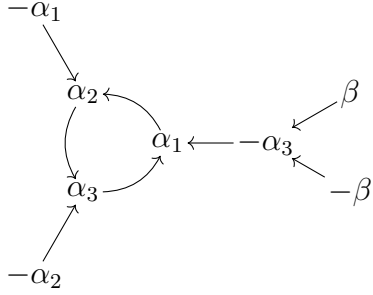

Now, note that none of the $\alpha_i$ may be rational: Indeed, if even one of the $\alpha_i$ lies in $\bbQ$, then all three must be rational (since $\alpha_i = \alpha_{i-1}^2 + c_0$, with indices taken modulo $3$), but then $\pm \beta$ are the roots of $x^2 + c_0 = -\alpha_3$, hence would be defined over a \textit{quadratic} extension of $\bbQ$.

Since $\alpha_1$, $\alpha_2$, and $\alpha_3$ are the only $K$-rational points of period $3$ for $f_{c_0}$, we claim that $K/\bbQ$ is Galois and that the $\alpha_i$ must be $\Gal(K/\bbQ)$-conjugates: Indeed, the third dynatomic polynomial $\Phi_3(c_0,x)$ has degree $6$, and each $\alpha_i$ is a root of an irreducible cubic factor. But we assumed $f_{c_0}$ has only three $K$-rational points of period $3$, so they must be roots of the same irreducible cubic.

Finally, since $\alpha_3$ and $\alpha_1$ are $\Gal(K/\bbQ)$-conjugates, and since the square roots (namely $\pm \beta$) of $-c_0 - \alpha_3$ lie in $K$, it must be that the square roots of $-c_0 - \alpha_2$ must also lie in $K$. But this would then yield a $K$-rational point on the curve $X_1({\rm 10(3)b})$, which has only finitely many cubic points.\\

\noindent {\bf 8(4).} 
We claim that there are only finitely many weakly cubic points on $X_1({\rm 8(4)})$. One could give an argument similar to the one used for 8(3), though it is a bit more involved; instead, we provide a different (and shorter) proof in this case that relies more on the geometry of $X_1({\rm 8(4)})$ than on dynamical properties.

Suppose for the sake of contradiction that $X_1({\rm 8(4)})$ has infinitely many weakly cubic points. The curve $X_1({\rm 8(4)})$ is isomorphic to the \textit{classical} modular curve $X_1(16)$ (see \cite[p. 93]{morton:1998}), so its Jacobian has rank $0$ over $\bbQ$. Thus, there are finitely many degree-$3$ morphisms $X_1({\rm 8(3)}) \to \bbP^1$ (up to automorphisms on $\bbP^1$), and hence by the pigeonhole principle infinitely many weakly cubic points must be linearly equivalent. That is, there must exist some degree-$3$ morphism $\varphi: X_1({\rm 8(3)}) \to \bbP^1$ such that for infinitely many $y \in \bbP^1(\mathbb{Q})$, the fiber $\varphi^{-1}(y)$ is (the Galois orbit of) a weakly cubic point. By definition, if $z \in X_1({\rm 8(3)})$ is a weakly cubic point, then $c(z)$ is rational point, so that $c(\varphi^{-1}(y))$ is a single geometric point. By Lemma~\ref{lem:map must factor}, we see that the $c$ map must factor through $\varphi$; however, computation of the Galois group of (the Galois closure of) the map $c: X_1({\rm 8(3)}) \to \bbP^1)$ shows that the map $c$ does not factor through any degree-$3$ map.\\

\noindent {\bf 10(3,1,1) and 10(3,2).} Let $\calP$ be either of these portraits. Then $X = X_1(\calP)$ has genus $2$, and the Jacobian of $X$ has rank zero over $\bbQ$. The curve $X$ admits a cubic map $\varphi : X \to \bbP^1$ which is disjoint from $c : X \to \bbP^1$, so the cubic preimages of rational points under $\varphi$ will generally satisfy $c \notin \bbQ$ by Hilbert irreducibility. 
In other words, we get infinitely many strictly cubic points.

On the other hand, the map $c : X \to \bbP^1$ also factors through a cubic map defined over $\bbQ$. More precisely, if we let $\sigma\in\Aut(X)$ be the order-$3$ automorphism of $X$ corresponding to a cyclic permutation on the $3$-cycle in $\calP$, then the quotient of $X$ by $\sigma$ is isomorphic to $\bbP^1$, hence the map $c : X \to \bbP^1$ factors through a degree-$3$ map $\varphi : X \to \bbP^1$. Therefore, $X$ has infinitely many weakly cubic points.\\

\noindent {\bf 12(3,3).} We claim that the genus-$4$ curve $X = X_1({\rm 12(3,3)})$ has finitely many strictly cubic points. More precisely, if $K$ is a cubic field, $c \in K$, and $G(f_c,K)$ contains 12(3,3), then with finitely many exceptions, actually $c \in \bbQ$ and $G(f_c,\bbQ)$ contains 6(3).

We first show that the Jacobian $J(X)$ has rank $0$ over $\bbQ$. What follows comes mainly from \cite[\textsection 4]{morton:1992} (the function field of $X$ is the field $N$ from \cite{morton:1992}) plus some Magma computations.

The curve $X$ covers two distinct curves of genus $2$ ($N_-$ and $N_+$ in \cite{morton:1992}); more precisely, if we think of $X$ as given by the model \eqref{eq:X1(12(3,3))}, the genus-$2$ curves are the quotients of $X$ by the automorphisms
    \[
        (c,x,y)\mapsto(c,f_c(x),f_c(y)) \quad\text{and}\quad 
        (c,x,y)\mapsto(c,f_c^2(x),f_c(y)).
    \]
The Jacobians of both curves have rank $0$, so it suffices to show that $J(X)$ is, up to isogeny, a product of the Jacobians of the two genus-$2$ curves. Indeed, the first genus-$2$ quotient has an elliptic subcover, hence has an elliptic isogeny factor, but a Magma computation shows that the second genus-$2$ quotient is simple. Thus, $J(X)$ has rank $0$, so all but finitely many cubic points on $X$ are pullbacks of rational points on $\bbP^1$ by a degree-$3$ map.

We now show that there are only two degree-$3$ maps $X \to \bbP^1$, up to automorphisms of $\bbP^1$. Indeed, there are two natural projection maps
    \[
        X = X_1({\rm 12(3,3)}) \to X_1({\rm 6(3)}) \cong \bbP^1,
    \]
each of which has degree $3$; in particular, $X$ cannot be hyperelliptic (by Theorem~\ref{thm:CS}, for example). A non-hyperelliptic curve admits at most two distinct cubic maps to $\bbP^1$, where by ``distinct" we mean that the pullbacks of $\mathcal{O}_{\mathbb{P}^1}(1)$ are distinct, or equivalently in this case, that one cannot be obtained from the other by postcomposition by an element of $\Aut(\bbP^1)$.
These two projection maps are distinct, hence they are the only cubic maps to $\bbP^1$.

It now follows that all but finitely many cubic points on $X$ are preimages of rational points on $X_1({\rm 6(3)})$; in other words, with finitely many exceptions, if $K$ is a cubic field, $c \in K$, and $G(f_c,K)$ contains 12(3,3), then $c \in \bbQ$ and $G(f_c,\bbQ)$ contains 6(3,3). 

\begin{proof}[Proof of Theorem~\ref{thm:weakly-cubic}]
Let $\calP \in \Gamma(3)$, and suppose that there are infinitely many $c \in \bbQ$ such that $G(f_c,\bbQ) \subsetneq G(f_c,K) \cong \calP$ for some cubic number field $K$. Then there are infinitely many weakly cubic points on $X_1(\calP)$, thus, by the results of this section, $\calP$ must be isomorphic to one of
        \begin{center}
        \rm
        6(3),\quad 10(3,1,1),\quad  10(3,2),\quad and\quad  12(3,3).
        \end{center}
Now let $\calP$ be one of these four portraits. In each case, the map $c : X_1(\calP)\to\bbP^1$ factors as
    \[
    X_1(\calP) \overset{\varphi}{\longto} \bbP^1 \overset{\psi}{\longto} \bbP^1
    \]
with $\deg\varphi = 3$. Choose $x_0 \in \bbP^1(\bbQ)$ with $\psi(x_0) \in \bbQ^\times$ and a prime $p$ of good reduction for $\psi$ such that $\ord_p(\psi(x_0)) \ge 0$, and let $[x_0]_p \subset \bbP^1(\bbQ)$ be the mod-$p$ residue class of $x_0$. Then an argument similar to that of Proposition~\ref{prop:realizing-portraits} (combining Hilbert irreducibility with Proposition~\ref{prop:good-red}) shows that there are infinitely many $t \in \psi([x_0]_p)$ for which there is a cubic field $K$ (specifically, the field of definition of any geometric point in the fiber $c^{-1}(t)$) with $G(f_t,K) \cong \calP$. On the other hand, since the points in the fiber over $t$ are cubic, it follows that $G(f_t,\bbQ) \not\cong \calP$.
\end{proof}

\begin{proof}[Proof of Theorem~\ref{thm:strictly-cubic}]
Let $\calP \in \Gamma(3)$, and suppose that there are infinitely many $c \in \bbQ^{(3)} \smallsetminus \bbQ$ such that $G(f_c,\bbQ(c)) \cong \calP$. Then $X_1(\calP)$ has infinitely many strictly cubic points, and therefore $\calP \not\cong \rm 12(3,3)$.

Now suppose that $\calP \in \Gamma(3)\smallsetminus\{\rm12(3,3)\}$. If $\calP$ is not one of 6(3), 10(3,1,1), or 10(3,2), then, as explained at the beginning of this section, {\it every} cubic point on $X_1(\calP)$ is strictly cubic, so the infinitely many $c\in \bbQ^{(3)}$ guaranteed by Theorem~\ref{thm:main_dynamical} satisfy $c\notin \bbQ$; in other words, there are infinitely many $c \in \bbQ^{(3)}\smallsetminus \bbQ$ such that $G(f_c,\bbQ(c)) \cong \calP$. 

Thus, we let $\calP$ be one of the remaining portraits, namely 6(3), 10(3,1,1), and 10(3,2). Then there is a degree-$3$ morphism $\theta:X_1(\calP)\to\bbP^1$ which is disjoint from the map $c : X_1(\calP)\to\bbP^1$; this was explained above for 10(3,1,1) and 10(3,2), and for 6(3) we simply use the fact that $X_1(\rm6(3)) \cong \bbP^1$. Once again we choose an appropriate prime $p$ and mod-$p$ residue class $[y_0]_p\subset\bbP^1(\bbQ)$ and verify, as in the proof of Proposition~\ref{prop:realizing-portraits} that there are infinitely many cubic $t \in c\big(\theta^{-1}([y_0]_p)\big)$ for which $G(f_t,\bbQ(t)) \cong \calP$.
\end{proof}

\appendix

\section{List of portraits}\label{app:portraits}

\subsection{Graphs appearing in the main classification statements}

We provide here all of the portraits that appear in Theorems~\ref{thm:main_dynamical} and~\ref{thm:main_curves}. For each portrait $\calP$, we also include (when possible) a reference for a model of the curve $X_1(\calP)$.

\newpage
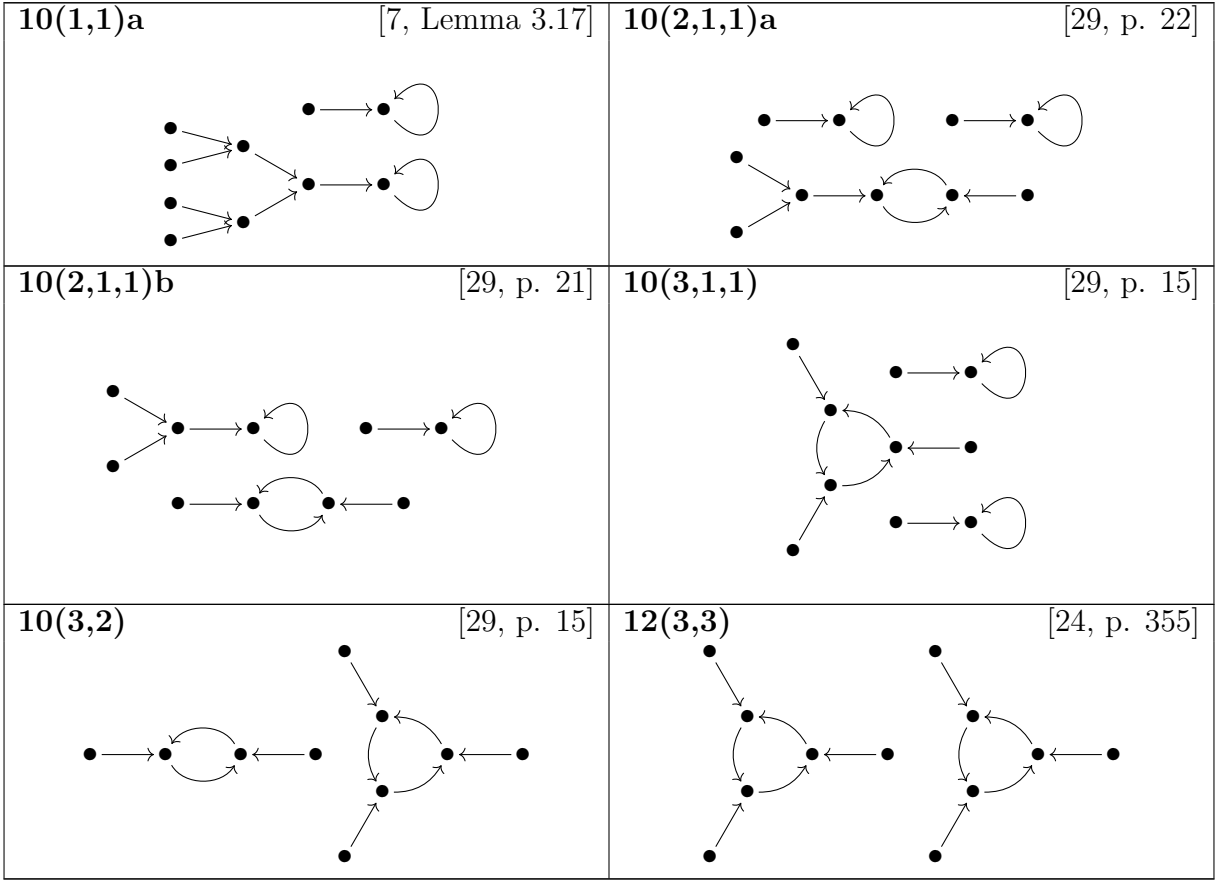
\begin{longtable}{|M{.465\textwidth}|M{.465\textwidth}|}
\hline
	\multicolumn{1}{|l|}{{\bf 4(1,1)} \hfill \cite[Theorem 1]{walde/russo:1994}} &
	\multicolumn{1}{l|}{{\bf 4(2)} \hfill \cite[Theorem 1]{walde/russo:1994}}\\
	\begin{tikzpicture}[scale=1]
	\tikzset{vertex/.style = {}}
	\tikzset{every loop/.style={min distance=10mm,in=45,out=-45,->}}
	\tikzset{edge/.style={decoration={markings,mark=at position 1 with %
	    {\arrow[scale=1.2,>=stealth]{>}}},postaction={decorate}}}
	    
	\node[inner sep=.4mm] (01a) at (1,0) {$\bullet$};
	\node[inner sep=.4mm] (11a) at (0,0) {$\bullet$};
	\node[inner sep=.4mm] (01b) at (3.5,0) {$\bullet$};
	\node[inner sep=.4mm] (11b) at (2.5,0) {$\bullet$};
	
	\draw[->] (01a) to [out=-45, in=45, looseness=10] (01a);
	\draw[->] (01b) to [out=-45, in=45, looseness=10] (01b);
	\draw[->] (11a) to (01a);
	\draw[->] (11b) to (01b);
	\end{tikzpicture}&
	
	\begin{tikzpicture}[scale=1]
	\tikzset{vertex/.style = {}}
	\tikzset{every loop/.style={min distance=10mm,in=45,out=-45,->}}
	\tikzset{edge/.style={decoration={markings,mark=at position 1 with %
	    {\arrow[scale=1.2,>=stealth]{>}}},postaction={decorate}}}
	    
	\node[inner sep=.4mm] (02a) at (1,0) {$\bullet$};
	\node[inner sep=.4mm] (02b) at (2,0) {$\bullet$};
	\node[inner sep=.4mm] (12a) at (0,0) {$\bullet$};
	\node[inner sep=.4mm] (12b) at (3,0) {$\bullet$};
	
	\draw[->] (02a) to [bend right=60] (02b);
	\draw[->] (02b) to [bend right=60] (02a);
	\draw[->] (12a) to (02a);
	\draw[->] (12b) to (02b);
	\end{tikzpicture}\\

	\hline

	\multicolumn{1}{|l|}{{\bf 6(1,1)} \hfill \cite[Theorem 3]{poonen:1998}}&
	\multicolumn{1}{l|}{{\bf 6(2)}  \hfill \cite[Theorem 3]{poonen:1998}}\\
	
	\begin{tikzpicture}[scale=1]
	\tikzset{vertex/.style = {}}
	\tikzset{every loop/.style={min distance=10mm,in=45,out=-45,->}}
	\tikzset{edge/.style={decoration={markings,mark=at position 1 with %
	    {\arrow[scale=1.2,>=stealth]{>}}},postaction={decorate}}}
	    
	\node[inner sep=.4mm] (01a) at (1,0) {$\bullet$};
	\node[inner sep=.4mm] (11a) at (0,0) {$\bullet$};
	\node[inner sep=.4mm] (12a) at (-.866,.5) {$\bullet$};
	\node[inner sep=.4mm] (12b) at (-.866,-.5) {$\bullet$};
	\node[inner sep=.4mm] (01b) at (3.5,0) {$\bullet$};
	\node[inner sep=.4mm] (11b) at (2.5,0) {$\bullet$};
	
	\draw[->] (01a) to [out=-45, in=45, looseness=10] (01a);
	\draw[->] (01b) to [out=-45, in=45, looseness=10] (01b);
	\draw[->] (11a) to (01a);
	\draw[->] (11b) to (01b);
	\draw[->] (12a) to (11a);
	\draw[->] (12b) to (11a);
	\end{tikzpicture}&

	\begin{tikzpicture}[scale=1]
	\tikzset{vertex/.style = {}}
	\tikzset{every loop/.style={min distance=10mm,in=45,out=-45,->}}
	\tikzset{edge/.style={decoration={markings,mark=at position 1 with %
	    {\arrow[scale=1.2,>=stealth]{>}}},postaction={decorate}}}
	    
	\node[inner sep=.4mm] (02a) at (1,0) {$\bullet$};
	\node[inner sep=.4mm] (02b) at (2,0) {$\bullet$};
	\node[inner sep=.4mm] (12a) at (0,0) {$\bullet$};
	\node[inner sep=.4mm] (12b) at (3,0) {$\bullet$};
	\node[inner sep=.4mm] (22a) at (-.866,.5) {$\bullet$};
	\node[inner sep=.4mm] (22b) at (-.866,-.5) {$\bullet$};
	
	\draw[->] (02a) to [bend right=60] (02b);
	\draw[->] (02b) to [bend right=60] (02a);
	\draw[->] (12a) to (02a);
	\draw[->] (12b) to (02b);
	\draw[->] (22a) to (12a);
	\draw[->] (22b) to (12a);
	\end{tikzpicture}\\ \hline

	\multicolumn{1}{|l|}{{\bf 6(3)} \hfill \cite[Theorem 3]{walde/russo:1994}}&
	\multicolumn{1}{l|}{{\bf 8(1,1)a}  \hfill \cite[p. 19]{poonen:1998}}\\
	
	\begin{tikzpicture}[scale=1]
	\tikzset{vertex/.style = {}}
	\tikzset{every loop/.style={min distance=10mm,in=45,out=-45,->}}
	\tikzset{edge/.style={decoration={markings,mark=at position 1 with %
	    {\arrow[scale=1.2,>=stealth]{>}}},postaction={decorate}}}
	    
	\node[inner sep=.4mm] (03a) at (0,0) {$\bullet$};
	\node[inner sep=.4mm] (03b) at (-.866,.5) {$\bullet$};
	\node[inner sep=.4mm] (03c) at (-.866,-.5) {$\bullet$};
	\node[inner sep=.4mm] (13a) at (1,0) {$\bullet$};
	\node[inner sep=.4mm] (13b) at (-1.366,1.366) {$\bullet$};
	\node[inner sep=.4mm] (13c) at (-1.366,-1.366) {$\bullet$};
	
	\draw[->] (03a) to [bend right=30] (03b);
	\draw[->] (03b) to [bend right=30] (03c);
	\draw[->] (03c) to [bend right=30] (03a);
	\draw[->] (13a) to (03a);
	\draw[->] (13b) to (03b);
	\draw[->] (13c) to (03c);
	\end{tikzpicture}&
	
	\begin{tikzpicture}[scale=1]
	\tikzset{vertex/.style = {}}
	\tikzset{every loop/.style={min distance=10mm,in=45,out=-45,->}}
	\tikzset{edge/.style={decoration={markings,mark=at position 1 with %
	    {\arrow[scale=1.2,>=stealth]{>}}},postaction={decorate}}}
	    
	\node[inner sep=.4mm] (01a) at (0,0) {$\bullet$};
	\node[inner sep=.4mm] (11a) at (-1,0) {$\bullet$};
	\node[inner sep=.4mm] (21a) at (-1.866,.5) {$\bullet$};
	\node[inner sep=.4mm] (21b) at (-1.866,-.5) {$\bullet$};
	\node[inner sep=.4mm] (01b) at (3,0) {$\bullet$};
	\node[inner sep=.4mm] (11b) at (2,0) {$\bullet$};
	\node[inner sep=.4mm] (21c) at (1.134,-.5) {$\bullet$};
	\node[inner sep=.4mm] (21d) at (1.134,.5) {$\bullet$};
	
	\draw[->] (01a) to [out=-45, in=45, looseness=10] (01a);
	\draw[->] (11a) to (01a);
	\draw[->] (21a) to (11a);
	\draw[->] (21b) to (11a);
	\draw[->] (01b) to [out=-45, in=45, looseness=10] (01b);
	\draw[->] (11b) to (01b);
	\draw[->] (21c) to (11b);
	\draw[->] (21d) to (11b);
	\end{tikzpicture}\\
	
	\hline
	
	\multicolumn{1}{|l|}{{\bf 8(1,1)b} \hfill \cite[p. 22]{poonen:1998}}&
	\multicolumn{1}{l|}{{\bf 8(2)a}  \hfill \cite[p. 20]{poonen:1998}}\\
	\begin{tikzpicture}[scale=1]
	\tikzset{vertex/.style = {}}
	\tikzset{every loop/.style={min distance=10mm,in=45,out=-45,->}}
	\tikzset{edge/.style={decoration={markings,mark=at position 1 with %
	    {\arrow[scale=1.2,>=stealth]{>}}},postaction={decorate}}}
	    
	\node[inner sep=.4mm] (01a) at (1,0) {$\bullet$};
	\node[inner sep=.4mm] (11a) at (0,0) {$\bullet$};
	\node[inner sep=.4mm] (12a) at (-.866,.5) {$\bullet$};
	\node[inner sep=.4mm] (12b) at (-.866,-.5) {$\bullet$};
	\node[inner sep=.4mm] (22a) at (-1.732,1) {$\bullet$};
	\node[inner sep=.4mm] (22b) at (-1.732,0) {$\bullet$};
	\node[inner sep=.4mm] (01b) at (1,1) {$\bullet$};
	\node[inner sep=.4mm] (11b) at (0,1) {$\bullet$};
	
	\draw[->] (01a) to [out=-45, in=45, looseness=10] (01a);
	\draw[->] (01b) to [out=-45, in=45, looseness=10] (01b);
	\draw[->] (11a) to (01a);
	\draw[->] (11b) to (01b);
	\draw[->] (12a) to (11a);
	\draw[->] (12b) to (11a);
	\draw[->] (22a) to (12a);
	\draw[->] (22b) to (12a);
	\end{tikzpicture}&
	
	\begin{tikzpicture}[scale=1]
	\tikzset{vertex/.style = {}}
	\tikzset{every loop/.style={min distance=10mm,in=45,out=-45,->}}
	\tikzset{edge/.style={decoration={markings,mark=at position 1 with %
	    {\arrow[scale=1.2,>=stealth]{>}}},postaction={decorate}}}
	    
	\node[inner sep=.4mm] (02a) at (1,0) {$\bullet$};
	\node[inner sep=.4mm] (02b) at (2,0) {$\bullet$};
	\node[inner sep=.4mm] (12a) at (0,0) {$\bullet$};
	\node[inner sep=.4mm] (12b) at (3,0) {$\bullet$};
	\node[inner sep=.4mm] (22a) at (-.866,.5) {$\bullet$};
	\node[inner sep=.4mm] (22b) at (-.866,-.5) {$\bullet$};
	\node[inner sep=.4mm] (22c) at (3.866,.5) {$\bullet$};
	\node[inner sep=.4mm] (22d) at (3.866,-.5) {$\bullet$};
	
	\draw[->] (02a) to [bend right=60] (02b);
	\draw[->] (02b) to [bend right=60] (02a);
	\draw[->] (12a) to (02a);
	\draw[->] (12b) to (02b);
	\draw[->] (22a) to (12a);
	\draw[->] (22b) to (12a);
	\draw[->] (22c) to (12b);
	\draw[->] (22d) to (12b);
	\end{tikzpicture}\\
	
	\hline
	
	\multicolumn{1}{|l|}{{\bf 8(2)b} \hfill \cite[p. 23]{poonen:1998}}&
	\multicolumn{1}{l|}{{\bf 8(3)} \hfill \cite[p. 24]{poonen:1998}}\\
	\begin{tikzpicture}[scale=1]
	\tikzset{vertex/.style = {}}
	\tikzset{every loop/.style={min distance=10mm,in=45,out=-45,->}}
	\tikzset{edge/.style={decoration={markings,mark=at position 1 with %
	    {\arrow[scale=1.2,>=stealth]{>}}},postaction={decorate}}}
	    
	\node[inner sep=.4mm] (02a) at (1,0) {$\bullet$};
	\node[inner sep=.4mm] (02b) at (2,0) {$\bullet$};
	\node[inner sep=.4mm] (12a) at (0,0) {$\bullet$};
	\node[inner sep=.4mm] (12b) at (3,0) {$\bullet$};
	\node[inner sep=.4mm] (22a) at (-.866,.5) {$\bullet$};
	\node[inner sep=.4mm] (22b) at (-.866,-.5) {$\bullet$};
	\node[inner sep=.4mm] (32a) at (-1.732,1) {$\bullet$};
	\node[inner sep=.4mm] (32b) at (-1.732,0) {$\bullet$};
	
	\draw[->] (02a) to [bend right=60] (02b);
	\draw[->] (02b) to [bend right=60] (02a);
	\draw[->] (12a) to (02a);
	\draw[->] (12b) to (02b);
	\draw[->] (22a) to (12a);
	\draw[->] (22b) to (12a);
	\draw[->] (32a) to (22a);
	\draw[->] (32b) to (22a);
	\end{tikzpicture}&
	
	\begin{tikzpicture}[scale=1]
	\tikzset{vertex/.style = {}}
	\tikzset{every loop/.style={min distance=10mm,in=45,out=-45,->}}
	\tikzset{edge/.style={decoration={markings,mark=at position 1 with %
	    {\arrow[scale=1.2,>=stealth]{>}}},postaction={decorate}}}
	    
	\node[inner sep=.4mm] (03a) at (0,0) {$\bullet$};
	\node[inner sep=.4mm] (03b) at (-.866,.5) {$\bullet$};
	\node[inner sep=.4mm] (03c) at (-.866,-.5) {$\bullet$};
	\node[inner sep=.4mm] (13a) at (1,0) {$\bullet$};
	\node[inner sep=.4mm] (13b) at (-1.366,1.366) {$\bullet$};
	\node[inner sep=.4mm] (13c) at (-1.366,-1.366) {$\bullet$};
	\node[inner sep=.4mm] (23a) at (1.866,.5) {$\bullet$};
	\node[inner sep=.4mm] (23b) at (1.866,-.5) {$\bullet$};
	
	\draw[->] (03a) to [bend right=30] (03b);
	\draw[->] (03b) to [bend right=30] (03c);
	\draw[->] (03c) to [bend right=30] (03a);
	\draw[->] (13a) to (03a);
	\draw[->] (13b) to (03b);
	\draw[->] (13c) to (03c);
	\draw[->] (23a) to (13a);
	\draw[->] (23b) to (13a);
	\end{tikzpicture}\\
	
	\hline
	
	\multicolumn{1}{|l|}{{\bf 8(4)} \hfill \cite[p. 93]{morton:1998}}&
	\multicolumn{1}{l|}{{\bf 8(2,1,1)} \hfill \cite[Theorem 2]{poonen:1998}}\\
	
	\begin{tikzpicture}[scale=1.1]
	\tikzset{vertex/.style = {}}
	\tikzset{every loop/.style={min distance=10mm,in=45,out=-45,->}}
	\tikzset{edge/.style={decoration={markings,mark=at position 1 with %
	    {\arrow[scale=1.2,>=stealth]{>}}},postaction={decorate}}}
	    
	\node[inner sep=.4mm] (04a) at (0,0) {$\bullet$};
	\node[inner sep=.4mm] (04b) at (1,0) {$\bullet$};
	\node[inner sep=.4mm] (04c) at (1,1) {$\bullet$};
	\node[inner sep=.4mm] (04d) at (0,1) {$\bullet$};
	\node[inner sep=.4mm] (14a) at (-.707,-.707) {$\bullet$};
	\node[inner sep=.4mm] (14b) at (1.707,-.707) {$\bullet$};
	\node[inner sep=.4mm] (14c) at (1.707,1.707) {$\bullet$};
	\node[inner sep=.4mm] (14d) at (-.707,1.707) {$\bullet$};
	
	\draw[->] (04a) to [bend right=30] (04b);
	\draw[->] (04b) to [bend right=30] (04c);
	\draw[->] (04c) to [bend right=30] (04d);
	\draw[->] (04d) to [bend right=30] (04a);
	\draw[->] (14a) to (04a);
	\draw[->] (14b) to (04b);
	\draw[->] (14c) to (04c);
	\draw[->] (14d) to (04d);
	\end{tikzpicture}&
	
	\begin{tikzpicture}[scale=1]
	\tikzset{vertex/.style = {}}
	\tikzset{every loop/.style={min distance=10mm,in=45,out=-45,->}}
	\tikzset{edge/.style={decoration={markings,mark=at position 1 with %
	    {\arrow[scale=1.2,>=stealth]{>}}},postaction={decorate}}}
	    
	\node[inner sep=.4mm] (01a) at (.5,1) {$\bullet$};
	\node[inner sep=.4mm] (11a) at (-.5,1) {$\bullet$};
	\node[inner sep=.4mm] (01b) at (3,1) {$\bullet$};
	\node[inner sep=.4mm] (11b) at (2,1) {$\bullet$};
	\node[inner sep=.4mm] (02a) at (1,0) {$\bullet$};
	\node[inner sep=.4mm] (02b) at (2,0) {$\bullet$};
	\node[inner sep=.4mm] (12a) at (0,0) {$\bullet$};
	\node[inner sep=.4mm] (12b) at (3,0) {$\bullet$};
	
	\draw[->] (01a) to [out=-45, in=45, looseness=10] (01a);
	\draw[->] (01b) to [out=-45, in=45, looseness=10] (01b);
	\draw[->] (11a) to (01a);
	\draw[->] (11b) to (01b);
	\draw[->] (02a) to [bend right=60] (02b);
	\draw[->] (02b) to [bend right=60] (02a);
	\draw[->] (12a) to (02a);
	\draw[->] (12b) to (02b);
	\end{tikzpicture}\\
	
	\hline
	\newpage
    \hline
	\multicolumn{1}{|l|}{{\bf 10(1,1)a} \hfill \cite[Lemma 3.17]{doyle/faber/krumm:2014}}&
	\multicolumn{1}{l|}{{\bf 10(2,1,1)a} \hfill \cite[p. 22]{poonen:1998}}\\
	
	\begin{tikzpicture}[scale=1]
	\tikzset{vertex/.style = {}}
	\tikzset{every loop/.style={min distance=10mm,in=45,out=-45,->}}
	\tikzset{edge/.style={decoration={markings,mark=at position 1 with %
	    {\arrow[scale=1.2,>=stealth]{>}}},postaction={decorate}}}
	    
	\node[inner sep=.4mm] (01a) at (1,0) {$\bullet$};
	\node[inner sep=.4mm] (11a) at (0,0) {$\bullet$};
	\node[inner sep=.4mm] (12a) at (-.866,.5) {$\bullet$};
	\node[inner sep=.4mm] (12b) at (-.866,-.5) {$\bullet$};
	\node[inner sep=.4mm] (22a) at (-1.834,.75) {$\bullet$};
	\node[inner sep=.4mm] (22b) at (-1.834,.25) {$\bullet$};
	\node[inner sep=.4mm] (01b) at (1,1) {$\bullet$};
	\node[inner sep=.4mm] (11b) at (0,1) {$\bullet$};
	\node[inner sep=.4mm] (22c) at (-1.834,-.75) {$\bullet$};
	\node[inner sep=.4mm] (22d) at (-1.834,-.25) {$\bullet$};
	
	\draw[->] (01a) to [out=-45, in=45, looseness=10] (01a);
	\draw[->] (01b) to [out=-45, in=45, looseness=10] (01b);
	\draw[->] (11a) to (01a);
	\draw[->] (11b) to (01b);
	\draw[->] (12a) to (11a);
	\draw[->] (12b) to (11a);
	\draw[->] (22a) to (12a);
	\draw[->] (22b) to (12a);
	\draw[->] (22c) to (12b);
	\draw[->] (22d) to (12b);
	\end{tikzpicture}&
	
	\begin{tikzpicture}[scale=1]
	\tikzset{vertex/.style = {}}
	\tikzset{every loop/.style={min distance=10mm,in=45,out=-45,->}}
	\tikzset{edge/.style={decoration={markings,mark=at position 1 with %
	    {\arrow[scale=1.2,>=stealth]{>}}},postaction={decorate}}}
	    
	\node[inner sep=.4mm] (01a) at (.5,1) {$\bullet$};
	\node[inner sep=.4mm] (11a) at (-.5,1) {$\bullet$};
	\node[inner sep=.4mm] (01b) at (3,1) {$\bullet$};
	\node[inner sep=.4mm] (11b) at (2,1) {$\bullet$};
	\node[inner sep=.4mm] (02a) at (1,0) {$\bullet$};
	\node[inner sep=.4mm] (02b) at (2,0) {$\bullet$};
	\node[inner sep=.4mm] (12a) at (0,0) {$\bullet$};
	\node[inner sep=.4mm] (12b) at (3,0) {$\bullet$};
	\node[inner sep=.4mm] (22a) at (-.866,.5) {$\bullet$};
	\node[inner sep=.4mm] (22b) at (-.866,-.5) {$\bullet$};
	
	\draw[->] (01a) to [out=-45, in=45, looseness=10] (01a);
	\draw[->] (01b) to [out=-45, in=45, looseness=10] (01b);
	\draw[->] (11a) to (01a);
	\draw[->] (11b) to (01b);
	\draw[->] (02a) to [bend right=60] (02b);
	\draw[->] (02b) to [bend right=60] (02a);
	\draw[->] (12a) to (02a);
	\draw[->] (12b) to (02b);
	\draw[->] (22a) to (12a);
	\draw[->] (22b) to (12a);
	\end{tikzpicture}\\
	
	\hline
	
	\multicolumn{1}{|l|}{{\bf 10(2,1,1)b} \hfill \cite[p. 21]{poonen:1998}}&
	\multicolumn{1}{l|}{{\bf 10(3,1,1)} \hfill \cite[p. 15]{poonen:1998}}\\
	
	\begin{tikzpicture}[scale=1]
	\tikzset{vertex/.style = {}}
	\tikzset{every loop/.style={min distance=10mm,in=45,out=-45,->}}
	\tikzset{edge/.style={decoration={markings,mark=at position 1 with %
	    {\arrow[scale=1.2,>=stealth]{>}}},postaction={decorate}}}
	    
	\node[inner sep=.4mm] (01a) at (1,1) {$\bullet$};
	\node[inner sep=.4mm] (11a) at (0,1) {$\bullet$};
	\node[inner sep=.4mm] (01b) at (3.5,1) {$\bullet$};
	\node[inner sep=.4mm] (11b) at (2.5,1) {$\bullet$};
	\node[inner sep=.4mm] (02a) at (1,0) {$\bullet$};
	\node[inner sep=.4mm] (02b) at (2,0) {$\bullet$};
	\node[inner sep=.4mm] (12a) at (0,0) {$\bullet$};
	\node[inner sep=.4mm] (12b) at (3,0) {$\bullet$};
	\node[inner sep=.4mm] (21a) at (-.866,1.5) {$\bullet$};
	\node[inner sep=.4mm] (21b) at (-.866,.5) {$\bullet$};
	
	\draw[->] (01a) to [out=-45, in=45, looseness=10] (01a);
	\draw[->] (01b) to [out=-45, in=45, looseness=10] (01b);
	\draw[->] (11a) to (01a);
	\draw[->] (11b) to (01b);
	\draw[->] (02a) to [bend right=60] (02b);
	\draw[->] (02b) to [bend right=60] (02a);
	\draw[->] (12a) to (02a);
	\draw[->] (12b) to (02b);
	\draw[->] (21a) to (11a);
	\draw[->] (21b) to (11a);
	\end{tikzpicture}&
	
	\begin{tikzpicture}[scale=1]
	\tikzset{vertex/.style = {}}
	\tikzset{every loop/.style={min distance=10mm,in=45,out=-45,->}}
	\tikzset{edge/.style={decoration={markings,mark=at position 1 with %
	    {\arrow[scale=1.2,>=stealth]{>}}},postaction={decorate}}}
	    
	\node[inner sep=.4mm] (03a) at (0,0) {$\bullet$};
	\node[inner sep=.4mm] (03c) at (-.866,-.5) {$\bullet$};
	\node[inner sep=.4mm] (03b) at (-.866,.5) {$\bullet$};
	\node[inner sep=.4mm] (13a) at (1,0) {$\bullet$};
	\node[inner sep=.4mm] (13c) at (-1.366,-1.366) {$\bullet$};
	\node[inner sep=.4mm] (13b) at (-1.366,1.366) {$\bullet$};
	\node[inner sep=.4mm] (01a) at (1,1) {$\bullet$};
	\node[inner sep=.4mm] (11a) at (0,1) {$\bullet$};
	\node[inner sep=.4mm] (01b) at (1,-1) {$\bullet$};
	\node[inner sep=.4mm] (11b) at (0,-1) {$\bullet$};
	
	\draw[->] (03a) to [bend right=30] (03b);
	\draw[->] (03b) to [bend right=30] (03c);
	\draw[->] (03c) to [bend right=30] (03a);
	\draw[->] (13a) to (03a);
	\draw[->] (13b) to (03b);
	\draw[->] (13c) to (03c);
	\draw[->] (01a) to [out=-45, in=45, looseness=10] (01a);
	\draw[->] (01b) to [out=-45, in=45, looseness=10] (01b);
	\draw[->] (11a) to (01a);
	\draw[->] (11b) to (01b);
	\end{tikzpicture}\\
	
	\hline
	
	\multicolumn{1}{|l|}{{\bf 10(3,2)} \hfill \cite[p. 15]{poonen:1998}}&
	\multicolumn{1}{l|}{{\bf 12(3,3)} \hfill \cite[p. 355]{morton:1992}}\\
	
	\begin{tikzpicture}[scale=1]
	\tikzset{vertex/.style = {}}
	\tikzset{every loop/.style={min distance=10mm,in=45,out=-45,->}}
	\tikzset{edge/.style={decoration={markings,mark=at position 1 with %
	    {\arrow[scale=1.2,>=stealth]{>}}},postaction={decorate}}}
	    
	\node[inner sep=.4mm] (03a) at (0,0) {$\bullet$};
	\node[inner sep=.4mm] (03b) at (-.866,.5) {$\bullet$};
	\node[inner sep=.4mm] (03c) at (-.866,-.5) {$\bullet$};
	\node[inner sep=.4mm] (13a) at (1,0) {$\bullet$};
	\node[inner sep=.4mm] (13b) at (-1.366,1.366) {$\bullet$};
	\node[inner sep=.4mm] (13c) at (-1.366,-1.366) {$\bullet$};
	\node[inner sep=.4mm] (02a) at (-3.75,0) {$\bullet$};
	\node[inner sep=.4mm] (02b) at (-2.75,0) {$\bullet$};
	\node[inner sep=.4mm] (12a) at (-4.75,0) {$\bullet$};
	\node[inner sep=.4mm] (12b) at (-1.75,0) {$\bullet$};
	
	\draw[->] (03a) to [bend right=30] (03b);
	\draw[->] (03b) to [bend right=30] (03c);
	\draw[->] (03c) to [bend right=30] (03a);
	\draw[->] (13a) to (03a);
	\draw[->] (13b) to (03b);
	\draw[->] (13c) to (03c);
	\draw[->] (02a) to [bend right=60] (02b);
	\draw[->] (02b) to [bend right=60] (02a);
	\draw[->] (12a) to (02a);
	\draw[->] (12b) to (02b);
	\end{tikzpicture}&
	
	\begin{tikzpicture}[scale=1]
	\tikzset{vertex/.style = {}}
	\tikzset{every loop/.style={min distance=10mm,in=45,out=-45,->}}
	\tikzset{edge/.style={decoration={markings,mark=at position 1 with %
	    {\arrow[scale=1.2,>=stealth]{>}}},postaction={decorate}}}
	    
	\node[inner sep=.4mm] (03a) at (0,0) {$\bullet$};
	\node[inner sep=.4mm] (03b) at (-.866,.5) {$\bullet$};
	\node[inner sep=.4mm] (03c) at (-.866,-.5) {$\bullet$};
	\node[inner sep=.4mm] (13a) at (1,0) {$\bullet$};
	\node[inner sep=.4mm] (13b) at (-1.366,1.366) {$\bullet$};
	\node[inner sep=.4mm] (13c) at (-1.366,-1.366) {$\bullet$};
	\node[inner sep=.4mm] (03d) at (-3,0) {$\bullet$};
	\node[inner sep=.4mm] (03e) at (-3.866,.5) {$\bullet$};
	\node[inner sep=.4mm] (03f) at (-3.866,-.5) {$\bullet$};
	\node[inner sep=.4mm] (13d) at (-2,0) {$\bullet$};
	\node[inner sep=.4mm] (13e) at (-4.366,1.366) {$\bullet$};
	\node[inner sep=.4mm] (13f) at (-4.366,-1.366) {$\bullet$};
	
	\draw[->] (03a) to [bend right=30] (03b);
	\draw[->] (03b) to [bend right=30] (03c);
	\draw[->] (03c) to [bend right=30] (03a);
	\draw[->] (13a) to (03a);
	\draw[->] (13b) to (03b);
	\draw[->] (13c) to (03c);
	\draw[->] (03d) to [bend right=30] (03e);
	\draw[->] (03e) to [bend right=30] (03f);
	\draw[->] (03f) to [bend right=30] (03d);
	\draw[->] (13d) to (03d);
	\draw[->] (13e) to (03e);
	\draw[->] (13f) to (03f);
	\end{tikzpicture} \\
	
	\hline
\end{longtable}

\subsection{Additional portraits}

We now list various portraits that, despite not appearing in the statements of the main results, are used in the proofs of those results. As above, for each such portrait $\calP$ we include, when possible, a reference for a model for $X_1(\calP)$.

In an attempt to standardize the labels of the portraits, we have adopted the labeling scheme from \cite{doyle/faber/krumm:2014}. This requires us to rename some of the portraits that had previously been introduced in \cite{doyle/faber/krumm:2014, doyle:2018quad}, so before listing the portraits in this section, we provide a dictionary between the two labeling schemes. The graph 10(2) appears in \cite{doyle/faber/krumm:2014}, and the graphs labeled $G_n$ appear in \cite{doyle:2018quad}.

\begin{center}
\begin{tabular}{ll}
Previous label & New label\\
\hline
10(2) & 10(2)a\\
$G_1$ & 10(1,1)c\\
$G_2$ & 10(2)b\\
$G_3$ & 10(2)c\\
$G_4$ & 12(2,1,1)c\\
$G_5$ & 12(2,1,1)d\\
$G_6$ & 12(2,1,1)e\\
$G_7$ & 12(3,1,1)a\\
$G_8$ & 12(3,1,1)b\\
$G_9$ & 12(3,2)a\\
$G_{10}$ & 12(3,2)b
\end{tabular}
\end{center}

\begin{longtable}{|M{.465\textwidth}|M{.465\textwidth}|}
\hline
	\multicolumn{1}{|l|}{{\bf 10(1,1)b} \hfill \cite[Lemma 4.9]{doyle/krumm:2024}}&
	\multicolumn{1}{l|}{{\bf 10(1,1)c} \hfill \cite[Proposition 5.3]{doyle:2018quad}}\\
	\begin{tikzpicture}[scale=1]
	\tikzset{vertex/.style = {}}
	\tikzset{every loop/.style={min distance=10mm,in=45,out=-45,->}}
	\tikzset{edge/.style={decoration={markings,mark=at position 1 with %
	    {\arrow[scale=1.2,>=stealth]{>}}},postaction={decorate}}}
	    
	\node[inner sep=.4mm] (01a) at (1,0) {$\bullet$};
	\node[inner sep=.4mm] (11a) at (0,0) {$\bullet$};
	\node[inner sep=.4mm] (21a) at (-.866,.5) {$\bullet$};
	\node[inner sep=.4mm] (21b) at (-.866,-.5) {$\bullet$};
	\node[inner sep=.4mm] (31a) at (-1.732,1) {$\bullet$};
	\node[inner sep=.4mm] (31b) at (-1.732,0) {$\bullet$};
	\node[inner sep=.4mm] (01b) at (1,1) {$\bullet$};
	\node[inner sep=.4mm] (11b) at (0,1) {$\bullet$};
	\node[inner sep=.4mm] (41a) at (-2.598,1.5) {$\bullet$};
	\node[inner sep=.4mm] (41b) at (-2.598,.5) {$\bullet$};
	
	\draw[->] (01a) to [out=-45, in=45, looseness=10] (01a);
	\draw[->] (01b) to [out=-45, in=45, looseness=10] (01b);
	\draw[->] (11a) to (01a);
	\draw[->] (11b) to (01b);
	\draw[->] (21a) to (11a);
	\draw[->] (21b) to (11a);
	\draw[->] (31a) to (21a);
	\draw[->] (31b) to (21a);
	\draw[->] (41a) to (31a);
	\draw[->] (41b) to (31a);
	\end{tikzpicture}&
	
	\begin{tikzpicture}[scale=1]
	\tikzset{vertex/.style = {}}
	\tikzset{every loop/.style={min distance=10mm,in=45,out=-45,->}}
	\tikzset{edge/.style={decoration={markings,mark=at position 1 with %
	    {\arrow[scale=1.2,>=stealth]{>}}},postaction={decorate}}}
	    
	\node[inner sep=.4mm] (01a) at (0,0) {$\bullet$};
	\node[inner sep=.4mm] (11a) at (-1,0) {$\bullet$};
	\node[inner sep=.4mm] (21a) at (-1.866,-.5) {$\bullet$};
	\node[inner sep=.4mm] (21b) at (-1.866,.5) {$\bullet$};
	\node[inner sep=.4mm] (01b) at (0,-1.5) {$\bullet$};
	\node[inner sep=.4mm] (11b) at (-1,-1.5) {$\bullet$};
	\node[inner sep=.4mm] (21c) at (-1.866,-1) {$\bullet$};
	\node[inner sep=.4mm] (21d) at (-1.866,-2) {$\bullet$};
	\node[inner sep=.4mm] (31a) at (-2.732,-1.5) {$\bullet$};
	\node[inner sep=.4mm] (31b) at (-2.732,-.5) {$\bullet$};
	
	\draw[->] (01a) to [out=-45, in=45, looseness=10] (01a);
	\draw[->] (11a) to (01a);
	\draw[->] (21a) to (11a);
	\draw[->] (21b) to (11a);
	\draw[->] (01b) to [out=-45, in=45, looseness=10] (01b);
	\draw[->] (11b) to (01b);
	\draw[->] (21c) to (11b);
	\draw[->] (21d) to (11b);
	\draw[->] (31a) to (21c);
	\draw[->] (31b) to (21c);
	\end{tikzpicture}\\

	\hline

	\multicolumn{1}{|l|}{{\bf 10(2)a}}&
	\multicolumn{1}{l|}{{\bf 10(2)b} \hfill \cite[Lemma 3.30]{doyle/faber/krumm:2014}}\\
	\begin{tikzpicture}[scale=1]
	\tikzset{vertex/.style = {}}
	\tikzset{every loop/.style={min distance=10mm,in=45,out=-45,->}}
	\tikzset{edge/.style={decoration={markings,mark=at position 1 with %
	    {\arrow[scale=1.2,>=stealth]{>}}},postaction={decorate}}}
	    
	\node[inner sep=.4mm] (02a) at (1,0) {$\bullet$};
	\node[inner sep=.4mm] (02b) at (2,0) {$\bullet$};
	\node[inner sep=.4mm] (12a) at (0,0) {$\bullet$};
	\node[inner sep=.4mm] (12b) at (3,0) {$\bullet$};
	\node[inner sep=.4mm] (22a) at (-.866,.5) {$\bullet$};
	\node[inner sep=.4mm] (22b) at (-.866,-.5) {$\bullet$};
	\node[inner sep=.4mm] (32a) at (-1.732,1) {$\bullet$};
	\node[inner sep=.4mm] (32b) at (-1.732,0) {$\bullet$};
	\node[inner sep=.4mm] (42a) at (-2.598,1.5) {$\bullet$};
	\node[inner sep=.4mm] (42b) at (-2.598,.5) {$\bullet$};
	
	\draw[->] (02a) to [bend right=60] (02b);
	\draw[->] (02b) to [bend right=60] (02a);
	\draw[->] (12a) to (02a);
	\draw[->] (12b) to (02b);
	\draw[->] (22a) to (12a);
	\draw[->] (22b) to (12a);
	\draw[->] (32a) to (22a);
	\draw[->] (32b) to (22a);
	\draw[->] (42a) to (32a);
	\draw[->] (42b) to (32a);
	\end{tikzpicture}&
	
	\begin{tikzpicture}[scale=1]
	\tikzset{vertex/.style = {}}
	\tikzset{every loop/.style={min distance=10mm,in=45,out=-45,->}}
	\tikzset{edge/.style={decoration={markings,mark=at position 1 with %
	    {\arrow[scale=1.2,>=stealth]{>}}},postaction={decorate}}}
	    
	\node[inner sep=.4mm] (02a) at (1,0) {$\bullet$};
	\node[inner sep=.4mm] (02b) at (2,0) {$\bullet$};
	\node[inner sep=.4mm] (12a) at (0,0) {$\bullet$};
	\node[inner sep=.4mm] (12b) at (3,0) {$\bullet$};
	\node[inner sep=.4mm] (22a) at (-.866,.5) {$\bullet$};
	\node[inner sep=.4mm] (22b) at (-.866,-.5) {$\bullet$};
	\node[inner sep=.4mm] (32a) at (-1.732,1) {$\bullet$};
	\node[inner sep=.4mm] (32b) at (-1.732,0) {$\bullet$};
	\node[inner sep=.4mm] (22c) at (3.866,.5) {$\bullet$};
	\node[inner sep=.4mm] (22d) at (3.866,-.5) {$\bullet$};
	
	\draw[->] (02a) to [bend right=60] (02b);
	\draw[->] (02b) to [bend right=60] (02a);
	\draw[->] (12a) to (02a);
	\draw[->] (12b) to (02b);
	\draw[->] (22a) to (12a);
	\draw[->] (22b) to (12a);
	\draw[->] (32a) to (22a);
	\draw[->] (32b) to (22a);
	\draw[->] (22c) to (12b);
	\draw[->] (22d) to (12b);
	\end{tikzpicture}\\

	\hline

	\multicolumn{1}{|l|}{{\bf 10(2)c} \hfill \cite[Lemma 3.49]{doyle/faber/krumm:2014}}&
	\multicolumn{1}{l|}{\bf 10(3)a}\\
	\begin{tikzpicture}[scale=1]
	\tikzset{vertex/.style = {}}
	\tikzset{every loop/.style={min distance=10mm,in=45,out=-45,->}}
	\tikzset{edge/.style={decoration={markings,mark=at position 1 with %
	    {\arrow[scale=1.2,>=stealth]{>}}},postaction={decorate}}}
	    
	\node[inner sep=.4mm] (02a) at (1,0) {$\bullet$};
	\node[inner sep=.4mm] (02b) at (2,0) {$\bullet$};
	\node[inner sep=.4mm] (12a) at (0,0) {$\bullet$};
	\node[inner sep=.4mm] (12b) at (3,0) {$\bullet$};
	\node[inner sep=.4mm] (22a) at (-.866,.5) {$\bullet$};
	\node[inner sep=.4mm] (22b) at (-.866,-.5) {$\bullet$};
	\node[inner sep=.4mm] (32a) at (-1.834,.75) {$\bullet$};
	\node[inner sep=.4mm] (32b) at (-1.834,.25) {$\bullet$};
	\node[inner sep=.4mm] (32c) at (-1.834,-.25) {$\bullet$};
	\node[inner sep=.4mm] (32d) at (-1.834,-.75) {$\bullet$};
	
	\draw[->] (02a) to [bend right=60] (02b);
	\draw[->] (02b) to [bend right=60] (02a);
	\draw[->] (12a) to (02a);
	\draw[->] (12b) to (02b);
	\draw[->] (22a) to (12a);
	\draw[->] (22b) to (12a);
	\draw[->] (32a) to (22a);
	\draw[->] (32b) to (22a);
	\draw[->] (32c) to (22b);
	\draw[->] (32d) to (22b);
	\end{tikzpicture}&
	
	\begin{tikzpicture}[scale=1]
	\tikzset{vertex/.style = {}}
	\tikzset{every loop/.style={min distance=10mm,in=45,out=-45,->}}
	\tikzset{edge/.style={decoration={markings,mark=at position 1 with %
	    {\arrow[scale=1.2,>=stealth]{>}}},postaction={decorate}}}
	    
	\node[inner sep=.4mm] (03a) at (0,0) {$\bullet$};
	\node[inner sep=.4mm] (03b) at (-.866,.5) {$\bullet$};
	\node[inner sep=.4mm] (03c) at (-.866,-.5) {$\bullet$};
	\node[inner sep=.4mm] (13a) at (1,0) {$\bullet$};
	\node[inner sep=.4mm] (13b) at (-1.366,1.366) {$\bullet$};
	\node[inner sep=.4mm] (13c) at (-1.366,-1.366) {$\bullet$};
	\node[inner sep=.4mm] (23a) at (1.866,.5) {$\bullet$};
	\node[inner sep=.4mm] (23b) at (1.866,-.5) {$\bullet$};
	\node[inner sep=.4mm] (33a) at (2.732,1) {$\bullet$};
	\node[inner sep=.4mm] (33b) at (2.732,0) {$\bullet$};
	
	\draw[->] (03a) to [bend right=30] (03b);
	\draw[->] (03b) to [bend right=30] (03c);
	\draw[->] (03c) to [bend right=30] (03a);
	\draw[->] (13a) to (03a);
	\draw[->] (13b) to (03b);
	\draw[->] (13c) to (03c);
	\draw[->] (23a) to (13a);
	\draw[->] (23b) to (13a);
	\draw[->] (33a) to (23a);
	\draw[->] (33b) to (23a);
	\end{tikzpicture}\\
	
	\hline
	
	\multicolumn{1}{|l|}{\bf 10(3)b}&
	\multicolumn{1}{l|}{{\bf 10(4)} \hfill \cite[Lemma 3.41]{doyle/faber/krumm:2014}}\\
	
	\begin{tikzpicture}[scale=1]
	\tikzset{vertex/.style = {}}
	\tikzset{every loop/.style={min distance=10mm,in=45,out=-45,->}}
	\tikzset{edge/.style={decoration={markings,mark=at position 1 with %
	    {\arrow[scale=1.2,>=stealth]{>}}},postaction={decorate}}}
	    
	\node[inner sep=.4mm] (03a) at (0,0) {$\bullet$};
	\node[inner sep=.4mm] (03b) at (-.866,.5) {$\bullet$};
	\node[inner sep=.4mm] (03c) at (-.866,-.5) {$\bullet$};
	\node[inner sep=.4mm] (13a) at (1,0) {$\bullet$};
	\node[inner sep=.4mm] (13b) at (-1.366,1.366) {$\bullet$};
	\node[inner sep=.4mm] (13c) at (-1.366,-1.366) {$\bullet$};
	\node[inner sep=.4mm] (23a) at (-2.232,1.866) {$\bullet$};
	\node[inner sep=.4mm] (23b) at (-1.366,2.366) {$\bullet$};
	\node[inner sep=.4mm] (23c) at (-2.232,-1.866) {$\bullet$};
	\node[inner sep=.4mm] (23d) at (-1.366,-2.366) {$\bullet$};
	
	\draw[->] (03a) to [bend right=30] (03b);
	\draw[->] (03b) to [bend right=30] (03c);
	\draw[->] (03c) to [bend right=30] (03a);
	\draw[->] (13a) to (03a);
	\draw[->] (13b) to (03b);
	\draw[->] (13c) to (03c);
	\draw[->] (23a) to (13b);
	\draw[->] (23b) to (13b);
	\draw[->] (23c) to (13c);
	\draw[->] (23d) to (13c);
	\end{tikzpicture}&
	
	\begin{tikzpicture}[scale=1.1]
	\tikzset{vertex/.style = {}}
	\tikzset{every loop/.style={min distance=10mm,in=45,out=-45,->}}
	\tikzset{edge/.style={decoration={markings,mark=at position 1 with %
	    {\arrow[scale=1.2,>=stealth]{>}}},postaction={decorate}}}
	    
	\node[inner sep=.4mm] (04a) at (0,0) {$\bullet$};
	\node[inner sep=.4mm] (04b) at (1,0) {$\bullet$};
	\node[inner sep=.4mm] (04c) at (1,1) {$\bullet$};
	\node[inner sep=.4mm] (04d) at (0,1) {$\bullet$};
	\node[inner sep=.4mm] (14a) at (-.707,-.707) {$\bullet$};
	\node[inner sep=.4mm] (14b) at (1.707,-.707) {$\bullet$};
	\node[inner sep=.4mm] (14c) at (1.707,1.707) {$\bullet$};
	\node[inner sep=.4mm] (14d) at (-.707,1.707) {$\bullet$};
	\node[inner sep=.4mm] (24c) at (-.966,2.673) {$\bullet$};
	\node[inner sep=.4mm] (24d) at (-1.673,1.966) {$\bullet$};
	
	\draw[->] (04a) to [bend right=30] (04b);
	\draw[->] (04b) to [bend right=30] (04c);
	\draw[->] (04c) to [bend right=30] (04d);
	\draw[->] (04d) to [bend right=30] (04a);
	\draw[->] (14a) to (04a);
	\draw[->] (14b) to (04b);
	\draw[->] (14c) to (04c);
	\draw[->] (14d) to (04d);
	\draw[->] (24c) to (14d);
	\draw[->] (24d) to (14d);
	\end{tikzpicture}\\
	
	\hline
	
	\multicolumn{1}{|l|}{\bf 12(1,1)a}&
	\multicolumn{1}{l|}{\bf 12(1,1)b}\\
	
	\begin{tikzpicture}[scale=1]
	\tikzset{vertex/.style = {}}
	\tikzset{every loop/.style={min distance=10mm,in=45,out=-45,->}}
	\tikzset{edge/.style={decoration={markings,mark=at position 1 with %
	    {\arrow[scale=1.2,>=stealth]{>}}},postaction={decorate}}}
		    
	\node[inner sep=.4mm] (01a) at (1,0) {$\bullet$};
	\node[inner sep=.4mm] (11a) at (0,0) {$\bullet$};
	\node[inner sep=.4mm] (21a) at (-.866,.5) {$\bullet$};
	\node[inner sep=.4mm] (21b) at (-.866,-.5) {$\bullet$};
	\node[inner sep=.4mm] (31a) at (-1.834,.75) {$\bullet$};
	\node[inner sep=.4mm] (31b) at (-1.834,.25) {$\bullet$};
	\node[inner sep=.4mm] (01b) at (1,1) {$\bullet$};
	\node[inner sep=.4mm] (11b) at (0,1) {$\bullet$};
	\node[inner sep=.4mm] (31c) at (-1.834,-.75) {$\bullet$};
	\node[inner sep=.4mm] (31d) at (-1.834,-.25) {$\bullet$};
	\node[inner sep=.4mm] (41a) at (-2.668,1) {$\bullet$};
	\node[inner sep=.4mm] (41b) at (-2.668,.5) {$\bullet$};
		
	\draw[->] (01a) to [out=-45, in=45, looseness=10] (01a);
	\draw[->] (01b) to [out=-45, in=45, looseness=10] (01b);
	\draw[->] (11a) to (01a);
	\draw[->] (11b) to (01b);
	\draw[->] (21a) to (11a);
	\draw[->] (21b) to (11a);
	\draw[->] (31a) to (21a);
	\draw[->] (31b) to (21a);
	\draw[->] (31c) to (21b);
	\draw[->] (31d) to (21b);
	\draw[->] (41a) to (31a);
	\draw[->] (41b) to (31a);
	\end{tikzpicture}&

	\begin{tikzpicture}[scale=1]
	\tikzset{vertex/.style = {}}
	\tikzset{every loop/.style={min distance=10mm,in=45,out=-45,->}}
	\tikzset{edge/.style={decoration={markings,mark=at position 1 with %
	    {\arrow[scale=1.2,>=stealth]{>}}},postaction={decorate}}}
	    
	\node[inner sep=.4mm] (01a) at (1,0) {$\bullet$};
	\node[inner sep=.4mm] (11a) at (0,0) {$\bullet$};
	\node[inner sep=.4mm] (21a) at (-.866,.5) {$\bullet$};
	\node[inner sep=.4mm] (21b) at (-.866,-.5) {$\bullet$};
	\node[inner sep=.4mm] (31a) at (-1.834,.75) {$\bullet$};
	\node[inner sep=.4mm] (31b) at (-1.834,.25) {$\bullet$};
	\node[inner sep=.4mm] (01b) at (1,1.5) {$\bullet$};
	\node[inner sep=.4mm] (11b) at (0,1.5) {$\bullet$};
	\node[inner sep=.4mm] (31c) at (-1.834,-.75) {$\bullet$};
	\node[inner sep=.4mm] (31d) at (-1.834,-.25) {$\bullet$};
	\node[inner sep=.4mm] (21c) at (-.866,2) {$\bullet$};
	\node[inner sep=.4mm] (21d) at (-.866,1) {$\bullet$};
	
	\draw[->] (01a) to [out=-45, in=45, looseness=10] (01a);
	\draw[->] (01b) to [out=-45, in=45, looseness=10] (01b);
	\draw[->] (11a) to (01a);
	\draw[->] (11b) to (01b);
	\draw[->] (21a) to (11a);
	\draw[->] (21b) to (11a);
	\draw[->] (31a) to (21a);
	\draw[->] (31b) to (21a);
	\draw[->] (31c) to (21b);
	\draw[->] (31d) to (21b);
	\draw[->] (21c) to (11b);
	\draw[->] (21d) to (11b);
	\end{tikzpicture}\\
	
	\hline	
    \newpage
    \hline
	\multicolumn{1}{|l|}{{\bf 12(2,1,1)a} \hfill\cite[Lemma 3.34]{doyle/faber/krumm:2014}}&
	\multicolumn{1}{l|}{{\bf 12(2,1,1)b} \hfill\cite[Lemma 3.37]{doyle/faber/krumm:2014}}\\
	
	\begin{tikzpicture}[scale=1]
	\tikzset{vertex/.style = {}}
	\tikzset{every loop/.style={min distance=10mm,in=45,out=-45,->}}
	\tikzset{edge/.style={decoration={markings,mark=at position 1 with %
	    {\arrow[scale=1.2,>=stealth]{>}}},postaction={decorate}}}
	    
	\node[inner sep=.4mm] (01a) at (.5,1) {$\bullet$};
	\node[inner sep=.4mm] (11a) at (-.5,1) {$\bullet$};
	\node[inner sep=.4mm] (01b) at (3,1) {$\bullet$};
	\node[inner sep=.4mm] (11b) at (2,1) {$\bullet$};
	\node[inner sep=.4mm] (02a) at (1,0) {$\bullet$};
	\node[inner sep=.4mm] (02b) at (2,0) {$\bullet$};
	\node[inner sep=.4mm] (12a) at (0,0) {$\bullet$};
	\node[inner sep=.4mm] (12b) at (3,0) {$\bullet$};
	\node[inner sep=.4mm] (22a) at (-.866,.5) {$\bullet$};
	\node[inner sep=.4mm] (22b) at (-.866,-.5) {$\bullet$};
	\node[inner sep=.4mm] (22c) at (3.866,.5) {$\bullet$};
	\node[inner sep=.4mm] (22d) at (3.866,-.5) {$\bullet$};
	
	\draw[->] (01a) to [out=-45, in=45, looseness=10] (01a);
	\draw[->] (01b) to [out=-45, in=45, looseness=10] (01b);
	\draw[->] (11a) to (01a);
	\draw[->] (11b) to (01b);
	\draw[->] (02a) to [bend right=60] (02b);
	\draw[->] (02b) to [bend right=60] (02a);
	\draw[->] (12a) to (02a);
	\draw[->] (12b) to (02b);
	\draw[->] (22a) to (12a);
	\draw[->] (22b) to (12a);
	\draw[->] (22c) to (12b);
	\draw[->] (22d) to (12b);
	\end{tikzpicture}&
	
	\begin{tikzpicture}[scale=1]
	\tikzset{vertex/.style = {}}
	\tikzset{every loop/.style={min distance=10mm,in=45,out=-45,->}}
	\tikzset{edge/.style={decoration={markings,mark=at position 1 with %
	    {\arrow[scale=1.2,>=stealth]{>}}},postaction={decorate}}}
	    
	\node[inner sep=.4mm] (02a) at (1,0) {$\bullet$};
	\node[inner sep=.4mm] (02b) at (2,0) {$\bullet$};
	\node[inner sep=.4mm] (12a) at (0,0) {$\bullet$};
	\node[inner sep=.4mm] (12b) at (3,0) {$\bullet$};
	\node[inner sep=.4mm] (01a) at (1,1) {$\bullet$};
	\node[inner sep=.4mm] (11a) at (0,1) {$\bullet$};
	\node[inner sep=.4mm] (01b) at (3.5,1) {$\bullet$};
	\node[inner sep=.4mm] (11b) at (2.5,1) {$\bullet$};
	\node[inner sep=.4mm] (21a) at (-.866,1.5) {$\bullet$};
	\node[inner sep=.4mm] (21b) at (-.866,.5) {$\bullet$};
	\node[inner sep=.4mm] (31a) at (-1.732,2) {$\bullet$};
	\node[inner sep=.4mm] (31b) at (-1.732,1) {$\bullet$};
	
	\draw[->] (01a) to [out=-45, in=45, looseness=10] (01a);
	\draw[->] (01b) to [out=-45, in=45, looseness=10] (01b);
	\draw[->] (11a) to (01a);
	\draw[->] (11b) to (01b);
	\draw[->] (02a) to [bend right=60] (02b);
	\draw[->] (02b) to [bend right=60] (02a);
	\draw[->] (12a) to (02a);
	\draw[->] (12b) to (02b);
	\draw[->] (21a) to (11a);
	\draw[->] (21b) to (11a);
	\draw[->] (31a) to (21a);
	\draw[->] (31b) to (21a);
	\end{tikzpicture}\\
	
	\hline

	\multicolumn{1}{|l|}{{\bf 12(2,1,1)c} \hfill\cite[Proposition 5.10]{doyle:2018quad}}&
	\multicolumn{1}{l|}{{\bf 12(2,1,1)d} \hfill\cite[Proposition 5.14]{doyle:2018quad}}\\

	\begin{tikzpicture}[scale=1]
	\tikzset{vertex/.style = {}}
	\tikzset{every loop/.style={min distance=10mm,in=45,out=-45,->}}
	\tikzset{edge/.style={decoration={markings,mark=at position 1 with %
	    {\arrow[scale=1.2,>=stealth]{>}}},postaction={decorate}}}
	    
	\node[inner sep=.4mm] (02a) at (1,0) {$\bullet$};
	\node[inner sep=.4mm] (02b) at (2,0) {$\bullet$};
	\node[inner sep=.4mm] (12a) at (0,0) {$\bullet$};
	\node[inner sep=.4mm] (12b) at (3,0) {$\bullet$};
	\node[inner sep=.4mm] (22a) at (-.866,.5) {$\bullet$};
	\node[inner sep=.4mm] (22b) at (-.866,-.5) {$\bullet$};
	\node[inner sep=.4mm] (32a) at (-1.732,1) {$\bullet$};
	\node[inner sep=.4mm] (32b) at (-1.732,0) {$\bullet$};
	\node[inner sep=.4mm] (01a) at (.5,1) {$\bullet$};
	\node[inner sep=.4mm] (11a) at (-.5,1) {$\bullet$};
	\node[inner sep=.4mm] (01b) at (3,1) {$\bullet$};
	\node[inner sep=.4mm] (11b) at (2,1) {$\bullet$};
	
	\draw[->] (02a) to [bend right=60] (02b);
	\draw[->] (02b) to [bend right=60] (02a);
	\draw[->] (12a) to (02a);
	\draw[->] (12b) to (02b);
	\draw[->] (22a) to (12a);
	\draw[->] (22b) to (12a);
	\draw[->] (32a) to (22a);
	\draw[->] (32b) to (22a);
	\draw[->] (01a) to [out=-45, in=45, looseness=10] (01a);
	\draw[->] (01b) to [out=-45, in=45, looseness=10] (01b);
	\draw[->] (11a) to (01a);
	\draw[->] (11b) to (01b);
	\end{tikzpicture}&
	
	\begin{tikzpicture}[scale=1]
	\tikzset{vertex/.style = {}}
	\tikzset{every loop/.style={min distance=10mm,in=45,out=-45,->}}
	\tikzset{edge/.style={decoration={markings,mark=at position 1 with %
	    {\arrow[scale=1.2,>=stealth]{>}}},postaction={decorate}}}
	    
	\node[inner sep=.4mm] (01a) at (0,0) {$\bullet$};
	\node[inner sep=.4mm] (11a) at (-1,0) {$\bullet$};
	\node[inner sep=.4mm] (21a) at (-1.866,.5) {$\bullet$};
	\node[inner sep=.4mm] (21b) at (-1.866,-.5) {$\bullet$};
	\node[inner sep=.4mm] (01b) at (0,1.5) {$\bullet$};
	\node[inner sep=.4mm] (11b) at (-1,1.5) {$\bullet$};
	\node[inner sep=.4mm] (21c) at (-1.866,1) {$\bullet$};
	\node[inner sep=.4mm] (21d) at (-1.866,2) {$\bullet$};
	\node[inner sep=.4mm] (02a) at (2,.75) {$\bullet$};
	\node[inner sep=.4mm] (02b) at (3,.75) {$\bullet$};
	\node[inner sep=.4mm] (12a) at (1,.75) {$\bullet$};
	\node[inner sep=.4mm] (12b) at (4,.75) {$\bullet$};
	
	\draw[->] (01a) to [out=-45, in=45, looseness=10] (01a);
	\draw[->] (11a) to (01a);
	\draw[->] (21a) to (11a);
	\draw[->] (21b) to (11a);
	\draw[->] (01b) to [out=-45, in=45, looseness=10] (01b);
	\draw[->] (11b) to (01b);
	\draw[->] (21c) to (11b);
	\draw[->] (21d) to (11b);
	\draw[->] (02a) to [bend right=60] (02b);
	\draw[->] (02b) to [bend right=60] (02a);
	\draw[->] (12a) to (02a);
	\draw[->] (12b) to (02b);
	\end{tikzpicture}\\
	
	\hline
	
	\multicolumn{1}{|l|}{{\bf 12(2,1,1)e} \hfill\cite[Proposition 5.17]{doyle:2018quad}}&
	\multicolumn{1}{l|}{{\bf 12(3,1,1)a} \hfill\cite[Lemma 3.53]{doyle/faber/krumm:2014}}\\

	\begin{tikzpicture}[scale=1]
	\tikzset{vertex/.style = {}}
	\tikzset{every loop/.style={min distance=10mm,in=45,out=-45,->}}
	\tikzset{edge/.style={decoration={markings,mark=at position 1 with %
	    {\arrow[scale=1.2,>=stealth]{>}}},postaction={decorate}}}
	    
	\node[inner sep=.4mm] (01a) at (0,1) {$\bullet$};
	\node[inner sep=.4mm] (11a) at (-1,1) {$\bullet$};
	\node[inner sep=.4mm] (21a) at (-1.866,1.5) {$\bullet$};
	\node[inner sep=.4mm] (21b) at (-1.866,.5) {$\bullet$};
	\node[inner sep=.4mm] (01b) at (0,-1) {$\bullet$};
	\node[inner sep=.4mm] (11b) at (-1,-1) {$\bullet$};
	\node[inner sep=.4mm] (02a) at (-2,0) {$\bullet$};
	\node[inner sep=.4mm] (02b) at (-3,0) {$\bullet$};
	\node[inner sep=.4mm] (12a) at (-1,0) {$\bullet$};
	\node[inner sep=.4mm] (12b) at (-4,0) {$\bullet$};
	\node[inner sep=.4mm] (22a) at (-4.866,.5) {$\bullet$};
	\node[inner sep=.4mm] (22b) at (-4.866,-.5) {$\bullet$};
	
	\draw[->] (01a) to [out=-45, in=45, looseness=10] (01a);
	\draw[->] (11a) to (01a);
	\draw[->] (21a) to (11a);
	\draw[->] (21b) to (11a);
	\draw[->] (01b) to [out=-45, in=45, looseness=10] (01b);
	\draw[->] (11b) to (01b);
	\draw[->] (02a) to [bend right=60] (02b);
	\draw[->] (02b) to [bend right=60] (02a);
	\draw[->] (12a) to (02a);
	\draw[->] (12b) to (02b);
	\draw[->] (22a) to (12b);
	\draw[->] (22b) to (12b);
	\end{tikzpicture}&
	
	\begin{tikzpicture}[scale=1]
	\tikzset{vertex/.style = {}}
	\tikzset{every loop/.style={min distance=10mm,in=45,out=-45,->}}
	\tikzset{edge/.style={decoration={markings,mark=at position 1 with %
	    {\arrow[scale=1.2,>=stealth]{>}}},postaction={decorate}}}
	    
	\node[inner sep=.4mm] (03a) at (0,0) {$\bullet$};
	\node[inner sep=.4mm] (03b) at (-.866,.5) {$\bullet$};
	\node[inner sep=.4mm] (03c) at (-.866,-.5) {$\bullet$};
	\node[inner sep=.4mm] (13a) at (1,0) {$\bullet$};
	\node[inner sep=.4mm] (13b) at (-1.366,1.366) {$\bullet$};
	\node[inner sep=.4mm] (13c) at (-1.366,-1.366) {$\bullet$};
	\node[inner sep=.4mm] (23a) at (1.866,.5) {$\bullet$};
	\node[inner sep=.4mm] (23b) at (1.866,-.5) {$\bullet$};
	\node[inner sep=.4mm] (01a) at (1,1) {$\bullet$};
	\node[inner sep=.4mm] (11a) at (0,1) {$\bullet$};
	\node[inner sep=.4mm] (01b) at (1,-1) {$\bullet$};
	\node[inner sep=.4mm] (11b) at (0,-1) {$\bullet$};
	
	\draw[->] (03a) to [bend right=30] (03b);
	\draw[->] (03b) to [bend right=30] (03c);
	\draw[->] (03c) to [bend right=30] (03a);
	\draw[->] (13a) to (03a);
	\draw[->] (13b) to (03b);
	\draw[->] (13c) to (03c);
	\draw[->] (23a) to (13a);
	\draw[->] (23b) to (13a);
	\draw[->] (01a) to [out=-45, in=45, looseness=10] (01a);
	\draw[->] (01b) to [out=-45, in=45, looseness=10] (01b);
	\draw[->] (11a) to (01a);
	\draw[->] (11b) to (01b);
	\end{tikzpicture}\\
	
	\hline
	
	\multicolumn{1}{|l|}{{\bf 12(3,1,1)b} \hfill\cite[Proposition 5.21]{doyle:2018quad}}&
	\multicolumn{1}{l|}{{\bf 12(3,2)a} \hfill\cite[Lemma 3.57]{doyle/faber/krumm:2014}}\\

	\begin{tikzpicture}[scale=1]
	\tikzset{vertex/.style = {}}
	\tikzset{every loop/.style={min distance=10mm,in=45,out=-45,->}}
	\tikzset{edge/.style={decoration={markings,mark=at position 1 with %
	    {\arrow[scale=1.2,>=stealth]{>}}},postaction={decorate}}}
	    
	\node[inner sep=.4mm] (03a) at (0,0) {$\bullet$};
	\node[inner sep=.4mm] (03c) at (-.866,-.5) {$\bullet$};
	\node[inner sep=.4mm] (03b) at (-.866,.5) {$\bullet$};
	\node[inner sep=.4mm] (13a) at (1,0) {$\bullet$};
	\node[inner sep=.4mm] (13c) at (-1.366,-1.366) {$\bullet$};
	\node[inner sep=.4mm] (13b) at (-1.366,1.366) {$\bullet$};
	\node[inner sep=.4mm] (01a) at (2,1) {$\bullet$};
	\node[inner sep=.4mm] (11a) at (1,1) {$\bullet$};
	\node[inner sep=.4mm] (01b) at (2,-1) {$\bullet$};
	\node[inner sep=.4mm] (11b) at (1,-1) {$\bullet$};
	\node[inner sep=.4mm] (21a) at (.134,1.5) {$\bullet$};
	\node[inner sep=.4mm] (21b) at (.134,.5) {$\bullet$};
	
	\draw[->] (03a) to [bend right=30] (03b);
	\draw[->] (03b) to [bend right=30] (03c);
	\draw[->] (03c) to [bend right=30] (03a);
	\draw[->] (13a) to (03a);
	\draw[->] (13b) to (03b);
	\draw[->] (13c) to (03c);
	\draw[->] (01a) to [out=-45, in=45, looseness=10] (01a);
	\draw[->] (01b) to [out=-45, in=45, looseness=10] (01b);
	\draw[->] (11a) to (01a);
	\draw[->] (11b) to (01b);
	\draw[->] (21a) to (11a);
	\draw[->] (21b) to (11a);
	\end{tikzpicture}&
	
	\begin{tikzpicture}[scale=1]
	\tikzset{vertex/.style = {}}
	\tikzset{every loop/.style={min distance=10mm,in=45,out=-45,->}}
	\tikzset{edge/.style={decoration={markings,mark=at position 1 with %
	    {\arrow[scale=1.2,>=stealth]{>}}},postaction={decorate}}}
	    
	\node[inner sep=.4mm] (03a) at (0,0) {$\bullet$};
	\node[inner sep=.4mm] (03b) at (-.866,.5) {$\bullet$};
	\node[inner sep=.4mm] (03c) at (-.866,-.5) {$\bullet$};
	\node[inner sep=.4mm] (13a) at (1,0) {$\bullet$};
	\node[inner sep=.4mm] (13b) at (-1.366,1.366) {$\bullet$};
	\node[inner sep=.4mm] (13c) at (-1.366,-1.366) {$\bullet$};
	\node[inner sep=.4mm] (23a) at (1.866,.5) {$\bullet$};
	\node[inner sep=.4mm] (23b) at (1.866,-.5) {$\bullet$};
	\node[inner sep=.4mm] (02a) at (1,1.366) {$\bullet$};
	\node[inner sep=.4mm] (02b) at (2,1.366) {$\bullet$};
	\node[inner sep=.4mm] (12a) at (0,1.366) {$\bullet$};
	\node[inner sep=.4mm] (12b) at (3,1.366) {$\bullet$};
	
	\draw[->] (03a) to [bend right=30] (03b);
	\draw[->] (03b) to [bend right=30] (03c);
	\draw[->] (03c) to [bend right=30] (03a);
	\draw[->] (13a) to (03a);
	\draw[->] (13b) to (03b);
	\draw[->] (13c) to (03c);
	\draw[->] (23a) to (13a);
	\draw[->] (23b) to (13a);
	\draw[->] (02a) to [bend right=60] (02b);
	\draw[->] (02b) to [bend right=60] (02a);
	\draw[->] (12a) to (02a);
	\draw[->] (12b) to (02b);
	\end{tikzpicture}\\
	
	\hline

    \multicolumn{1}{|l|}{{\bf 12(3,2)b} \hfill\cite[Proposition 5.26]{doyle:2018quad}}&
	\multicolumn{1}{l|}{{\bf 14(3,3)}}\\

	\begin{tikzpicture}[scale=1]
	\tikzset{vertex/.style = {}}
	\tikzset{every loop/.style={min distance=10mm,in=45,out=-45,->}}
	\tikzset{edge/.style={decoration={markings,mark=at position 1 with %
	    {\arrow[scale=1.2,>=stealth]{>}}},postaction={decorate}}}
	    
	\node[inner sep=.4mm] (03a) at (0,0) {$\bullet$};
	\node[inner sep=.4mm] (03b) at (-.866,.5) {$\bullet$};
	\node[inner sep=.4mm] (03c) at (-.866,-.5) {$\bullet$};
	\node[inner sep=.4mm] (13a) at (1,0) {$\bullet$};
	\node[inner sep=.4mm] (13b) at (-1.366,1.366) {$\bullet$};
	\node[inner sep=.4mm] (13c) at (-1.366,-1.366) {$\bullet$};
	\node[inner sep=.4mm] (02a) at (-3,0) {$\bullet$};
	\node[inner sep=.4mm] (02b) at (-4,0) {$\bullet$};
	\node[inner sep=.4mm] (12a) at (-2,0) {$\bullet$};
	\node[inner sep=.4mm] (12b) at (-5,0) {$\bullet$};
	\node[inner sep=.4mm] (22a) at (-5.866,.5) {$\bullet$};
	\node[inner sep=.4mm] (22b) at (-5.866,-.5) {$\bullet$};

	\draw[->] (03a) to [bend right=30] (03b);
	\draw[->] (03b) to [bend right=30] (03c);
	\draw[->] (03c) to [bend right=30] (03a);
	\draw[->] (13a) to (03a);
	\draw[->] (13b) to (03b);
	\draw[->] (13c) to (03c);
	\draw[->] (02a) to [bend right=60] (02b);
	\draw[->] (02b) to [bend right=60] (02a);
	\draw[->] (12a) to (02a);
	\draw[->] (12b) to (02b);
	\draw[->] (22a) to (12b);
	\draw[->] (22b) to (12b);
	\end{tikzpicture}&
    
	\begin{tikzpicture}[scale=1]
	\tikzset{vertex/.style = {}}
	\tikzset{every loop/.style={min distance=10mm,in=45,out=-45,->}}
	\tikzset{edge/.style={decoration={markings,mark=at position 1 with %
	    {\arrow[scale=1.2,>=stealth]{>}}},postaction={decorate}}}
	    
	\node[inner sep=.4mm] (03a) at (0,0) {$\bullet$};
	\node[inner sep=.4mm] (03b) at (-.866,.5) {$\bullet$};
	\node[inner sep=.4mm] (03c) at (-.866,-.5) {$\bullet$};
	\node[inner sep=.4mm] (13a) at (1,0) {$\bullet$};
	\node[inner sep=.4mm] (13b) at (-1.366,1.366) {$\bullet$};
	\node[inner sep=.4mm] (13c) at (-1.366,-1.366) {$\bullet$};
	\node[inner sep=.4mm] (03d) at (-3,0) {$\bullet$};
	\node[inner sep=.4mm] (03e) at (-3.866,.5) {$\bullet$};
	\node[inner sep=.4mm] (03f) at (-3.866,-.5) {$\bullet$};
	\node[inner sep=.4mm] (13d) at (-2,0) {$\bullet$};
	\node[inner sep=.4mm] (13e) at (-4.366,1.366) {$\bullet$};
	\node[inner sep=.4mm] (13f) at (-4.366,-1.366) {$\bullet$};
	\node[inner sep=.4mm] (23a) at (1.866,.5) {$\bullet$};
	\node[inner sep=.4mm] (23b) at (1.866,-.5) {$\bullet$};
	
	\draw[->] (03a) to [bend right=30] (03b);
	\draw[->] (03b) to [bend right=30] (03c);
	\draw[->] (03c) to [bend right=30] (03a);
	\draw[->] (13a) to (03a);
	\draw[->] (13b) to (03b);
	\draw[->] (13c) to (03c);
	\draw[->] (03d) to [bend right=30] (03e);
	\draw[->] (03e) to [bend right=30] (03f);
	\draw[->] (03f) to [bend right=30] (03d);
	\draw[->] (13d) to (03d);
	\draw[->] (13e) to (03e);
	\draw[->] (13f) to (03f);
	\draw[->] (23a) to (13a);
	\draw[->] (23b) to (13a);
	\end{tikzpicture}\\
	\hline
\end{longtable}

\newpage

\bibliography{refs.bib}

\bibliographystyle{amsplain}

\end{document}